\documentclass[oneside, reqno]{amsart}
\usepackage{amsmath,amsthm}
\usepackage{amssymb,latexsym, eucal}
\usepackage{amscd} %For commutative diagrams.
\usepackage{verbatim}  %For comments while editing.
\usepackage{epic} %For enhanced pictures.
\newcommand\pref[1]{\textbf{\ref{#1}}}
\newcommand\headtt[1]{{        \textrm{(}\tt #1\textrm{)}      }}
\newlength\savedwidth %\whline for wide rules in tables
        \newcommand\whline{\noalign{\global\savedwidth\arrayrulewidth\global\arrayrulewidth 1.2pt}%
        \hline
        \noalign{\global\arrayrulewidth\savedwidth}}

\theoremstyle{plain}
\newtheorem{theorem}{Theorem}[subsection]
\newtheorem{lemma}[theorem]{Lemma}
\newtheorem{corollary}[theorem]{Corollary}
\newtheorem{proposition}[theorem]{Proposition}

\theoremstyle{definition}
\newtheorem{definition}[theorem]{Definition}

\newtheorem{example}[theorem]{Example}

\newtheorem{remark}[theorem]{Remark}

\newtheorem{pgraph}[theorem]{{}}

\DeclareMathOperator\ad{ad}
\DeclareMathOperator\Ad{Ad}
\DeclareMathOperator\Aff{Aff}
\DeclareMathOperator\Aut{Aut}
\DeclareMathOperator\bfAut{\mathbf{Aut}}
\DeclareMathOperator\Br{Br}
\DeclareMathOperator{\Centre}{Z}
\DeclareMathOperator{\card}{card}
\DeclareMathOperator{\Cd}{C}
\DeclareMathOperator\characteristic{char}

\DeclareMathOperator\diag{diag}
\DeclareMathOperator\disc{disc}
\DeclareMathOperator\End{End}
\DeclareMathOperator\Gal{Gal}
\DeclareMathOperator\gl{\mathfrak{gl}}
\DeclareMathOperator\GL{GL}
\DeclareMathOperator\SPL{SL}

\DeclareMathOperator\Hom{Hom}

\DeclareMathOperator\Ker{Ker}
\DeclareMathOperator\Lp{L} %Loop
\DeclareMathOperator\Mat{M}

\DeclareMathOperator\Out{Out}

\DeclareMathOperator\rad{rad}
\DeclareMathOperator\rot{rot}
\DeclareMathOperator\spl{\mathfrak{sl}}

\DeclareMathOperator\tr{tr}
\DeclareMathOperator\spann{span}

\newcommand \andd{\quad\text{and}\quad}
\newcommand\ck{^{\vee}}
\newcommand \cardnum[1]{\card(#1)}
\newcommand\form{(\,\, \vert\,\, )}
\newcommand\isom{\simeq}
\newcommand\op {\text{op}}
\newcommand \order[1]{\left\vert #1 \right\vert}
\newcommand\ot {\otimes}
\newcommand \set[1]{\{#1 \}}
\newcommand \suchthat { \mid }

\newcommand \bbC {\mathbb C}
\newcommand \bbE {\mathbb E}
\newcommand \bbF {\mathbb F}
\newcommand \bbI {\mathbb I}

\newcommand \bbM {\mathbb M}
\newcommand \bbQ{\mathbb Q}
\newcommand \bbR{\mathbb R}
\newcommand \bbZ {\mathbb Z}

\newcommand \cA {\mathcal A}

\newcommand \cE {\mathcal E}

\newcommand \cH {\mathcal H}
\newcommand \cI {\mathcal I}
\newcommand \cJ {\mathcal J}
\newcommand \cK {\mathcal K}
\newcommand \cL {\mathcal L}

\newcommand \cO {\mathcal O}
\newcommand \cS {\mathcal S}
\newcommand \cT {\mathcal T}

\newcommand \fg {{\mathfrak g}}
\newcommand \fh {{\mathfrak h}}

\newcommand \al {\alpha}

\newcommand \Dl {\Delta}

\newcommand \Lm {\Lambda}
\newcommand \ph{\varphi}
\newcommand \sg {\sigma}

\newcommand\boldkay  {{\boldsymbol k}}
\newcommand\boldm    {{\boldsymbol m}}
\newcommand\boldone  {{\boldsymbol 1}}
\newcommand\boldq    {{\boldsymbol q}}
\newcommand\boldsg   {{\boldsymbol{\sigma}}}
\newcommand\boldsgt  {{\tilde{\boldsymbol{\sigma}}}}
\newcommand\boldtau   {{\boldsymbol{\tau}}}

\newcommand \bari     {{\bar \imath}}
\newcommand \barj     {{\bar \jmath}}
\newcommand \bkay     {{\bar k}}
\newcommand \bboldkay  {{\bar{\boldsymbol{k}}}}
\newcommand\bF       {{\bar F}}

\newcommand \bV      {{\bar V}}
\newcommand \bPhi    {{\bar \Phi}}

\newcommand\brva     {{\breve a}}
\newcommand \brvA    {{\breve A}}
\newcommand \brvDl  {{\breve \Delta}}
\newcommand \brvI   {{\breve {\tt I}}}

\newcommand \brvW   {{\breve W }}

\newcommand \tcL {\widetilde{\cL}}
\newcommand \tCd {\widetilde{C}}
\newcommand \tH {{\widetilde H}}

\newcommand \tsg {{\tilde\sg}}

\newcommand \core {c}
\newcommand \ccore {{cc}}

\newcommand \real{\text{re}}

\newcommand\sears[3]{\textrm{#1}_{#2}^{#3}}
\newcommand \Tindex[4]{{\vphantom{#2}}^{#1}{\textrm{#2}}_{#3}^{#4}}
\newcommand \type[2]{\textrm{#1}_{#2}}
\newcommand \typeaff[3]{{\textrm{#1}}_{#2}^{(#3)}}
\newcommand\typemult[4]{\textrm{#1}_{#2}^{(#3,#4)}}

\newcommand \prim {\theta}
\newcommand \Qp {Q(\prim)}
\newcommand\boldprim  {{\boldsymbol{\prim}}}

\newcommand \Aute{\Aut^\text{e}}
\newcommand \Autz{\Aut^0}
\newcommand \base {{\Bbbk}}
\newcommand\chev{\omega} %Chevalley involution
\newcommand \gd {\dot{\mathfrak g}}
\newcommand \fgp {{\fg'}}
\newcommand \fgb {\overline{\fg'}}
\newcommand{\fnp}{{\mathfrak n_+}}

\newcommand\isoR{\simeq_{R_2}}
\newcommand \kay {k} %Use this for superscripts in multiloop algebras
\newcommand \ttI {{\tt I}} %Index set
\newcommand \Lpbm {\Lp_\boldm}
\newcommand \Lpm {\Lp_m}
\newcommand \pd {{\dot p}}
\newcommand \piDlb {{\overline{\pi_\sigma(\Delta)}}}
\newcommand\rkaff{\ell}  %Use for rank in untwisted affine algebras
\newcommand\rkfin{k}  %Use for rank in twisted affine algebras
\newcommand \sgd {\dot{\sigma}}
\newcommand\slashspace{,\ \ }%Separation in tables
\newcommand \Qt {Q(\theta)}

\newcommand \Zmn {\mathbb Z^n_\boldm}
\newcommand \Zn {{\mathbb Z^n}}

\begin{document}

\title[Nullity 2 multiloop algebras]{Multiloop algebras,
iterated loop algebras and extended affine Lie algebras of nullity 2}
\author{Bruce Allison}
\address[Bruce Allison]{Department of Mathematics and Statistics \\ University
of Victoria \\  PO BOX 3060 STN CSC\\ Victoria BC Canada  V8W 3R4}
\email{ballison@uvic.ca}
\author{Stephen Berman}
\address[Stephen Berman]
{Saskatoon, Saskatchewan, Canada}
\email{sberman@shaw.ca}
\author{Arturo Pianzola}
\address[Arturo Pianzola]
{Department of Mathematical and Statistical Sciences\\
University of Alberta\\Edmonton, Alberta, Canada T6G 2G1}
\address[Arturo Pianzola]
{Instituto Argentino de Matemática\\ Saavedra 15, 1083 Buenos
Aires, Argentina} \email{a.pianzola@ualberta.ca}
\thanks{B. Allison and A. Pianzola
gratefully acknowledge the support of the Natural Sciences and
Engineering Research Council of Canada.  A. Pianzola gratefully
acknowledges the support of CONICET}
\date{Feb 11, 2010}
%17B65: Infinite-dimensional Lie (super) algebras
%17B67: Kac-Moody (super)algebras; extended affine Lie algebras; toroidal Lie algebras

\begin{abstract} Let $\bbM_n$ be the class of all
multiloop algebras  of finite dimensional simple Lie algebras
relative to $n$-tuples of commuting finite order automorphisms.
It is a classical result that $\bbM_1$ is the class of all
derived algebras modulo their centres of affine Kac-Moody Lie algebras.
This combined with the Peterson-Kac conjugacy theorem for affine algebras
results in a classification of the algebras in $\bbM_1$.
In this paper, we classify the algebras in $\bbM_2$,
and further determine the relationship between $\bbM_2$ and two
other classes of Lie algebras:  the class of all loop algebras of affine Lie algebras and
the class of all extended affine Lie algebras of nullity~2.
\end{abstract}

\maketitle

\section {Introduction} \
\label{sec:intro}

Affine Kac-Moody  algebras over an algebraically closed field of
of
charact\-eristic 0 comprise one of the most widely studied and
applied classes of infinite dimensional Lie algebras.  Kac's
realization theorem is of fundamental importance in this area as
it provides an explicit construction of affine algebras from finite
dimensional simple algebras. More precisely, the version of this
theorem that we use states that the algebras of the form $\fgb := \fgp/Z(\fgp)$,
where $\fg$ is an affine algebra, are  the same  up to isomorphism as
the loop algebras of the form $\Lp(\gd,\sg)$ where $\gd$ is a finite dimensional
simple and $\sg$ a diagram automorphism of $\gd$. (See Section \ref{subsec:loopdef}
below   for the definition of a loop algebra.)

Now finite dimensional simple algebras and affine algebras can be regarded as nullity 0 and nullity 1 Lie algebras, in the sense
that the additive group generated by the isotropic roots is free of rank 0 or 1 respectively.
Taking this point of view, it is irresistible to look for  higher nullity analogues of these algebras,
and to hope that loop algebras in some form  will provide constructions of algebras of increasing nullity.

A considerable amount    of interesting work has been done on higher nullity algebras in recent years using a number of different approaches.
These approaches  have included:
axiomatic characterizations;
realizations as multiloop algebras, iterated loop algebras,
or matrix algebras over nonassociative coordinate algebras;
constructions  by means of  generators and relations determined by root systems; constructions
of representations; and the study of higher nullity algebras
as forms of untwisted loop algebras and their relation to torsors over Laurent polynomial  rings.

In this paper, we will focus on three of these approaches (although some of the other approaches will also be
mentioned), namely the construction of higher nullity algebras using multiloop algebras
and iterated loop algebras, and the axiomatic description of higher nullity algebras.
To outline our main results, we now  introduce three classes of algebras $\bbM_n$, $\bbI_n$ and $\bbE_n$,
which depend on a nonnegative integer $n$, and we say that the algebras
 in these classes have \emph{nullity} $n$.   In each case,
we mention a sample of the papers which deal with these topics.  There are many other interesting
related articles, several of which are listed in the bibliography.

\begin{list}{$\bullet$}
{\setlength{\leftmargin}{.3in}\setlength{\rightmargin}{.3in}}
\item \textbf{Multiloop algebras and $\bbM_n$}
\cite{EMY, vdL, GP1, ABFP1, ABFP2, Na}
There is a natural generalization
of the loop construction that produces a \emph{multiloop algebra}
$\Lp(\fg,\boldsg)$ from a Lie algebra $\fg$ using an $n$-tuple $\boldsg = (\sg_1,\dots,\sg_n)$
of commuting finite order automorphisms of $\fg$  (see Section \ref{subsec:multiloop} below). We denote by $\bbM_n$ the class of all Lie
algebras that are isomorphic to multiloop algebras of the form $\Lp(\gd,\boldsg)$, where $\gd$ is finite dimensional and simple
and $\boldsg = (\sg_1,\dots,\sg_n)$ is as indicated.  By convention, $\bbM_0$ is the class
of all finite dimensional simple Lie algebras.

\item\textbf{Iterated loop algebras and $\bbI_n$}
\cite{W, Po, ABP1, ABP2, ABP2.5}
Let $\bbI_0$  be the class
of all finite dimensional simple Lie algebras; and, for $n \ge 1$,
let $\bbI_n$ be the class of all Lie algebras that are isomorphic
to loop algebras of the form $\Lp(\fg,\sg)$, where $\fg\in \bbI_{n-1}$
and $\sg$ is a finite order automorphism of $\fg$.

\item\textbf{EALAs and  $\bbE_n$}
\cite{Sai, H-KT, BGK, AABGP, Neh2, Neh3}
An \emph{extended affine Lie algebra}, or EALA for short, is an algebra $\fg$
that satisfies a list of natural axioms
that are modeled on well known properties of finite dimensional simple Lie algebras and affine algebras
(see Section \ref{subsec:EALAdef} below).
One of the axioms requires that the rank of the group generated by the isotropic
roots of $\fg$  be finite, and this rank is called the \emph{nullity} of $\fg$.
Further, the subalgebra
$\fg_c$ of $\fg$ generated by the root spaces corresponding to nonisotropic roots is called the \emph{core} of $\fg$.  (In general
$\fg_c \subseteq \fg'$, and these algebras are equal in the affine case.)
We denote by $\bbE_n$ the class
of all algebras that are isomorphic to algebras of the form $\fg_c/Z(\fg_c)$ for some EALA $\fg$ of nullity n.
\end{list}

It is not difficult to show that $\bbM_0 = \bbI_0 = \bbE_0,$ and that this is the
class of finite dimensional simple Lie algebras (Proposition \ref{prop:nullity0}).
Of course the classification of these algebras is well known.
Also, we will see using Kac's realization theorem that  $\bbM_1 = \bbI_1 = \bbE_1$  and that
this is the class of derived algebras modulo their centres of affine algebras (Corollary \ref{cor:KacReal1}).
Moreover, the Peterson-Kac conjugacy theorem for affine algebras provides a classification of the algebras in $\bbM_1$ (Corollary \ref{cor:KacReal3}).
In general $\bbM_n \subseteq \bbI_n$ (see \pref{pgraph:MnIn}), but no other containments  hold,  even for $n=2$.

This paper concludes a sequence \cite{ABP1}, \cite{ABP2}, \cite{ABP2.5}, \cite{ABFP1}, \allowbreak \cite{ABFP2}
of papers devoted to the study of the three classes $\bbM_n$,  $\bbI_n$,  $\bbE_n$.  The preceding five articles contain a number
of general results about these classes, and in this paper, we apply these results to the case $n=2$.  Our results provide the following:
\begin{list}{$\bullet$}
{\setlength{\leftmargin}{.3in}\setlength{\rightmargin}{.3in}}
\item A complete understanding of the relationship betweens the classes $\bbM_2$, $\bbI_2$ and $\bbE_2$.
\item A classification, up to isomorphism, of the algebras in $\bbM_2$.
\end{list}

To discuss this in more detail, note first that by Kac's realization theorem,
the algebras in $\bbI_2$ are, up to isomorphism, the algebras
of the form
\begin{equation}
\label{eq:I2form}
\Lp(\fgb,\sg),\ \text{with $\fg$ affine and $\sg$ a finite order automorphism of $\fgb$.}
\end{equation}
We show in Section \ref{sec:char},
that an algebra of this form is in
$\bbM_2$ if and only if $\sg$ is of first kind.
(The notion of first and second kind for finite order automorphisms of $\fgb$  is recalled
in \S \ref{sec:autaff}.)
In fact, our first main theorem, Theorem \ref{thm:char}, shows that
the algebras in $\bbM_2$ are, up to isomorphism, the algebras
of the form
\begin{equation}
\label{eq:M2form}
\Lp(\fgb,\sg),\ \text{with $\fg$ affine  and $\sg$ a diagram automorphism of $\fgb$;}
\end{equation}
or equivalently they are the algebras of the form
\begin{equation}
\label{eq:M2formu}
\Lp(\fgb,\sg),\ \text{with $\fg$ untwisted affine  and $\sg$ a diagram automorphism of $\fgb$.}
\end{equation}
Using this theorem we are able in Section \ref{sec:char} to describe the relationship between the classes $\bbM_2$, $\bbI_2$ and $\bbE_2$ in a sequence of corollaries.
This information is summarized in Figure~\ref{fig:classdiag}.

Our second main theorem, Theorem \ref{thm:class}, which provides a classification
theorem for algebras in $\bbM_2$, is proved by determining when two algebras of the form \eqref{eq:M2formu} are isomorphic.
Two key invariants that are used for this
are the relative and absolute types of an algebra $\cL \in \bbM_2$;
these are defined to be the relative and absolute types respectively
of the central closure of $\cL$, which is finite dimensional and simple
over its centroid.  The absolute type of $\cL$ is easy to compute (see
\pref{pgraph:abseasy}), but the relative type requires  considerably more work.

We now outline the structure of the paper, and briefly discuss some of our other results and methods.
First, after some preliminaries in Section \ref{sec:prelim},  we record in Sections \ref{sec:type}, \ref{sec:loop} and  \ref{sec:EALA}
the basic definitions and results we need about relative and absolute type, multiloop algebras, iterated loop algebras,
and EALAs.

In Section \ref{sec:autKM}, we obtain an erasing theorem, Theorem \ref{thm:erase}, for loop algebras of symmetrizable Kac-Moody Lie algebras.
This theorem  is an extension of a theorem in \cite{ABP2}, which allows us to replace an arbitrary $\sg$
in \eqref{eq:I2form} by an outer automorphism, just as Kac did for loop algebras
of finite dimensional simple algebras in \cite[Prop.~8.5]{K2}.

In Section \ref{sec:autaff}, we discuss automorphisms of $\fgb$, where $\fg$ is affine.
In particular, we  show that the kind of a finite order automorphism $\sg$
is determined by the structure of the centroid of $\Lp(\fgb,\sg)$.

In Sections \ref{sec:looprot} and \ref{sec:loopnt}, we study in detail the structure
of a loop algebra as in \eqref{eq:M2form}. Section \ref{sec:looprot} considers
the \emph{rotation case} when $\fg$ is of type $\typeaff{A}{\ell}{1}$ and $\sg$ is a rotation of the Dynkin diagram;
we see in this case that $\Lp(\fgb,\sg)$ is a special linear algebra  over a quantum torus.
Section \ref{sec:loopnt} considers the \emph{nontransitive case} when $\sg$ is not transitive on the Dynkin diagram; we
obtain in this case a description of the relative type of $\Lp(\fgb,\sg)$ in terms of a projection
of the root system of $\fg$.

As indicated previously, Section \ref{sec:char} contains our results about the relationship between $\bbM_2$,
$\bbI_2$ and $\bbE_2$.

In Sections \ref{sec:isomat} and \ref{sec:projaff},
we study the isomorphism problem for loop algebras as in \eqref{eq:M2form}.
In Section  \ref{sec:isomat}, we solve the problem in the rotation case,
using calculations involving cyclic algebras over the field of rational functions in 2 variables.
In  Section \ref{sec:projaff}, we use projections of affine root systems to compute the relative type
of $\Lp(\fgb,\sg)$ in the transitive case.  For this we use methods
of Fuchs, Schellekens and Schweigert \cite{FSS} and of Bausch  \cite{Bau}.

Section \ref{sec:class} contains the classification theorem for algebras in $\bbM_2$.
The algebras are listed and assigned labels  in Table \ref{tab:reltypeu}.
The list consists of 14 infinite  families and 9 exceptional algebras.

In the last section,  Section \ref{sec:links}, we establish links with some work of other
researchers by calculating the index of each algebra $\cL$ in $\bbM_2$ and the Saito extended affine root
system \cite{Sai} of each isotropic algebra $\cL$ in $\bbM_2$ (see Table \ref{tab:TitsEARSu}).
Here   the index of $\cL$ is defined to be the index, in the sense of  \cite{T}, of the connected component of the identity in the automorphism group
of the central closure of $\cL$ over its centroid;
whereas the SEARS of $\cL$ is defined to be the set of nonisotropic roots
of a certain EALA of nullity 2 whose core modulo its centre is $\cL$.

It  turns out that by the classification theorem there is only one infinite family of algebras $\cL \in \bbM_2$ that satisfy the following condition:
\begin{itemize}
\item[(AA)] The absolute type of $\cL$ is $\type{A}{\rkfin}$ for some $\rkfin\ge 1$ and the relative type of $\cL$ is $\type{A}{r}$ for some $r\ge 0$.
\end{itemize}
(Saying that $\cL$ has relative type $\type{A}{0}$ means that $\cL$ is anisotropic.
See Definition \ref{def:reltype1}.)
Moreover, we see in Sections \ref{sec:class} and \ref{sec:links}
that with the exception of these algebras, an algebra in $\bbM_2$ is determined up to isomorphism by
its relative and absolute type (together), or by its index, or  by its  SEARS.  The second of these facts was conjectured in \cite{GP1}.

To conclude this introduction, we make a brief  comment on the title of this article.  Originally
this project was envisaged as a sequence of three papers,
beginning with \cite{ABP1} and \cite{ABP2},
on covering algebras, an intentionally imprecise term describing
the various incarnations of loop algebras that occur
in the study of EALAs. (See the introduction to \cite{ABP1}.) For this reason our working title was
\emph{Covering algebras III: Classification of nullity 2 multiloop Lie algebras}.  However,
as the sequence expanded and our topic broadened, we decided that the present title more accurately reflects
the content of the paper.

\smallskip
\noindent\textbf{Acknowledgements:}  We wish to thank Erhard Neher and John Faulkner for many helpful conversations about
multiloop algebras, extended affine Lie algebras and their root systems.  In particular, John's collaboration
on \cite{ABFP1} and \cite{ABFP2} was essential in establishing the link between these topics.
We also wish to thank the Banff International Research Station, where several of the results
in this paper were obtained during a Research in Teams Programme in 2005.  Finally,
we thank the Mathematisches Forschungsinstitut Oberwolfach  for the opportunity
to announce the main results of this paper in \cite{ABPOb}.

\section{Preliminaries} \
\label{sec:prelim}

In this short section, we establish some notation and conventions that we will use in this work.

\emph{We suppose  throughout the paper that $\base$ is an algebraically closed field of characteristic~0.}

\subsection{Algebras}\

Unless  indicated otherwise, \emph{algebra} will mean algebra over~$\base$.

Many  of the algebras that we will consider have natural
gradings, for example by $\bbZ^n$ or by a root lattice.  Nevertheless,
unless specified to the contrary, \emph{two algebras will be regarded as
isomorphic if there is an ungraded isomorphism between them}.

If  $F$ is a unital commutative  associative $\base$-algebra and $\cL_1$ and $\cL_2$
are algebras over $F$, we write $\cL_1 \simeq_F \cL_2$  to mean that
$\cL_1$ and $\cL_2$ are isomorphic
as algebras over $F$.  If $F = \base$ we simply write
$\cL_1 \simeq \cL_2$.

If $\cL$ is an algebra, the automorphism group of $\cL$ over
$\base$ is denoted by $\Aut_\base(\cL)$, or simply $\Aut(\cL)$.

If $\cL$ is a Lie algebra, the derived algebra of $\cL$ is (unless indicated otherwise)
denoted by  $\cL'$ and
the centre of $\cL$ is denoted by $\Centre(\cL)$.

If $n$ is a positive integer, we let
\[R_n = \base[t_1^{\pm 1},\dots,t_n^{\pm 1}]\]
denote the algebra of Laurent polynomials in the variables $t_1,\dots,t_n$ over $\base$.

\subsection{Matrix Lie algebras} \
\label{subsection:matrix}

Suppose that $\cA$ is a unital associative algebra.  We use the notation $\cA^-$ for
the Lie algebra $\cA$ under the commutator product $[x,y] = xy-yx$.

If $g\ge 1$,  let $\Mat_g(\cA)$ denote
the associative algebra of $n\times n$-matrices over $\cA$.  We then let
$\gl_g(\cA) = \Mat_g(\cA)^-$ and $\spl_g(\cA) = \gl_g(\cA)'$.
It is well known (and easily checked) that
\[\spl_g(\cA) = \set{x\in \gl_g(\cA) \suchthat \tr(x) \in [\cA,\cA]}.\]
In particular $\spl_1(\cA) = [\cA,\cA]$ under the commutator product.
Finally, if $u$ is a unit in $\Mat_g(\cA)$, we define $\Ad(u)\in \Aut(\Mat_g(\cA))$ by
$\Ad(u) x = u x u^{-1}$ for $x\in \Mat_g(\cA)$.  Then $\Ad(u)$ is also an automorphism
of $\gl_g(\cA)$, and we denote its restriction to $\spl_g(\cA)$ also by $\Ad(u)$.

\subsection{Root systems}
\label{subsec:rtsystem}

If $\fg$  is a Lie algebra and $\fh$ is  ad-diagonalizable abelian subalgebra of $\fg$,
then we have the \emph{root space decomposition} $\fg = \sum_{\al\in \fh^*} \fg_\al$ of $\fg$ with respect to $\fh$,
where $\fh^*$ is the dual space of $\fh$ and $\fg_\al = \set{x\in \fg \suchthat [h,x] = \al(h)x \text{ for } h\in \fh}$.  An element
$\al\in \fh^*$ will be called a \emph{root} of $\fg$ relative to $\fh$ if $\fg_\al \ne 0$, and
the set $\set{\al\in \fh^* \suchthat \fg_\al \ne 0}$ is called the \emph{root system of $\fg$}
relative to $\fh$.
We emphasize that with this definition $0$ is a root of $\fg$ relative to $\fh$ unless $\fg = \set{0}$.

We use the standard definition of finite (not necessarily reduced)
root system, except that we regard $0$ is an element of the system.
Thus, in this article a \emph{finite root system} is a finite subset $\Dl$ of a vector space over  $\base$
such that $0\in \Dl$ and $\Dl\setminus\set{0}$ is
a finite root system as defined in  \cite[Chap. VI, \S 1, Def.~1]{Bo2}.
Our notation for the type of a   reduced   irreducible finite
root system is standard and follows \cite[Plates I-IX]{Bo2}.  As usual we identify
$\type{A}{1} =  \type{B}{1} = \type{C}{1}$, $\type{B}{2} = \type{C}{2}$ and $\type{A}{3}= \type{D}{3}$.
In addition we will use the notation $\type{BC}{\rkfin}$ for the type of the unique  nonreduced irreducible finite root system of rank $\rkfin\ge 1$
\cite[Chap. VI, \S 4, no.~14]{Bo2}.

\subsection{Symmetrizable Kac-Moody Lie algebras}\
\label{subsec:symm}

Here we establish the notation we will use for symmetrizable Kac-Moody Lie algebras.
The reader can consult
\cite[Chapters 1--4]{K2}, [MP, Chapter 4] or \cite[\S 4]{KW}
for any necessary background regarding these algebras.

Suppose that $A = (a_{ij})_{i,j \in \ttI}$ is an indecomposable symmetrizable GCM (generalized Cartan matrix),
where $\ttI$ is a finite subset of $\bbZ$.
The (symmetrizable) Kac-Moody Lie algebra determined by $A$ is the Lie algebra
$\fg = \fg(A)$ generated by $\fh$ and the symbols $\set{e_i}_{i\in \ttI}$
and $\set{f_i}_{i\in \ttI}$ subject to the relations
\begin{gather*}
[\fh,\fh] = \set{0},\quad [e_i,f_j] = \delta_{ij} \al_i\ck,\quad
[h,e_i] = \al_i(h)e_i,\quad [h,f_i] = -\al_i(h)f_i,,\\
\ad(e_i)^{1-a_{ij}}(e_j)= \ad(f_i)^{1-a_{ij}}(f_j) =0 \ (i\ne j),
\end{gather*}
where  $(\fh,\Pi,\Pi\ck) = (\fh,\set{\al_i}_{i\in \ttI},\set{\al_i}\ck_{i\in \ttI})$
is a realization of $A$.

Let $\Dl$ be the set of roots (including 0)
of $\fg$ relative to $\fh$,
let $Q = Q(\Dl) = \sum_{i\in \ttI} \bbZ \al_i$ be the root lattice of $\fg$,
let $\fg = \sum_{\al\in Q} \fg_\al$ be the root space decomposition of $\fg$, let $W$ be the Weyl group of $\fg$,
and let
$\Dl^\real = \cup_{i\in \ttI}W\al_i$ be the set of real roots of $\fg$.

Let $\fgp$ be the derived algebra of $\fg$, in which case
$\Centre(\fg) =
\Centre(\fgp) = \set{h\in \fh : \al_i(h) = 0 \text{ for } i\in \ttI}$.
We  set
\[\fgb := \fgp / \Centre(\fgp),\]
and let
$-  : \fgp \to \fgb$ be the canonical map.
We also let
\[\fh'  := \fh\cap \fgp = \textstyle \sum_{i\in \ttI} \base\al_i\ck.\]
(Of course, contrary to our convention, $\fh'$ is not the derived algebra of $\fh$.)
Note that $\fgp =  \sum_{\al\in Q} (\fgp)_\al$ and
$\fgb =  \sum_{\al\in Q} (\fgb)_\al$ are $Q$-graded algebras with
\[(\fgp)_\al = \fg_\al \cap \fg \andd (\fgb)_\al = \overline{(\fgp)_\al}\]
for $\al\in Q$.

Our main interest is in the special cases  when $A$ is of finite type
or $A$ is of affine type. If $A$ is of finite type, $\fg$ is a split simple Lie algebra with splitting Cartan subalgebra
$\fg$, and hence $\Dl$ is an irreducible reduced finite root system.

If $A$ is of affine type, $\fg$  is an affine Kac-Moody Lie algebra.
We  label $A$
using Tables  Aff1--Aff3 of
\cite[Chapter 4]{K2}.\footnote{A different system is used in \cite{MP} to label affine matrices.
We will see in Remark \ref{rem:labels}
that the system we are using from \cite{K2} is related to the absolute
type of $\fgb$, whereas the system in \cite{MP} is related to the relative type
of $\fgb$.}
If $A$ has label $\typeaff{X}{k}{m}$ using this
system, we say that $\fg$ (or $A$) has \emph{type} $\typeaff{X}{k}{m}$.
We say that
$\fg$ (or $A$)  is \emph{untwisted} (resp.~\emph{twisted}) if $m = 1$ (resp.~$m \ne 1$).  We will recall some
more (standard) notation for affine algebras when needed in Section~\ref{subsec:notationaff}.

\section{Relative and absolute type}
\label{sec:type}

In this section, we discuss some isomorphism invariants for Lie algebras that will play a crucial role in
our work on classification.

\subsection{The centroid}\
\label{subsec:centroid}

We begin by recalling the definition of the centroid.  Since we will work
with associative algebras as well as Lie algebras later in the paper,
we make the initial  definitions for arbitrary (not necessarily Lie) algebras.

\begin{definition}
If $\cL$ is an arbitrary algebra,  the \emph{centroid }of $\cL$ is the  subalgebra of $\End_\base(\cL)$ defined by
\[\Cd(\cL) = \Cd_\base(\cL) := \set{e\in \End_\base(\cL) \suchthat e(xy) = e(x) y = x e(y) \text{ for } x,y\in \cL}.\]
If  $\cL$ is perfect ($\cL \cL = \cL$), then $\Cd(\cL)$  is commutative.
Note that $\cL$ is naturally a left $\Cd(\cL)$-module, and, if $\cL$ is perfect,
$\cL$ is an algebra over $\Cd(\cL)$.
$\cL$ is said to be \emph{fgc} if $\cL$ is finitely generated as a left $\Cd(\cL)$-module.
Finally, $\cL$ is said to be \emph{central} if  $\Cd(\cL) = \base 1$.
\end{definition}

\begin{pgraph}
\label{pgraph:centralsimple}  If $\cL$ is a simple algebra and $F = \Cd(\cL)$, then
$F$ is a field and $\cL$ is central simple as an algebra over $F$.  Conversely,
if $F$ is an extension of $\base$ and $\cL$ is a central simple algebra
over $F$, then $\cL$ is a simple algebra over $\base$ with centroid naturally isomorphic
to $F$.  (See \cite[\S X.1]{J} for these facts.)
\end{pgraph}

\begin{pgraph}
\label{pgraph:cfunctor}
If  $\cL_1$ and $\cL_2$ are algebras and
$\sg: \cL_1 \to \cL_2$ is an isomorphism,
then the map $\Cd(\sg) : \Cd(\cL_1) \to \Cd(\cL_2)$ defined by $\Cd(\sg)(e) = \sg e \sg^{-1}$ is an algebra
isomorphism.
\end{pgraph}

\subsection{Relative and absolute type for simple fgc Lie algebras}
\label{subsec:typesimple}

\begin{definition}
\label{def:reltype1}
Suppose that $\cL$ is a  simple fgc Lie algebra, and let
$F = \Cd(\cL)$.
Then, by  \pref{pgraph:centralsimple},
$\cL$  is a
finite dimensional central simple Lie algebra over $F$.
Choose a \emph{MAD $F$-subalgebra} $\cT$ of $\cL$, by which we mean a maximal ad-diagonalizable (necessarily abelian) $F$-subalgebra of $\cL$.
We say that
$\cL$ is \emph{isotropic} (resp.~\emph{anisotropic})
if $\cT \ne 0$ (resp.~$\cT = 0$). If $\cL$ is isotropic,
the root system  of $\cL$ relative to the adjoint action of $\cT$
is a (possibly non-reduced) irreducible finite root system  \cite[\S I.2]{Se2}, and
the type of that root system  is called the \emph{relative type} of $\cL$.  If $\cL$
is anisotropic we define the relative type of $\cL$ to be $\type{A}{0}$ (which  is not
the type of an irreducible finite root system).
These notions are well-defined since  $\Aut_F(\cL)$ acts transitively on the MAD $F$-subalgebras
of $\cL$ \cite[\S I.3, Thm.~2]{Se2}.
Also, by \cite[\S X.1, Lemma 1]{J},
$\cL \ot_F \bar F$ is a finite dimensional (central) simple Lie algebra over $\bar F$,
where $\bar F$ is an algebraic closure of $F$.
We  define the \emph{absolute type} of $\cL$
to be the type of the root system  of $\cL \ot_F \bar F$
relative to a MAD $\bF$-subalgebra of $\cL \ot_F \bar F$.
In other words, the absolute type of $\cL$ is the relative
type of  $\cL\ot_F\bar F$.
\end{definition}

\begin{pgraph} We emphasize that the absolute type of a simple fgc Lie algebra
is always reduced, whereas the relative type may not be.
For example, $\cK$ could have absolute type $\type{A}{2}$
and relative type $\type{BC}{1}$.
\end{pgraph}

\begin{remark}
\label{rem:typefd}
Suppose that $\fg$ is a finite dimensional simple Lie algebra.
Since $\base$ is algebraically closed, it follows that
$\fg$ is central \cite[Lemma X.1]{J}.  So, the relative
type of $\fg$ equals the absolute type, and as usual we call this the
\emph{type} of $\fg$.
\end{remark}

\subsection{Relative and absolute type for prime perfect fgc Lie algebras}\
\label{subsec:typeprime}

\begin{pgraph}  Recall   that a Lie algebra $\cL$ is said to be \emph{prime} if
$\cL \ne 0$ and, for all ideals
$\cI$ and $\cJ$ of $\cL$, $[\cI,\cJ] = 0$ implies that $\cI= 0$ or $\cJ = 0$.
\end{pgraph}

\begin{proposition}
\label{prop:cclosure} Suppose that $\cL$ is a prime perfect fgc Lie  algebra.
Then $\Cd(\cL)$ is an integral domain and
\[\tcL := \cL \ot_{\Cd(\cL)}
\tCd(\cL)\]
is a simple fgc Lie algebra with centroid naturally isomorphic to $\tCd(\cL)$,
where $\tCd(\cL)$ denotes  the quotient field of $\Cd(\cL)$.  Moreover, the map
$x \mapsto x\ot 1$ identifies $\cL$ as a    $\Cd(\cL)$-subalgebra of  $\tcL$.
\end{proposition}

\begin{proof}  It is easy to see (and well known) that $\Cd(\cL)$ is an integral domain.
It is proved in \cite[Prop. 8.7]{ABP2.5} that $\tcL$ is a finite dimensional
central simple Lie algebra over $\tCd(\cL)$,   so $\tcL$ is simple over $\base$ by
\pref{pgraph:centralsimple}.  The  last statement
follows from \cite[Lemma 3.3(i)]{ABP2.5}.
\end{proof}

\begin{definition}
\label{def:reltype2}
Suppose that $\cL$ is a prime perfect fgc Lie algebra.  We call the
Lie algebra $\tcL$ in Proposition \ref{prop:cclosure}
the \emph{central closure} of $\cL$.  We define the \emph{relative type}
and the \emph{absolute type} of $\cL$ to be the relative type
and the absolute type respectively of the Lie algebra $\tcL$.
We say that $\cL$ is \emph{isotropic} (resp.~\emph{anisotropic}) if
$\tcL$ is isotropic (resp.~anisotropic).
\end{definition}

The following  tells us that the notions just defined are isomorphism invariants.

\begin{lemma}
\label{lem:typeinvariant}
Suppose that $\cL_1$ and $\cL_2$
are prime prefect fgc Lie algebras that are isomorphic (as $\base$-algebras). Then
$\cL_1$ and $\cL_2$ have the same relative type and the same absolute type.
\end{lemma}

\begin{proof}  It is easy to check using
\pref{pgraph:cfunctor} that the central closures
of $\cL_1$ and $\cL_2$ are isomorphic.  So,  replacing $\cL_i$ by $\tcL_i$, we can assume
that $\cL_i$ is a simple fgc algebra for $i=1,2$.
Let $F_i = \Cd(\cL_i$) and let
$\bar F_i$ for a algebraic closure of $F_i$ for $i=1,2$.
Again using  \pref{pgraph:cfunctor}, it is easy to see that
$\cL_1 \otimes_{F_1}\bar F_1$ and $\cL_2 \otimes_{F_2}\bar F_2$
are isomorphic.  Thus, it suffices to prove the statement
about relative type.
Suppose that $\varphi : \cL_1 \to \cL_2$ is a $\base$-algebra
isomorphism, and let $\gamma = \Cd(\varphi) : F_1 \to F_2$ be the induced
isomorphism.   Let $\cT_1$ be a MAD $F_1$-subalgebra of $\cL_1$, and set $\cT_2 = \varphi(\cT_1)$.
Then it is straightforward to check that
$\cT_2$ is a MAD $F_2$-subalgebra of $\cL_2$, and that the map
$\alpha \mapsto \gamma\circ  \alpha\circ  \varphi^{-1}$ is an isomorphism
of the root system of $\cL_1$ with respect to $\cT_1$
onto the root system of $\cL_2$ with respect to  $\cT_2$.
\end{proof}

\begin{proposition}
\label{prop:anisotropic}  Suppose that $\cL$ is a prime perfect fgc Lie algebra.
Then, $\cL$ is anisotropic if and only if the only ad-nilpotent element of
$\cL$ is{\/} $0$.
\end{proposition}

\begin{proof}  Let $\tcL$ be the central closure of $\cL$.
It is well known that $\tcL$ is anisotropic if and only if
the only ad-nilpotent element of $\tcL$ is $0$.  (This follows from
the fact that any nonzero ad-nilpotent element of $\tcL$ is part
of an $\spl_2$-triple \cite[Thm.~III.17]{J}.) Since the ad-nilpotent
elements of $\tcL$ are precisely the $\tCd(\cL)$-multiples of the ad-nilpotent
elements of $\cL$, the conclusion  follows.
\end{proof}

\section{Multiloop algebras, loop algebras and the classes $\bbM_n$ and $\bbI_n$}
\label{sec:loop}

For the rest of the paper \emph{we fix a compatible family $\set{\zeta_m}_{m \ge 1}$ of roots of unity in $\base$};
this means that
$\zeta_m$ is a primitive $m^\text{th}$ root of unity for $m\ge 1$ and
\[\zeta_{m \kay}^\kay = \zeta_m\]
for all $\kay, m\ge 1$.
For  each positive integer $m$, we let $\bbZ_m = \bbZ/m\bbZ  = \set{\bar k \suchthat k\in \bbZ}$ be the ring of integers
modulo $m$, where $\bar k = k + m\bbZ$ for $k\in \bbZ$.

In this section, we recall the definition of  multiloop algebras and loop algebras, and discuss some of their properties.
Some of these definitions and properties will be stated for arbitrary algebras.

\subsection{Multiloop algebras} \
\label{subsec:multiloop}

Suppose that   \textit{$n$ is a positive integer}.
Let
\[S_n = \textstyle
\base[z_1^{\pm 1},\dots,z_n^{\pm1}] = \sum_{\boldkay\in\Zn} \base z^\boldkay\]
be the $\Zn$-graded algebra of Laurent polynomials
over~$\base$, where $z^\boldkay = z_1^{\kay_1}\dots z_n^{\kay_n}$ for $\boldkay = (\kay_1,\dots,\kay_n)\in\Zn$.

Suppose that $\fg$ is an arbitrary algebra over~$\base$,
$\boldm = (m_1,\dots,m_n)$ is an $n$-tuple of positive integers,
$\boldsg = (\sg_1,\dots,\sg_n)$ is an $n$-tuple of commuting
automorphisms of $\fg$ such that $\boldsg^\boldm = \boldone$
(that is  $\sg^{m_i} = 1$ for $1\le i \le n$), and $\boldprim =  (\prim_1,\dots,\prim_n)$ is an  $n$-tuple of elements of
$\base$ such that $\prim_i$ has order $m_i$ in $\base^\times$ for $1\le i \le n$.

\begin{definition}
\label{def:multiloop} \headtt{Multiloop algebra}
Let  $\Zmn = \bbZ_{m_1}\oplus \dots \oplus \bbZ_{m_n}$
and let $\boldkay = (\kay_1,\dots,\kay_n) \mapsto \bboldkay = (\bkay_1,\dots,\bkay_n)$
be the canonical group homomorphism from $\Zn$ onto $\Zmn$.
For $\boldkay = (\kay_1,\dots,\kay_n) \in \Zmn$,
let
\begin{equation}
\label{eq:simeig}
 \fg^\bboldkay = \set{u\in \fg \suchthat \sg_j u
= \prim_j^{\kay_j} u \text{ for } 1\le j\le n},
\end{equation}
in which case
$\fg = \sum_{\bboldkay\in \Zmn} \fg^\bboldkay$ is a $\Zmn$-graded algebra.
The \textit{multiloop algebra
of $\fg$ relative to $\boldsg$,  $\boldm$ and $\boldprim$} is the $\Zn$-graded subalgebra
\[
\textstyle
\Lpbm(\fg,\boldsg,\boldprim) = \sum_{\boldkay\in \Zn}
\fg^\bboldkay \ot z^\boldkay
\]
of $\fg \ot_\base S_n$.
In particular, we set
\[
\Lpbm(\fg,\boldsg) := \Lpbm(\fg,\boldsg,(\zeta_{m_1},\dots,\zeta_{m_n}))
\]
and call $\Lpbm(\fg,\boldsg)$ the
\textit{multiloop algebra
of $\fg$ relative to $\boldsg$  and
$\boldm$}.\footnote{In \cite{ABP2.5}, \cite{ABFP1} and \cite{ABFP2},  the notation
$\operatorname{M}_\boldm(\fg,\boldsg)$ was used in place of $\Lpbm(\fg,\boldsg)$.  We have changed notation
here to be compatible with the usual notation when $n=1$.  (See Definition \ref{def:loop} below.)}
\end{definition}

\begin{lemma} \
\label{lem:changemult}
\begin{itemize}
\item[(a)]
If $\ph\in \Aut(\fg)$, then $\Lpbm(\fg,\ph \boldsg\ph^{-1},\boldprim) \simeq  \Lpbm(\fg,\boldsg,\boldprim)$
as $\Zn$-graded algebras,
where $\ph \boldsg\ph^{-1} = (\ph \sg_1 \ph^{-1},\dots, \ph \sg_n \ph^{-1})$.

\item[(b)] If $r_i, s_i \in \bbZ$ are relatively prime to $m_i$ and $r_i s_i \equiv 1 \pmod {m_i}$ for $1\le i \le n$,
then $\Lpbm(\fg,\boldsg,(\prim_1^{r_1},\dots,\prim_n^{r_n} )) = \Lpbm(\fg,(\sg_1^{s_1},\dots,\sg_n^{s_n} ),\boldprim)$
as $\Zn$-graded algebras.

\item[(c)] If $\mathbf r = (r_1,\dots,r_n)$ is an $n$-tuple of integers such that
$\boldsg^{\mathbf r} = \boldone$ and $r_i | m_i$ for $1\le i \le n$, then
$\Lp_{\mathbf r}(\fg,\boldsg,(\prim_1^{m_1/ r_1},\dots,\prim_1^{m_n/ r_n}))\simeq \Lp_{\boldm}(\fg,\boldsg,\boldprim)$.
\end{itemize}
\end{lemma}

\begin{proof} (a) and (b): For  the case $n=1$, see  \cite{ABP2}.  The general case is easily established along similar lines.

(c): The linear extension of $x\ot z_1^{k_1}\dots z_n^{k_n} \mapsto x\ot z_1^{\frac {m_1} {r_1} k_1}\dots z_n^{\frac {m_n} {r_n} k_n}$
is an injective  algebra endomorphism of $\fg \ot_\base S_n$ which maps
$\Lp_{\mathbf r}(\fg,\boldsg,(\prim_1^{m_1/ r_1},\dots,\prim_1^{m_n/ r_n}))$ onto
$\Lp_{\boldm}(\fg,\boldsg,\boldprim)$.
\end{proof}

\begin{pgraph} By  Lemma \ref{lem:changemult}(b), any graded algebra of the form $\Lpbm(\fg,\boldsg,\boldprim)$
is equal to a graded algebra of the form $\Lpbm(\fg,\boldsymbol \tau)$ for some $\boldsymbol \tau$. Thus,
there is no loss of generality in considering only  graded algebras of the form $\Lpbm(\fg,\boldsg)$.
For the most part, we will do this in what follows.
\end{pgraph}

\begin{pgraph}  By  Lemma
\ref{lem:changemult}(c),   $\Lpbm(\fg,\boldsg)$  does not depend up to (ungraded) isomorphism
on the period $\boldm$. Therefore, when we are regarding
$\Lpbm(\fg,\boldsg)$ as an ungraded algebra, we often denote it simply by
$\Lp(\fg,\boldsg)$.
An algebra of the form $\Lp(\fg,\boldsg)$ for some
$\boldsg$ as above, will be called an
\emph{$n$-fold multiloop algebra of~$\fg$}.
\end{pgraph}

\begin{pgraph}  An $n$-fold  multiloop algebra of $\fg$ is a Lie algebra (resp.~an associative algebra)
if and only if $\fg$ is a Lie algebra (resp.~an associative algebra).
\end{pgraph}

\begin{pgraph}
\label{pgraph:identRS}
Using the n-tuple $\boldm$, we may identify $R_n = \base[t_1^{\pm 1},\dots,t_n^{\pm 1}]$ as a
subalgebra of $S_n= \base[z_1^{\pm 1},\dots,z_n^{\pm 1}]$ by setting
\[t_i = z_i^{m_i}\]
for $1\le i \le n$.  Now $\fg \otimes S_n$ is an  algebra over $S_n$ and hence also an
algebra over $R_n$.  Furthermore, $\Lpbm(\fg,\boldsg,\boldprim)$ is an $R_n$-subalgebra
of $\fg \otimes S_n$, and in this way $\Lpbm(\fg,\boldsg,\boldprim)$ is an algebra over $R_n$.
So we have a natural homomorphism of $R_n$
into the centroid  of $\Lpbm(\fg,\boldsg,\boldprim)$.
\end{pgraph}

Since  $\base$ is algebraically closed, finite dimensional simple algebras
are central simple \cite[Lemma X.1]{J}.
Thus, by \cite[Thm. 6.2 and Cor. 6.6]{ABP2.5}  (see also \cite[Lemmas 4.1.2 and 4.6.3]{GP1}),
we have:

\begin{proposition} \headtt{The centroid of a multiloop algebra}
\label{prop:centloop}  If
$\fg$ is  finite dimensional and simple, then the natural homomorphism
from $R_n$ into  $\Cd(\Lpbm(\fg,\boldsg,\boldprim))$ is an isomorphism \emph{(which
we often treat as an identification)}.
\end{proposition}

The following simple properties of multiloop  algebras
will also be useful.

\begin{lemma}
\label{lem:loopfactor}
Suppose  $\mathfrak b$  is an ideal of $\fg$ with $\sg_i(\mathfrak b) = \mathfrak b$ for $1\le i \le n$.
Let $\boldsg|_{\mathfrak b}$ denote the $n$-tuple  of automorphisms of $\mathfrak b$ obtained by restricting $\boldsg$, and let
$\bar \boldsg$ denote the $n$-tuple of automorphisms of $\bar\fg = \fg/\mathfrak b$ induced by $\boldsg$.
Then $\Lpbm(\mathfrak b,\boldsg|_\mathfrak b)$ is a $\bbZ^n$-graded
ideal of $\Lpbm(\fg,\boldsg)$ and $\Lpbm(\fg,\boldsg)/\Lpbm(\mathfrak b,\boldsg|_\mathfrak b)$ is graded
isomorphic to $\Lpbm(\bar\fg,\bar\boldsg)$.
\end{lemma}

\begin{proof}  Let $- : \fg \to \bar\fg$ be the canonical homomorphism.  Then
the map $\psi: a \ot z^\boldkay$ to $\bar a \ot z^\boldkay$
is a surjective graded homomorphism
of $\Lpbm(\fg,\boldsg)$ onto $\Lpbm(\bar\fg,\bar\boldsg)$. Also
$\Lpbm(\mathfrak b,\boldsg|_\mathfrak b)$ is contained in $\Ker(\psi)$; and, since
$\Ker(\psi)$ is graded,  the reverse inclusion is clear.
\end{proof}

\begin{lemma}
\label{lem:loopderived}
Suppose that $\fg$ is a Lie algebra.
Then
\begin{itemize}
\item[(a)] $\Lpbm(\fg,\boldsg)' = \Lpbm(\fg',\boldsg|_{\fg'})$.
\item[(b)]  $\Centre(\Lpbm(\fg,\boldsg)) = \Lpbm(\Centre(\fg),\boldsg|_{\Centre(\fg)})$.
\item[(c)] $\Lpbm(\fg,\boldsg)'/\Centre(\Lpbm(\fg,\boldsg)')$ is graded isomorphic to
$\Lpbm(\fg'/\Centre(\fg'),\overline{\boldsg|_{\fg'}}\ )$.
\end{itemize}
\end{lemma}

\begin{proof} (a) and (b) are clear since the left hand sides are  $\bbZ^n$-graded.
(c) now follows using (a), (b) (applied to $\fg'$) and Lemma \ref{lem:loopfactor}.
\end{proof}

Finally, we recall the following result on the permanence of absolute type for multiloop algebras from
\cite[Thm.~8.16]{ABP2.5}.

\begin{proposition}
\label{prop:loopperm}
Suppose that $\cL$ is an $n$-fold  multiloop algebra of a Lie algebra $\fg$.
If $\fg$ is prime, perfect and fgc,  then so is $\cL$.  Moreover, in that case,
the absolute type of $\cL$ is equal to the absolute type of~$\fg$.
\end{proposition}

\subsection{Loop algebras} \
\label{subsec:loopdef}

The special case of Definition \ref{def:multiloop} when $n=1$ is of particular importance.

\begin{definition} \headtt{Loop algebra}
\label{def:loop}
Suppose that
$\fg$ is an algebra over~$\base$, $m$ is a positive integer,
$\sg\in \Aut(\fg)$ with $\sg^m = 1$, and $\prim\in \base^\times$ has order $m$.
The  \textit{loop algebra
of $\fg$ relative to $\sg$,  $m$ and $\prim$} is the $\bbZ$-graded  subalgebra
\[
\textstyle
\Lpm(\fg,\sg,\prim) = \sum_{\kay\in\bbZ}
\fg^\bkay \ot z_1^\kay
\]
of $\fg\ot S_1$, where $S_1= \base[z_1^{\pm 1}]$ and
$\fg^\bkay = \set{u\in \fg \suchthat \sg u
= \prim^{\kay} u }$
for $\kay\in\bbZ$.
We mainly consider the special case
\[\Lpm(\fg,\sg) := \Lpm(\fg,\sg,\zeta_m),\]
which is called the  \textit{loop algebra of $\fg$ relative to $\sg$ and $m$}.
When we are regarding
$\Lpm(\fg,\sg)$ as an ungraded algebra, we often denote it by
$\Lp(\fg,\sg)$.  An algebra of the form $\Lp(\fg,\sg)$ for some
$\sg\in\Aut(\fg)$ of finite order  will be called a \emph{loop algebra of $\fg$}.
\end{definition}

\begin{remark}
\label{rem:itloop}   When $n\ge 2$, we have
\begin{equation*}
\label{eq:itloop}
\Lpbm(\fg,\boldsg,\boldprim) \simeq
\Lp_{m_n}(\Lp_{\boldm'}(\fg,\boldsg',\boldprim'),\sg_n\ot 1,\prim_n),
\end{equation*}
where $\boldm'$, $\boldsg'$ and $\boldprim'$ are obtained from $\boldm$, $\boldsg$
and $\boldprim$ respectively by deleting the last entry,
and
$\sg_n\ot 1$ denotes the restriction of $\sg_n\ot 1 \in \Aut(\fg\ot S_{n-1})$ to
$\Lp_{\boldm'}(\fg,\boldsg',\boldprim')$
\cite[Example 5.4]{ABP2.5}.  Thus multiloop algebras can be constructed by a sequence of loop
constructions.
\end{remark}

\subsection{The class $\bbM_n$, $n\ge 0$} \
\label{subsec:Mn}

\begin{pgraph}
By
convention, we will regard an algebra  as a $0$-fold multiloop algebra of itself.
\end{pgraph}

\begin{definition}
\label{def:Ln}  If $n$ is a nonnegative integer, we let $\bbM_n$ be the class of all algebras that are isomorphic to
$n$-fold  multiloop algebras of finite dimensional simple Lie algebras.
Note in particular, that $\bbM_0$ is the class of all finite dimensional simple Lie algebras.
We write $\cL\in \bbM_n$ to mean
that $\cL$ is an algebra in the class $\bbM_n$.  Algebras
in $\bbM_n$ will be called  \emph{nullity $n$ multiloop algebras}.
\end{definition}

\subsection{The class $\bbI_n$, $n\ge 0$} \
\label{subsec:In}

\begin{definition}
\label{def:In} Let $n\ge 0$. We define the class of Lie algebras $\bbI_n$ inductively.
First let $\bbI_0 = \bbM_0$ be the class of finite dimensional simple Lie algebras.
Second, if $n\ge 1$, we let $\bbI_n$ be the class of algebras isomorphic to loop algebras
of algebras in $\bbI_{n-1}$.  Algebras in $\bbI_n$ will be called \emph{nullity $n$  iterated  loop algebras.}
\end{definition}

\begin{pgraph}
\label{pgraph:MnIn}
It follows by induction  using Remark \ref{rem:itloop} that
\[\bbM_n \subseteq \bbI_n;\]
that is every
algebra in the class $\bbM_n$ is in  $\bbI_n$.
\end{pgraph}

The following propositions follow  from Proposition \ref{prop:loopperm} and Remark \ref{rem:typefd}.

\begin{proposition}
\label{prop:ppfgc}
Any algebra  $\cL$ in $\bbI_n$ is a
prime perfect fgc Lie algebra, and hence its relative and absolute types are defined.
\end{proposition}

\begin{proposition}
\label{prop:absolute}
If $\cL$ is isomorphic to an $n$-fold multiloop algebra of a finite dimensional simple
Lie algebra $\gd$, then the absolute type of $\cL$ is equal to the type of~$\gd$.
\end{proposition}

\begin{pgraph}
\label{pgraph:abseasy}
Proposition \ref{prop:absolute} tells us how to compute the absolute type of an algebra in~$\bbM_n$.
The determination of the relative type of an algebra in $\bbM_n$ is a more difficult
problem.  In fact, in later sections, a lot of work will be devoted to the solution of this problem
when $n=1$ (see Corollary \ref{cor:KacReal2}) and  most especially when $n=2$
(see  Corollary \ref{cor:rot}
and Theorem \ref{thm:relativecalc}).
\end{pgraph}

\subsection{An example: the quantum torus $\Qp$}\
\label{subsec:quantum}

Quantum tori, which are constructed from multiplicatively alternating
matrices $\boldq \in  M_n(\base)$ \cite[\S 4.6.1]{Ma},
play a fundamental role in the study of extended affine Lie algebras and multiloop algebras
(see for  example \cite{BGK} and \S \ref{subsec:structurerot} below).
In this section, we consider the special case that will  be important
for our purposes: the case when $n=2$

\begin{pgraph}
\label{pgraph:quantum} If $\prim\in \base^\times$, let $\Qp$ be
the unital associative
algebra presented by the generators $x_1,x_2,x_1^{-1},x_2^{-1}$ subject to the inverse relations
$x_1x_1^{-1} = x_1^{-1}x_1 = 1$, $x_2x_2^{-1} = x_2^{-1}x_2 = 1$, and the relation
\[x_1x_2 = \prim x_2x_1.\]
$\Qp$ has a $\bbZ^2$-grading $\Qp= \bigoplus_{\boldkay\in \bbZ^2} \Qp^\boldkay$
with $\Qp^{(k_1,k_2)} = \base x_1^{k_1}x_2^{k_2}$.
We call
$\Qp$ the \emph{quantum torus} determined by $\prim$, and we call $x_1,x_2,x_1^{-1},x_2^{-1}$ the \emph{distinguished
generators} of $\Qp$ over $\base$.
\end{pgraph}

\begin{pgraph}
For the rest of the section, assume that $\prim$ is an \emph{element of $\base^\times$ of finite order~$m$}.
\end{pgraph}

\begin{pgraph}
\label{pgraph:centquantum}
It is easy to check that the centre $Z(\Qp)$ of $\Qp$ is the algebra generated
by $x_1^{\pm m}$ and $x_2^{\pm m}$.  Thus we can identify
\[R_2 = \base[t_1^{\pm 1},t_2^{\pm 2}] = Z(\Qp)\]
by setting
\[t_1 = x_1^m \andd t_2 = x_2^m.\]
In particular, $\Qp$ is an algebra over $R_2$.  Moreover,
$\Qp$ is a free $R_2$-module with basis $\set{x_1^{k_1}x_2^{k_2} \suchthat 0\le k_1, k_2 \le m-1}$.
\end{pgraph}

To describe $\Qp$ as a multiloop algebra, we need some notation (which will also
be used later in Section \ref{sec:looprot}).

\begin{pgraph}
\headtt{Pauli matrices} \
\label{pgraph:Pauli}
If $F$ is a commutative associative unital $\base$-algebra, $u$ is a unit in $F$,
and $m\ge 1$, we let
\[d_m(u) = \diag(1,u,u^2,\dots,u^{m-1}) \andd p_m(u) = \left[
\begin{array}{cccc}
   0 & \dots& 0 & u \\
   1 &   \ddots &   & 0   \\
   \vdots   &\ddots & \ddots & \vdots \\
     0 &  \dots & 1  & 0
 \end{array}
\right]\]
in $\Mat_m(F)$. (If $m=1$, we interpret $d_1(u) = [1]$ and $p_1(u) = [u]$.)
Now $d_m(u)$ and $p_m(u)$ are units in $\Mat_m(F)$, so we can set
\[D_m(u) = \Ad(d_m(u)) \andd P_m(u) = \Ad(p_m(u))\]
in the automorphism group of $\Mat_m(F)$, $\gl_m(F)$ or $\spl_m(F)$.
(See  Section \ref{subsection:matrix} for the notation $\Ad$.)
Note  that $p_m(u)^m = uI$, so $P_m(u)^m = I$.  Also, if $u^m=1$, then $d_m(u)^m = 1$ and $D_m(u)^m = 1$.
\end{pgraph}

\begin{pgraph} Now let
\[\boldm = (m,m) \andd \boldprim = (\prim,\prim).\]
We have seen in \pref{pgraph:identRS}
that for any algebra $\fg$ and any  pair $\boldsg = (\sg_1,\sg_2)$ of period $m$ automorphisms
of $\fg$, the multiloop algebra $\Lpbm(\fg,\boldsg,\boldprim)$ is naturally
an  $R_2$-algebra, where $R_2 = \base[t_1^{\pm 1}, t_2^{\pm 1}]$ with
$t_1 = z_1^m$ and $t_2 = z_2^m$.  With this in mind we have:
\end{pgraph}

\begin{lemma}\label{lem:quantum} \
Let $\boldsg = (\sg_1,\sg_2)$, where
\[\sg_1 = D_m(\theta)\andd \sg_2 = P_m(1)\]
in $\Aut(\Mat_m(\base))$. Let $g$ be a positive integer and identify $\Mat_{gm}(\base) = \Mat_g(\Mat_m(\base))$ in the usual fashion.
With this identification, let  $\boldsgt = (\tilde \sg_1,\tilde \sg_2)$ be the pair of automorphisms
of $\Mat_{gm}(\base)$ induced by $\boldsg$ acting on the $m\times m$-blocks.  Then
\begin{itemize}
\item[(a)] $\Qp \isoR \Lpbm(\Mat_m(\base),\boldsg,\boldprim)$.
\item[(b)] $\Mat_g(\Qp) \isoR \Lpbm(\Mat_{gm}(\base),\boldsgt,\boldprim)$.
\item[(c)] $\gl_g(\Qp) \isoR \Lpbm(\gl_{gm}(\base),\boldsgt,\boldprim)$ and under this isomorphism
\begin{equation}
\label{eq:quantumsl}
\spl_g(\Qp) \isoR \Lpbm(\spl_{gm}(\base),\boldsgt|_{\spl_{gm}(\base)},\boldprim).
\end{equation}
\item[(d)] $\gl_g(\Qp) = R_2 1 \oplus \spl_g(\Qp)$ as Lie algebras over $R_2$.
\item[(e)]  If $gm>1$, the natural homomorphism $R_2 \mapsto C(\spl_g(\Qp))$ is
an isomorphism.
\end{itemize}
\end{lemma}

\begin{proof}  (a): Let $p = p_m(1)$ and $d = d_m(\prim)^{-1}$ in $\Mat_m(\base)$.
Then one checks that $pd = \prim dp$,
\[\sg_1(p) = \prim p,\quad \sg_2(p) = p,\quad \sg_1(d) = d \andd \sg_2(d) = \prim d.\]
So $\tilde x_1 := p \ot z_1$ and $\tilde x_2 := d\ot z_2$ are  units in
$\Lpbm(\Mat_m(\base),\boldsg,\boldprim)$ satisfying
$\tilde x_1 \tilde x_2 = \theta \tilde x_2 \tilde x_1$.  Thus we have an algebra
homomorphism $\ph : \Qt \to  \Lpbm(\Mat_m(\base),\boldsg,\boldprim)$
such that  $x_i \mapsto \tilde x_i$ for $i= 1,2$.  Using the well known fact that
$\set{p^{k_1}d^{k_2} \suchthat 0 \le k_1,k_2\le m-1}$ is a basis for $\Mat_m(\base)$, it is easy to check
that
$\ph$ is an isomorphism (see the argument in \cite[Example 9.8]{ABP2.5}).  Moreover, under $\ph$ we have
$x_1^m \mapsto  \tilde x_1^m = p^m \ot z_1^m = 1 \ot t_1$ and $x_2^m \mapsto  \tilde x_2^m = d^m \ot z_2^m = 1 \ot t_2$,
so $\ph$ is $R_2$-linear.

(b):  This follows from (a).

(c): The first isomorphism in (c) follows from (b); and then \eqref{eq:quantumsl} follows  from
Lemma \ref{lem:loopderived}(a).

(d): Now $\gl_{gm}(\base) = \base 1 \oplus \spl_{gm}(\base)$. So we have
$\Lpbm(\gl_{gm}(\base) ,\boldsgt,\boldprim) = R_2 1 \oplus \Lpbm(\spl_{gm}(\base) ,\boldsgt|_{\spl_{gm}(\base)},\boldprim)$,
and (d) follows using (c).

(e):  This follows from \eqref{eq:quantumsl} using Proposition \ref{prop:centloop}.
\end{proof}

\begin{remark}  Parts (d) and (e) of the lemma can be proved without using multiloop algebras.  In particular,
if $g > 1$, the Lie algebra $\spl_g(\Qp)$ is graded by the root system $\type{A}{g-1}$, and hence
(e)  follows from \cite[Thm.~5.15]{BN}.
\end{remark}

\begin{proposition} \label{prop:spgquantum} Suppose $g$ is a positive integer and $\theta\in \base^\times$
has order $m$.
\begin{itemize}
\item[(a)]
If $gm > 1$, then $\spl_g(\Qp)\in \bbM_2$.
\item[(b)] If $m >1$, then $\spl_1(\Qp)$ has relative type $\type{A}{0}$; that is, $\spl_1(\Qp)$ is aniso\-tropic.
\end{itemize}
\end{proposition}

\begin{proof} (a) follows from \eqref{eq:quantumsl}  since $\spl_{gm}(\base)$ is simple.
For (b), let $\cL = \spl_1(\Qp)$, and suppose for
contradiction (see Lemma \ref{prop:anisotropic})
that $\cL$ contains a nonzero ad-nilpotent element $x$.  Now $\cL$ is a $\bbZ^2$-graded subalgebra
of $\Qp^-$, so the homogeneous components of $x$ lie in $\cL$.  Replacing $x$ by its component
of highest degree (using the lexicographic order), we can assume that $x$ is homogeneous.
But since $\cL\in \bbM_2$ is prime by Proposition \ref{prop:ppfgc}, $\cL$ has trivial centre.  Hence,
$[x,y] \ne 0$ for some nonzero homogeneous $y\in \cL$.  Since the homogeneous components
of $\Qp$ are each spanned by a unit, it follows that
$[x,y] = a xy$ for some $a\in \base^\times$.
But then, by induction, we have $\ad(x)^r y = a^r x^r y$ for $r\ge 1$, giving a contradiction.
\end{proof}

\begin{pgraph}
Proposition \ref{prop:spgquantum} holds more generally
for the quantum torus $\base_\boldq$   \cite[\S 4.6.1]{Ma}
constructed from a multiplicatively alternating matrix $\boldq \in  M_n(\base)$ such that
each $q_{ij}$ has finite order in $\base^\times$.  Specifically, if either $g > 1$
or some $q_{ij} \ne 1$, then
$\spl_g(\base_\boldq)\in \bbM_n$.  Moreover, $\spl_1(\base_\boldq)$ is anisotropic if
some $q_{ij} \ne 1$. Our proof of these statements uses
\cite[Thm.~9.2.1]{ABFP1} as well as the argument above.  We omit the details since
we will not use these   facts.
\end{pgraph}

\section{Extended affine Lie algebras and the class $\bbE_n$}
\label{sec:EALA}

In this section,  we recall
some background on extended affine Lie algebras.

\subsection{The definition} \
\label{subsec:EALAdef}

Following \cite{Neh2} and \cite[\S 6.11]{Neh3}, we have the  following:

\begin{definition} \headtt{EALA}
\label{def:EALA}
Let $\fg$ be a Lie algebra satisfying the following axioms:
\begin{itemize}
\item[(EA1)]  $\fg$ has a nondegenerate invariant symmetric bilinear form $\form$.
\item[(EA2)] $\fg$ contains a nontrivial finite dimensional self-centralizing ad-diagonal\-izable
subalgebra $\fh$.
\end{itemize}
Let  $\fg = \sum_{\al\in \fh^*} \fg_\al$ be the root space decomposition
of $\fg$ with respect to $\fh$, and let
\[\Phi := \set{\al \in \fh^* : \fg_\al \ne 0}\]
be the root system of $\fg$ relative to $\fh$.
The  form  $\form$ restricted to $\fh$ is nondegenerate, and hence, as usual, we can
transfer $\form$ to a nondegenerate symmetric bilinear form on the dual space $\fh^*$
of $\fh$.
Let
\[\Phi^\times = \set{\al\in \Phi \suchthat (\al \vert \al) \ne 0} \andd \Phi^0 = \set{\al\in \Phi \suchthat (\al \vert \al) = 0} \]
be the sets of \emph{nonisotropic} (resp.~\emph{isotropic}) roots in $\Phi$.  Let
\[\fg_\core :=\text{subalgebra of $\fg$ generated by $\fg_\al$, $\al\in \Phi^\times$}\]
be the \emph{core} of $\fg$.  We say that $(\fg,\form,\fh)$
is an \emph{extended affine Lie algebra} (EALA) if in addition to (EA1) and (EA2), the following
axioms hold:
\begin{itemize}
\item[(EA3)] $\ad(x)$ is locally nilpotent for $x\in \fg_\al$, $\al\in \Phi^\times$.
\item[(EA4)] $\Phi^\times$ cannot be decomposed as the union of two orthogonal nonempty subsets.
\item[(EA5)] The centralizer of $\fg_\core$ in $\fg$ is contained in $\fg_\core$.
\item[(EA6)] The subgroup $\langle \Phi^0\rangle$ of $\fh^*$ generated by $\Phi^0$  is finitely generated.
\end{itemize}
We will then often say for short that $\fg$ is an EALA.
\end{definition}

\begin{pgraph}  An   EALA over $\bbC$ such that $\Phi$ is a discrete subgroup of $\fh^*$
is called a \emph{discrete EALA} \cite[6.14]{Neh3}.  (In fact, (EA6) is redundant in this definition \cite[6.15]{Neh3}.)
Discrete EALAs have been studied in a number of papers
including \cite{H-KT}, \cite{BGK}, \cite{AABGP}, \cite{ABGP}, \cite{AG} and \cite{ABP1} where, beginning with  \cite{AABGP},
they were called \emph{tame EALAs}.
\end{pgraph}

\begin{pgraph}
\label{pgraph:Chapter2}
The properties of the root
systems of discrete EALAs were developed
in \cite[Chapter 2]{AABGP}. It was shown that these root systems
can be described using a finite root system together with a family
of semilattices.
\end{pgraph}

\begin{pgraph}
\label{pgraph:LN}
More generally, as described in \cite[\S 6]{Neh3},
if $\Phi$ is the root system of an EALA $\fg$  then
$\Phi$ has the following properties:
$\Phi$ is a reduced symmetric affine reflection
system with irreducible finite quotient root system; all root strings
are unbroken; $\langle \Phi^0\rangle$ is finitely generated;
and $\Phi^0 \subseteq \Phi^\times + \Phi^\times$.  (The interested reader can consult
\cite{Neh3} or \cite{LN} for the terms used here.  We will not use
them subsequently.)
As a consequence, it follows from \cite[\S  4-5]{LN} (see also
\cite[\S 3.7]{Neh3}) that
the structure of $\Phi$ can be described in terms of extension data for the
finite quotient root system  of $\Phi$.
This generalizes the approach mentioned in \pref{pgraph:Chapter2}  that uses
semilattices.
As a result, most of the basic properties of discrete EALAs
carry over to general EALAs with similar and sometimes shorter
proofs.
In a few cases below we will only sketch an argument along these lines,
as the details are not hard to fill in.
\end{pgraph}

\begin{pgraph}
\label{pgraph:finitetype}\headtt{The quotient type of an EALA}

Suppose  that $(\fg,\form,\fh)$ is an EALA with root system  $\Phi$.
Then, using the properties mentioned in \pref{pgraph:LN}, it is easy to deduce the following facts about $\Phi$
from standard facts about finite root systems.  First,
let
\[V = \spann_\bbQ(\Phi)\]
in
$\fh^*$.  Then  $V$ is finite dimensional over $\bbQ$.
Further, we can and do normalize $\form$, by replacing it by a $\base^\times$ multiple of itself, so that
$(\al \vert \beta) \in \bbQ$ and $(\al\vert\al)\ge 0$  for $\al,\beta\in \Phi$.
In that case, $\form$ restricts to a  $\bbQ$-bilinear form  $V\times V \to \bbQ$; and this restriction,
which we also denote by $\form$, is positive semi-definite (that is $(\al \vert \al)\ge 0$ for $\al\in V$).
We let
\[V^0 = \rad(V)\]
be the radical of this form, and we set
\[\bV = V/V^0\]
with canonical map $- : V \to \bar V$.  So
$\form$ induces a positive definite $\bbQ$-bilinear form on $\bV$, which we again denote by~$\form$.
Further, $\bPhi$ is an
irreducible (possibly nonreduced) finite root system  in $\bV$.
We call $\bPhi$ the \emph{finite quotient root system}   for $(\fg,\form,\fh)$, or simply for $\fg$,
and we call its type the
\emph{quotient type}  of $\fg$.
\end{pgraph}

\begin{remark}
\label{rem:finitetype}
Suppose that $\fg$ is a discrete EALA (over $\bbC$).  Then the properties described in \pref{pgraph:finitetype} hold with the rational field $\bbQ$
everywhere replaced by the real field $\bbR$, in which case one obtains a finite quotient root system
over $\bbR$ rather than $\bbQ$ \cite[Chap.~1]{AABGP}.  It is not difficult to check
that this finite root system over $\bbR$ is obtained by base field extension (as described
in \cite[Chap. VI, \S 1, Remark 1]{Bo2}) from the
one over $\bbQ$.  So these two finite root systems have the same type.
Hence,   the notion of quotient type defined above
coincides with the notion of  type defined in \cite[p.~27]{AABGP}.
\end{remark}

\begin{lemma} \label{lem:EALAbasic}
Suppose that $\fg$ is an EALA with root system $\Phi$. Then
\begin{itemize}
\item[(a)]  $\fg_c$ is a perfect ideal of $\fg$.
\item[(b)] $\Phi^0 = \set{\al\in \Phi \suchthat \bbZ\al \subseteq \Phi}.$
Hence, $\Phi^0$ and $\Phi^\times$ are determined by $\Phi$ (without reference
to the form $\form$).
\end{itemize}
\end{lemma}
\begin{proof} In the case of discrete EALAs, (a) is proved in
\cite[\S 1]{AG}, and (b) follows from \cite[Cor.~2.31]{AABGP}.
The arguments in general follow the same lines, as discussed in \pref{pgraph:LN}.
\end{proof}

\subsection{The nullity of an EALA}\
\label{subsec:nullity}

\begin{definition} Suppose that $\fg$ is an EALA with root system  $\Phi$.
Since the additive group of $\fh^*$ is torsion free, it follows from (EA6) that
$\langle \Phi^0 \rangle$ is a finitely generated free abelian group.  We define
the \emph{nullity} of $\fg$ to be the rank of the group $\langle \Phi^0 \rangle$.
\end{definition}

\begin{remark}
\label{rem:nullity1}
Let $\fg$ be an EALA with notation as in \pref{pgraph:finitetype}.

(a) Using extension data, one can easily see that  $\spann_\bbQ(\Phi^0) = V^0$ (see
\cite[(2.10) and (2.11)(b)]{AABGP} in the discrete case).
Thus, the nullity of $\fg$ equals $\dim_\bbQ(V^0)$.

(b) Suppose that $\fg$ is discrete (over $\bbC$).   Let $V_\bbR$ be the real span of $R$ and let $V_\bbR^0$ be the
radical of the restriction of $\form$ to $V_\bbR$.  Then, by \cite[Prop.~1.4 and Cor.~2.31]{AABGP},
the rank of $\langle \Phi^0\rangle$ equals $\dim_\bbR(V_\bbR^0)$.
Thus, the notion of nullity defined above
coincides with the notion of nullity defined in \cite[p.~27]{AABGP}.
\end{remark}

Using results of \cite{ABFP2} and \cite{ABGP}, we have the following characterizations of
EALAs of nullity 0 and 1.

\begin{proposition}\
\label{prop:lownullity}
\begin{itemize}
\item[(a)] If $\fg$ is a finite dimensional simple Lie algebra with Killing form
$\form$ and Cartan subalgebra $\fh$, then  $(\fg,\form,\fh)$ is an EALA of nullity 0.
Conversely, if $(\fg,\form,\fh)$ is an EALA of nullity 0, then   $\fg$ is a finite dimensional simple Lie algebra and $\fg_c = \fg$.

\item[(b)]  If $\fg$ is an affine Kac-Moody Lie algebra with normalized invariant form $\form$ (see
\cite[\S 6.2]{K2} or \pref{pgraph:notationaff} below)
and with distinguished Cartan
subalgebra $\fh$, then
$(\fg,\form,\fh)$ is an EALA of nullity 1.  Conversely, if
$(\fg,\form,\fh)$ is an EALA of nullity~1, then
$\fg$  is isomorphic to an affine Kac-Moody Lie algebra and $\fg_c = \fg'$.
\end{itemize}
\end{proposition}

\begin{proof}  (a):  The first statement follows from standard facts about finite dimensional simple
Lie algebras.  For the converse, suppose that $(\fg,\form,\fh)$ is an EALA of nullity 0 with root system
$\Phi$.  Then, by
\cite[Prop 6.4]{Neh3}, $\fg_c$ is a Lie $(\Lm,\bPhi)$-torus,
where $\Lm = \langle \Phi^0\rangle = \set{0}$
(see \pref{pgraph:Lietori} for this terminology).
So, by \cite[Remark 1.2.4]{ABFP2}, $\fg_c$ is finite dimensional and simple.
Now let $d\in \fg$.  Then, since $\fg_c$ is an ideal of $\fg$ and every
derivation of $\fg_c$ is inner, it follows that $d-e$ centralizes $\fg_c$ for some
$e\in \fg_c$.  So by (EA5),  $d= e$.  Hence, $\fg = \fg_c$ is finite dimensional simple.

(b):  The first statement follows from standard facts about affine algebras.
The  second statement is proved for discrete EALAs over $\bbC$ in \cite{ABGP},
and the proof given there can be easily adapted to handle the general case (as discussed in \pref{pgraph:LN}).
Indeed, suppose that $(\fg,\form,\fh)$ is an EALA of nullity 1.
Then, using extension  data, one sees,
as in \cite[p.~677--678]{ABGP},
that the root system $\Phi$ of $\fg$ is one of the root systems for an affine algebra. Using that information
the proof of Theorem 2.31 of \cite{ABGP} gives our conclusion.
\end{proof}

\begin{pgraph}   If $\fg$ is a finite dimensional simple (resp.~an affine Kac-Moody) Lie algebra,
we will subsequently regard $\fg$ as an EALA with the choices of $\form$ and $\fh$
made in the first sentence  of Proposition \ref{prop:lownullity}(a) (resp.~Proposition \ref{prop:lownullity}(b)).
\end{pgraph}

\begin{pgraph} If $\fg$ is finite dimensional simple, the  quotient type of $\fg$ is the type of $\fg$.  Also, for each affine GCM $A$, the quotient type
of $\fg(A)$ is well known (see Remark \ref{pgraph:quotientrs} below).
\end{pgraph}

\subsection{Centreless cores of EALAs and Lie tori}
\label{subsec:ccore}

\begin{definition}
Suppose that $\fg$ is an EALA with  root system  $\Phi$.
We let
 \[\fg_\ccore = \fg_\core / Z(\fg_\core).\]
Since $\fg_\core$ is perfect, $\fg_\ccore$ is \emph{centreless}  (that is it
has trivial centre).  For this reason
$\fg_\ccore$ is called the \emph{centreless core} of~$\fg$.
\end{definition}

\begin{example}
If $\fg$ is a finite dimensional simple Lie algebra (resp.~an affine Kac-Moody
Lie algebra) then by Proposition \ref{prop:lownullity}  we have
$\fg_\ccore \simeq \fg$ (resp.~$\fg_\ccore = \fgb$).
\end{example}

\begin{pgraph}
\label{pgraph:Lietori}
\headtt{Lie tori}
In  order to give an axiomatic description of the centreless cores of EALAs,
Yoshii introduced Lie tori in \cite{Y2}.  To describe these, let
$\Lm$ be a finitely generated free abelian group and let
$\Dl$ be an irreducible finite root system  with root lattice $Q(\Dl)$,

A \emph{Lie $(\Lm,\Delta)$-torus} is a Lie algebra $\cL$ with  compatible gradings by
$\Lm$ and $Q(\Dl)$ so that the $Q(\Dl)$-support of $\cL$ is contained in $\Dl$
and four natural axioms hold.
We will not need to refer
directly to these axioms and instead direct the interested reader to \cite{Neh1}, \cite[\S 1]{ABFP2}
or (for an equivalent definition) \cite[\S 5.1]{Neh3}.

If $\cL$ is a Lie $(\Lm,\Delta)$-torus,
we say that $\cL$ has \emph{full root support}
if the $Q(\Dl)$-support of $\cL$ equals $\Dl$. There is no loss of generality
in assuming this when convenient, since it always holds
if we replace $\Dl$ by a suitable sub-root-system of $\Dl$.  (See \cite[Remark 1.1.11]{ABFP2}
for more about this.)

If $\cL$ is a centreless Lie $(\Lm,\Delta)$-torus, we say that $\cL$ is \emph{invariant}
if there exists a nondegenerate invariant graded symmetric bilinear form on $\cL$.
(This is equivalent to the definition in \cite[\S 5.1]{Neh3} in view of \cite[Prop.~1.2.2(vi)]{ABFP2}.)
\end{pgraph}

We have the following relationship between EALAs and centreless Lie tori.  (See
\cite{Neh2} or \cite[\S 6]{Neh3}.  For discrete EALAs,
see \cite[\S 1]{AG} and \cite{Y2}.)

\begin{proposition}\
\label{prop:LTEALA} Let $\Lm$ be a free abelian group of rank $n$, let
$\Delta$ be an irreducible finite root system of type $X_k$.
Then, a Lie algebra $\cL$ is isomorphic to the centreless core
of an EALA of nullity $n$ and  quotient type $X_k$ if and only if
$\cL$ is isomorphic to an invariant centreless Lie $(\Lm,\Delta)$-torus with full root support.
\end{proposition}

\begin{remark}
\label{rem:LTYishii}  Suppose that $\base = \bbC$. In \cite{Y2}, Yoshii
showed, using some classification results for centreless Lie tori,
that every centreless Lie torus over $\bbC$ is invariant. Also, it follows from
\cite{Y2} that the term ``EALA'' can be replaced by the term ``discrete EALA''
in Proposition \ref{prop:LTEALA}.  Hence,
centreless cores of EALAs are the same algebras as centerless cores of discrete
EALAs.  We will not make use of these facts in this article.
\end{remark}

\begin{pgraph}
\label{pgraph:coordinate} Centreless Lie tori
and hence centreless cores of EALAs have been characterized as certain
``matrix algebras''
over (in general) infinite dimensional graded coordinate algebras using theorems that are called
\emph{coordinatization theorems}.  There is one such theorem for each
 quotient type; the reader is referred to \cite{AF} for a overview of this work
by many authors.  This approach provides a wealth of detailed information
about centreless cores, but we will not use it in this work,
except to calculate some indices in
Section  \ref{subsec:TitsM2}.
\end{pgraph}

\subsection{The class $\bbE_n$}\
\label{subsec:En}

\begin{definition}
\label{def:En}  If $n$ is a nonnegative integer,  let
$\bbE_n$ be the class of all
Lie algebras that are isomorphic to the centreless  core of an
EALA of nullity $n$.  We call algebras in $\bbE_n$ \emph{nullity $n$
centreless cores}.
\end{definition}

\begin{pgraph} Note that Proposition \ref{prop:LTEALA} gives a characterization
of the algebras in $\bbE_n$ in terms of centreless Lie tori.
\end{pgraph}

For each $n \ge 0$, we have now defined three classes
$\bbM_n$, $\bbI_n$ and $\bbE_n$ of Lie algebras.
For $n=0$, we have the following:

\begin{proposition} \label{prop:nullity0} $\bbM_0 = \bbI_0 = \bbE_0$ is the class of finite dimensional simple Lie algebras.
\end{proposition}

\begin{proof}  All that must be proved is that $\bbE_0$ is the class of finite dimensional   simple
Lie algebras, and this follows from Proposition \ref{prop:lownullity}(a).
\end{proof}

\subsection{Fgc algebras in $\bbE_n$}\
\label{subsec:fgcEn}

We have the following from \cite{ABFP2}  and \cite{A}.

\begin{theorem}
\label{thm:LT}  Suppose that $\cL$ is an fgc centreless Lie $(\Lm,\Delta)$-torus with full root support,
where $\Lm$ has rank $n$ and $\Delta$ has type $X_k$.  Then,
$\cL\in \bbM_n$ and the relative type (as defined in Definition \ref{def:reltype2}) of $\cL$ is $X_k$.
\end{theorem}

\begin{proof}  The fact that $\cL\in \bbM_n$ is proved in Theorem 3.3.1 of
\cite{ABFP2}, and the fact that $\cL$ has relative type $X_k$ is proved in \cite{A}.
\end{proof}

\begin{corollary}
\label{cor:EALAmult} \
\begin{itemize}
\item[(a)]  If $\fg$ is an EALA of nullity $n$   such that $\fg_\ccore$ is fgc, then $\fg_\ccore\in \bbM_n$,
$\fg_\ccore$ is isotropic, and
the relative type of $\fg_\ccore$
is the    quotient type of $\fg$.
\item[(b)]  $\bbM_n \cap \bbE_n$ is the class of fgc algebras in $\bbE_n$.
\end{itemize}
\end{corollary}

\begin{proof}   (a)  Let $X_k$ be the   quotient type of $\fg$ and let $\cL =\fg_\ccore$.
By Proposition \ref{prop:LTEALA}   and Theorem \ref{thm:LT}, we know that $\cL\in \bbM_n$ and that the relative type of $\cL$ is $X_k$. Finally, since $k > 0$, it follows that $\cL$
is isotropic.

(b) follows from (a) and Proposition \ref{prop:ppfgc}.
\end{proof}

\section{Loop algebras of symmetrizable Kac-Moody Lie algebras}
\label{sec:autKM}

To obtain our main results, we will need to understand loop algebras
of finite dimensional simple Lie algebras and loop algebras of affine Lie algebras.
We  consider  in this section the more general topic of loop algebras of symmetrizable Kac-Moody Lie algebras.

We assume throughout the section that \emph{$\fg = \fg(A)$
is the Kac-Moody Lie algebra determined by an indecomposable symmetrizable
GCM}  $A = (a_{ij})_{i,j \in \ttI}$.  We use the notation
of Section \ref{subsec:symm}.

\subsection{Automorphisms of symmetrizable Kac-Moody Lie algebras} \
\label{subsec:autKM}

We now recall some facts about automorphisms of $\fg$, $\fgp$ and $\fgb$.

\begin{pgraph}
\label{pgraph:autsubgp}  We begin by discussing some subgroups of $\Aut(\fg)$.

First let $\Aut(A)$ be the group of automorphisms of the GCM $A$;  that is
$\Aut(A)$ is the group of permutations $\sg$ of $\ttI$ so that
$a_{\sg(i),\sg(j)} = a_{i,j}$ for $i,j\in\ttI$.
By \cite[\S 4.19]{KW}
there exists a group monomorphism $\sg \mapsto \tsg$ of $\Aut(A)$ into
$\Aut(\fg)$
such that $\tsg(\fh) = \fh$ and
$$\tsg(e_i) = e_{\sg(i)} \andd \tsg(f_i) = f_{\sg(i)},\quad$$
for $\sg\in \Aut(A)$ and $i\in \ttI$.
It is clear that $\sg \mapsto \tsg$ is unique when $A$ has  finite type,
and this is also true when $A$ is affine (see  Proposition \ref{prop:diagaff}  below).
In any case, we fix a choice of the monomorphism $\sg \mapsto \tsg$
and \emph{we use this map to identify $\Aut(A)$
as a subgroup of $\Aut(\fg)$.}

Next  let $\chev$ be the unique automorphism
of $\fg$  (the \emph{Chevalley automorphism} of $\fg$) such that
\[\chev(e_i) = -f_i,\quad \chev(f_i) = -e_i\andd \chev(h) = -h\]
for $i\in \ttI$ and $h\in \fh$.  Now $\chev$ has order 2 and it commutes with the
automorphisms in  $\Aut(A)$, so we may define the \emph{outer automorphism group of~$\fg$} by
\[\Out(A) :=
\left\{
  \begin{array}{ll}
   \Aut(A), & \hbox{if $A$ has finite type;} \\
    \langle \chev \rangle \times \Aut(A), & \hbox{otherwise.}
  \end{array}
\right.\]
Then $\Out(A)$ is a finite subgroup of $\Aut(\fg)$.

Next let $\Aute(\fg) = \langle \exp(\ad(x)) \suchthat x\in\fg_\al,\ \al\in\Dl^\real\rangle$
in $\Aut(\fg)$. We also have subgroups
$\Aut(\fg;\fh) := \set{\sg\in\Aut(\fg) \suchthat \sg(h) = h \text{ for } h\in \fh}$  and
$\Aut(\fg;\fgp) = \set{\sg\in\Aut(\fg) \suchthat \sg(x) = x \text{ for } x\in \fgp}$
of $\Aut(\fg)$.  Then
$\Aut(\fg;\fh)$ normalizes $\Aute(\fg)$, while $\Aut(\fg;\fgp)$ centralizes
both $\Aute(\fg)$ and $\Aut(\fg;\fh)$.  Hence,
$$\Autz(\fg) := \Aute(\fg) \Aut(\fg;\fh) \Aut(\fg;\fgp)$$
is a subgroup of $\Aut(\fg)$, which we call the \emph{inner automorphism group of $\fg$}.
We  will recall below in Proposition \ref{prop:Aut}
that $\Aut(\fg) = \Autz(\fg) \rtimes \Out(A)$.
If $\fg$ is finite dimensional, then $\Autz(\fg) $ coincides with
${\bf G}_{\rm ad}(k)$ where ${\bf G}_{\rm ad}$ is the group of adjoint type corresponding to $\fg$.
Furthermore, in that case,
${\bf G}_{\rm ad}$ is the connected component of the identity of
the algebraic  group $\textbf{Aut}(\fg).$
\end{pgraph}

\begin{pgraph}
\label{pgraph:identOut}
We have group homomorphisms
\begin{equation}
\label{map1}
\chi_1 : \Aut(\fg) \mapsto \Aut(\fgp) \andd \chi_2: \Aut(\fgp) \mapsto \Aut(\fgb),
\end{equation}
where $\chi_1(\upsilon) = \upsilon |_{\fgp}$ for $\upsilon\in \Aut(\fg)$
and $\chi_2(\tau) = \bar\tau$ for $\tau\in \Aut(\fgp)$,  and where
$\bar\tau$ denotes the automorphism  induced by $\tau$ on
$\fgb = \fgp/Z(\fgp)$.   Let
\[\Autz(\fgp) = \chi_1(\Autz(\fg)) \andd \Autz(\fgb) = (\chi_2\circ\chi_1) (\Autz(\fg)).\]
Also,  it  is easy to see that $\chi_1$ and $\chi_2\circ\chi_1$ restricted to $\Out(A)$ are injective, and  \emph{we
use these maps to identify $\Out(A)$  with a subgroup of $\Aut(\fgp)$ and  $\Aut(\fgb)$
respectively.}  In particular, we are regarding $\Aut(A)$ as a subgroup of
$\Aut(\fg)$, $\Aut(\fgp)$ and  $\Aut(\fgb)$.
\end{pgraph}

\begin{definition}
\label{def:diagsymm}
We  use the term \emph{diagram automorphism} to refer to an
automorphism of $\fg$, $\fgp$ or $\fgb$ that lies in
$\Aut(A)$.
\end{definition}

\begin{pgraph}
There is a unique action of $\Out(A)$ on $Q$ such that
$\nu(\al_i) = \al_{\nu(i)}$ and   $\chev(\al_i) = -\al_i$ for all $\nu\in \Aut(A)$, and $i\in \ttI$.
Then, with the identifications of \pref{pgraph:identOut}, we have
\[\sg(\fg_\al) = \fg_{\sg(\al)},\quad \sg((\fgp)_\al) =(\fgp)_{\sg(\al)} \andd
\sg((\fgb)_\al) =(\fgb)_{\sg(\al)} \]
for $\sg\in \Out(A)$ and $\al\in Q$.
\end{pgraph}

The following result on
the structure of $\Aut(\fg)$, $\Aut(\fgp)$ and $\Aut(\fgb)$ is due to Peterson and
Kac~\cite{PK} (see also \cite[Prop.~7.3]{ABP2}).

\begin{proposition}
\label{prop:Aut} $\chi_1$ is surjective with kernel $\Aut(\fg;\fgp)$, and
$\chi_2$ is an isomorphism.  Furthermore
\begin{equation}
\label{eq:Autdecomp}
\begin{gathered}
\Aut(\fg) = \Autz(\fg) \rtimes \Out(A), \\
\Aut(\fgp)= \Autz(\fgp) \rtimes \Out(A), \\
\Aut(\fgb)  = \Autz(\fgb) \rtimes \Out(A).
\end{gathered}
\end{equation}
\end{proposition}

\begin{pgraph}
\label{pgraph:project}
Next we let $p : \Aut(\fg) \to \Out(A)$, $p' : \Aut(\fgp) \to \Out(A)$ and
$\bar p: \Aut(\fgb) \to \Out(A)$
be the projections onto the second factor relative to the decompositions in \eqref{eq:Autdecomp}.
Then we have the commutative diagram:
\begin{equation}
\label{diagram:Aut}
\begin{CD}
\Aut(\fg) @>p>> \Out(A)\\
@V{\chi_1}VV @|\\
\Aut(\fgp) @>{p'}>> \Out(A)\\
@V{\chi_2}VV @|\\
\Aut(\fgb) @>{\bar p}>> \Out(A)\\
\end{CD}
\end{equation}
\end{pgraph}

\begin{definition}
\label{def:kind}
Suppose that $\sg$ is an automorphism
of $\fg$, $\fgp$ or $\fgb$.  We say that
$\sg$ is of \emph{first kind}
if $p(\sg) \in \Aut(A)$, $p'(\sg) \in \Aut(A)$ or $\bar p(\sg) \in \Aut(A)$
respectively.  Otherwise, we say that
$\sg$ is of
\emph{second kind}.\footnote{Although we will not need this fact, one can show that this definition of first and second
kind agrees with the usual one which is defined using the image of the positive Borel subalgebra
under $\sg$ \cite[\S 4.6]{KW}.}
\end{definition}

\begin{proposition}
\label{prop:section} The group homomorphism $\chi_1$ has a section (namely  a homomorphism that is a right inverse
of $\chi_1$).
\end{proposition}

\begin{proof}  We  use the notation and identifications of
Sections \ref{subsec:symm} and \ref{subsec:autKM}.
For convenience we set
\[G = \Aute(\fg),\quad \tH = \Aut(\fg;\fh)\andd K = \Aut(\fg;\fgp).\]
By Proposition \ref{prop:Aut}, $\chi_1$ is an epimorphism with kernel $K$ and
\begin{equation*}
\Aut(\fg) = \Autz(\fg) \rtimes \Out(A) = (G\tH K) \rtimes \Out(A).
\end{equation*}
Note also that $\tH$ normalizes $G$, so $G\tH$ is a subgroup of $\Aut(\fg)$.
Further $\Out(A)$ normalizes $G\tH$, so $(G\tH) \rtimes \Out(A)$ is a subgroup of $\Aut(\fg)$.
So $\chi_1|_{(G\tH) \rtimes \Out(A)} : (G\tH) \rtimes \Out(A) \to  \Aut(\fgp)$
is an epimorphism with kernel $\left((G\tH) \rtimes \Out(A)\right) \cap K$.
Thus it suffices to show that $\left((G\tH) \rtimes \Out(A)\right) \cap K = \set{1}$.  Hence
it suffices to show that
\begin{equation}
\label{eq:section}
(G\tH) \cap K = \set{1}.
\end{equation}

Let $\Dl_+$ be the set of positive roots of $\fg$;
let $U_+$ denote the subgroup of $G$ generated by the
automorphisms of the form
$\exp(\ad(x))$, where $x\in \fg_\al$, $\al\in \Dl^\real \cap \Dl_+$;
and let $\tilde B_+ = U_+ \tH$.
Then $\tilde B_+$ is a subgroup of $G\tH$ and
\begin{equation}
\label{eq:Bruhat}
G\tH = \cup_{w\in W} \tilde B_+ n_w \tilde B_+,
\end{equation}
where the family $\set{n_w}_{w\in W}$ of elements
of $G$ satisfies $n_1 = 1$, $n_w(\fh) = \fh$ and
$n_w|_{\fh} = w$ for $w\in W$.  (Here we are identifying $W$ as a subgroup of $\GL(\fh)$ as
usual \cite[Lemma 5.1.2]{MP}.)
Indeed, this follows from the Bruhat decomposition for the derived
group of $\fg$, since that group is mapped onto $G$ by the adjoint map \cite[Prop.~6.3.7]{MP}.

Observe also from the definition of $\tilde B_+$ that
\begin{equation}
\label{eq:B}
b\in \tilde B_+ \implies b(\fnp) \subseteq \fnp \text{ and } b(h) \in h+\fnp \text{ for } h\in \fh,
\end{equation}
where $\fnp  = \sum_{\al\in\Dl_+} \fg_\al$.

To show \eqref{eq:section}, suppose that $p\in (G\tH)\cap K$.  Then, by \eqref{eq:Bruhat}, we have
\[p = b_1 n_w b_2,\]
where $b_1, b_2\in \tilde B_+$ and $w\in W$.
Thus $n_w = b_1^{-1}pb_2^{-1}$.  So for $h\in \fh'$,
we have $n_w(h) \in  h + \fnp$ by
\eqref{eq:B}.  Thus $w|_{\fh'} = 1$,
and hence, by \cite[Cor.~5.2.1]{MP},   $w=1$. Therefore,
$p\in \tilde B_+\cap K$.
Now it easy to see that the elements of $K$
stabilize $\fh$  (see for example \cite[Prop.~7.5(a)]{ABP2}).
In particular, $p$ stabilizes $\fh$. So, since $p\in \tilde B_+$,
it follows from \eqref{eq:B} that $p\mid_\fh = 1$.
Since  $\fg = \fgp + \fh$, we have $p = 1$ as   desired.
\end{proof}

\subsection{Loop algebras of symmetrizable Kac-Moody Lie algebras} \
\label{subsec:loopKM}

\begin{pgraph}
If $\mu_1, \mu_2 \in \Out(A)$, we write $\mu_1\sim \mu_2$ to mean that
$\mu_1$ is conjugate to $\mu_2$ or $\mu_2^{-1}$ in $\Out(A)$.
\end{pgraph}

Using
the results of \cite{ABP2}, we can prove the following:

\begin{theorem} \label{thm:erase}
Suppose that $\fg = \fg(A)$
is the Kac-Moody Lie algebra determined by an indecomposable symmetrizable
GCM $A$.
\begin{itemize}
\item[(a)]  If  $\upsilon_1$ and $\upsilon_2$
are finite order automorphisms of $\fg$, then
$\Lp(\fg,\upsilon_1)\isom \Lp(\fg,\upsilon_2)$ if and only if
$p(\upsilon_1) \sim p(\upsilon_2)$.
\item[(b)]  If  $\tau_1$ and $\tau_2$
are finite order automorphisms of $\fgp$, then
$\Lp(\fgp,\tau_1)\isom \Lp(\fgp,\tau_2)$ if and only if $p'(\tau_1)\sim p'(\tau_2)$.
\item[(c)]  If  $\sg_1$ and $\sg_2$
are finite order automorphisms of $\fgb$, then $\bar p(\sg_1)\sim \bar p(\sg_2)$ implies that
$\Lp(\fgb,\sg_1)\isom \Lp(\fgb,\sg_2)$.
\end{itemize}
\end{theorem}

\begin{proof} (a): This is Theorem 9.3 of \cite{ABP2}.

(b):  This is  stated in \cite{ABP2} as a corollary of (a). For the  sake of completeness (since the proof is not entirely obvious) we include the necessary details here. By  Proposition \ref{prop:section}, we can choose $\upsilon_i \in \Aut(\fg)$ of finite order
so that $\chi_i(\upsilon_i) = \tau_i$ for $i=1,2$.  Then
$\Lp(\fgp,\tau_1)\isom \Lp(\fgp,\tau_2)$ if and only if $\Lp(\fg,\upsilon_1)\isom \Lp(\fg,\upsilon_2)$
(by \cite[Theorem 8.6]{ABP2}) which holds if and only if
$p(\upsilon_1) \sim p(\upsilon_2)$ (by (a)).  But, by \eqref{diagram:Aut}, $p'(\tau_i) = p(\upsilon_i)$
for $i=1,2$. So we have (b).

(c): Suppose that $\bar p(\sg_1)\sim \bar p(\sg_2)$. Since $\chi_1$ has a section and
$\chi_2$ is an isomorphism,
we can choose $\upsilon_i$ of finite order in $\Aut(\fg)$
such that $(\chi_2\circ \chi_1)(\upsilon_i) = \sg_i$ for $i=1,2$.
Then, by \eqref{diagram:Aut}, we have $p(\upsilon_1)\sim  p(\upsilon_2)$, so
$\Lp(\fg,\upsilon_1)\isom \Lp(\fg,\upsilon_2)$ by (a). Thus,
$\Lp(\fg,\upsilon_1)'/\Centre(\Lp(\fg,\upsilon_1)') \isom \Lp(\fg,\upsilon_2)'/\Centre(\Lp(\fg,\upsilon_2)')$.
So by Lemma \ref{lem:loopderived}, we have $\Lp(\fgb,\sg_1)\isom \Lp(\fgb,\sg_2)$.
\end{proof}

\begin{remark}\

(a) If $\fg$ is finite dimensional simple, then $\fgp = \fg$,
$\Centre(\fgp) = 0$, $\fgb \simeq \fg$, $\Out(A) = \Aut(A)$,
and the elements of $\Aut(A)$ are the classical diagram automorphisms.

(b) If $\fg$ is finite dimensional simple  and $\base = \bbC$,  Theorem \ref{thm:erase} is a result of Kac.
(See \cite{P1} and \cite{P2}  for other generalizations
of Kac's result.)

(c) If $\fg$ is affine and $\sg_1$ and $\sg_2$  are of first kind,
we will see in Corollary \ref{cor:class1} that
the converse in Theorem \ref{thm:erase}(c) is also true.
\end{remark}

\subsection[The class $\bbM_1$]{Loop algebras of finite dimensional simple algebras and the class~$\bbM_1$}
\label{subsec:Kacreal} \

The  following realization theorem for affine algebras is due to V.~Kac \cite[Lemma~22]{K1}.
In the untwisted case (see Corollary \ref{cor:KacReal0}), it was proved independently by R.~Moody \cite[Thm.~2]{Mo}.
A detailed proof is given in \cite[\S 7.4 and~\S 8.3]{K2} in the complex case, and that proof also works in  general.

\begin{theorem}[Kac]
\label{thm:Kacreal}\

\emph{(a)}
If $\gd$ is a finite dimensional simple Lie algebra  of type $\type{X}{\rkfin}$
and $\sgd$ is a diagram automorphism of
$\gd$ of order $m$, then $\Lp(\gd,\sgd) \simeq \fgb$, where
$\fg = \fg(A)$ is the Kac-Moody Lie algebra constructed from the
affine GCM $A$ of type  $\typeaff{X}{\rkfin}{m}$.

\emph{(b)} Conversely, if  $\fg = \fg(A)$,  where $A$ is the
affine GCM of type $\typeaff{X}{\rkfin}{m}$, then
there is an isomorphism $\varphi$ from $\fgb$ onto a loop algebra $\Lp_m(\gd,\sgd)$ contained
in $\gd\ot S_1$, where $\gd$ is a finite dimensional simple Lie algebra  of type $\type{X}{\rkfin}$
and $\sgd$ is a diagram automorphism of
$\gd$ of order $m$.  Moreover, this isomorphism can be chosen such that
\[t_1 \varphi((\fgb)_\al) = \varphi((\fgb)_{\al + m\delta})\]
for $\al\in Q$, where  $t_1 = z_1^m$ in  $R_1 = \base[t_1^{\pm 1}]$  (see \pref{pgraph:identRS}),
and $\delta$ is the standard null root  (see \eqref{eq:Xaffdiag2} below).
\end{theorem}

\begin{corollary}
\label{cor:KacReal0} A   Lie algebra $\cL$ is isomorphic to $\fgb$ for some untwisted affine Kac-Moody Lie
algebra if and only if $\cL \simeq \Lp(\gd,1)$ for some finite dimensional simple Lie algebra $\gd$.\end{corollary}

In the  complex case, the equivalence of (a), (b), (c) and (d) in the following corollary is due to Kac.

\begin{corollary}
\label{cor:KacReal1}
If  $\cL$ is a Lie algebra, the following statements are equivalent:
\begin{itemize}
\item[(a)] $\cL\in \bbM_1$.
\item[(b)]  $\cL \simeq \Lp(\gd,\sgd)$  for some finite dimensional simple Lie algebra
$\gd$ and some diagram automorphism $\sgd$ of $\gd$.
\item[(c)] $\cL \simeq \fgb$ for some affine Kac-Moody Lie algebra $\fg$.
\item[(d)] $\cL\in \bbI_1$.
\item[(e)] $\cL\in \bbE_1$.
\end{itemize}
\end{corollary}

\begin{proof}  First of all ``(a)$\Rightarrow$(b)'' follows from  Theorem \ref{thm:erase}(a) (applied to $\gd$);
while ``(b)$\Rightarrow$(a)'' is trivial.
Next (b) and (c) are equivalent by Theorem \ref{thm:Kacreal};
(a) and (d) are equivalent by definition;
and (c) and (e) are equivalent by Proposition    \ref{prop:lownullity}(b).
\end{proof}

\begin{corollary}
\label{cor:KacReal2} Suppose that  $\fg = \fg(A)$, where $A$ is the
affine GCM of type $\typeaff{X}{\rkfin}{m}$.
Then
$\fgb$ is prime, perfect and fgc; and the absolute type and
relative type of $\fgb$  are respectively $\type{X}{\rkfin}$
and the   quotient type of $\fg$
(see \pref{pgraph:finitetype}).
\end{corollary}

\begin{proof}
This follows from Theorem \ref{thm:Kacreal}(b), Propositions \ref{prop:ppfgc} and \ref{prop:absolute},
and Theorem \ref{cor:EALAmult}(a).
\end{proof}

\begin{pgraph}
\label{pgraph:quotientrs}
For each affine GCM $A$,  the quotient type of $\fg = \fg(A)$ is  calculated in \cite[Prop.~6.3]{K2}.
We display the results  in
Table \ref{tab:fq} below,  where $\typeaff{X}{k}{1}$ denotes any one
of the untwisted  affine types.  By Corollary \ref{cor:KacReal2},
this table also displays the relative type of $\fgb$.
\begin{table}[ht]
\renewcommand{\arraystretch}{1.5}
\begin{tabular}
[c]{|c |   c |} \hline
$A$ & \quad Quotient type of $\fg(A)$\quad \\
\whline
$\typeaff{X}{k}{1}$& $\type{X}{k}$\\
\hline
$\typeaff{A}{\rkfin}{2}$,\ $\rkfin\ge 2, \rkfin \ne 3$ & $\type{BC}{\frac \rkfin 2}$ ($\rkfin $ even) or $\type{C}{\frac {\rkfin +1}2}$ ($\rkfin$ odd) \\
\hline
$\typeaff{D}{\rkfin}{2}$,\ $\rkfin \ge 3$ & $\type{B}{\rkfin-1}$\\
\hline
$\typeaff{D}{4}{3}$ & $\type{G}{2}$\\
\hline
$\typeaff{E}{6}{2}$ & $\type{F}{4}$\\
\hline
\end{tabular}
\medskip
\caption{}
\label{tab:fq}
\end{table}
\end{pgraph}

\begin{remark}
\label{rem:labels} Suppose that $\fg = \fg(A)$ is affine.
If $A$ has type $\typeaff{X}{k}{m}$, using the labels from \cite{K2}  as above,
then  by Corollary \ref{cor:KacReal2} the absolute type of $\fgb$ can be read from
the label as $\type{X}{k}$.
In \cite[\S 3.5]{MP} a different system is used  to label
affine matrices  which allows one
to read the quotient type of $\fg$, and hence also
the relative type of $\fgb$, from the label in the same fashion.
For example if $A$ has label $\typeaff{D}{4}{3}$ as above,
the label for $A$ in \cite{MP} is $\typeaff{G}{2}{3}$.
Thus both systems
convey important information about the Lie algebra $\fgb$.  We
will continue to use the labels from \cite{K2} as discussed in  Section~\ref{subsec:symm}.
\end{remark}

The  realization theorem leads to a classification of the algebras in $\bbM_1$
using the description of these algebras in  Corollary \ref{cor:KacReal1}(b).

\begin{corollary}
\label{cor:KacReal3}\

\emph{(a)} If $\sgd_i$ is a diagram automorphism of
a finite dimensional simple Lie algebra $\gd_i$, $i=1,2$, then
$\Lp(\gd_1,\sgd_1) \isom\Lp(\gd_2,\sgd_2)$ implies that $\gd_1 \simeq \gd_2$.

\emph{(b)}  If $\sgd_i$ is a diagram automorphism of
a finite dimensional simple Lie algebra $\gd$, $i=1,2$, then
$\Lp(\gd,\sgd_1) \isom\Lp(\gd,\sgd_2)$ if and only if
$\sgd_1$ and $\sgd_2$ are conjugate in the group of diagram automorphisms of $\gd$.
\end{corollary}

\begin{proof} This follows from the Peterson-Kac conjugacy theorem
\cite[Thm.~2]{PK} and Theorem
\ref{thm:Kacreal}(a).
(See also Remark 8.13 of \cite{ABP2}.) Another  proof using Galois  cohomology is given in \cite{P2}.
Alternatively,  (a) follows from  Proposition \ref{prop:absolute};
and (b) follows from Theorem \ref{thm:erase}(a)  (since every element of
$\Aut(\fg)$ is conjugate to its inverse).
\end{proof}

\begin{corollary}
\label{cor:KacReal4} An algebra $\cL$ is in $\bbI_2$ if and only if
$\cL \simeq \Lp(\fgb,\sg)$ for some affine Lie algebra $\fg = \fg(A)$
and some finite order $\sg\in\Aut(\fg)$.  Moreover, in that
case, if $A$ has type $\typeaff{X}{\rkfin}{m}$, then $\cL$ has absolute type
$\type{X}{\rkfin}$.
\end{corollary}

\begin{proof} The first statement follows from the definition of
$\bbI_2$ and Corollary \ref{cor:KacReal1}; and the second statement
follows from Corollary \ref{cor:KacReal2} and Proposition \ref{prop:loopperm}.
\end{proof}

\section{Automorphisms of affine algebras}
\label{sec:autaff}

In  Section \ref{subsec:autKM}, we discussed some general properties of
automorphisms of symmetrizable Kac-Moody Lie algebras.  In this section,
we obtain some more detailed information in the affine case.

We assume throughout the section that \emph{$\fg = \fg(A)$ is the
Kac-Moody Lie algebra determined by an affine GCM $A = (a_{ij})_{i,j\in\ttI}$},
where $\ttI = \set{0,\dots,\rkaff}$ with $\rkaff\ge 1$.

\subsection{Notation and terminology for affine algebras}\
\label{subsec:notationaff}

Before looking at automorphisms, we collect some notation, recall some
terminology and discuss some basic properties of  affine algebras.
We will use  this material throughout the rest of this paper.

\begin{pgraph}
\label{pgraph:notationaff}
In addition to the notation of
Section \ref{subsec:symm},
we use the following  notation that is standard in the affine case \cite{K2}.
Choose relatively prime positive integers $a_0,\dots,a_\rkaff$
and relatively prime positive integers $a_0\ck,\dots,a_\rkaff\ck$ so that
$A [ a_0,\dots,a_\rkaff]^t = 0$ and $[a_0\ck, \dots, a_\rkaff\ck] A = 0.$
Note that if $\nu\in \Aut(A)$, then one has
\begin{equation}
\label{eq:Xaffdiag1}
a_{\nu(i)} = a_i \andd a_{\nu(i)}\ck = a_i\ck
\end{equation}
for $i\in \ttI$ \cite[p.~47]{FSS}.
Let
\begin{equation}
\label{eq:Xaffdiag2}
c = \textstyle \sum_{i\in\ttI} a_i\ck \al_i\ck \andd \delta = \sum_{i\in\ttI} a_i \al_i,
\end{equation}
in which case $\Centre(\fgp)= \base c$ and $\delta$ is a $\bbZ$-basis for the lattice of isotropic
roots of $\fg$. Let $d\in \fh$ be a scaling element satisfying
$\al_i(d) = \delta_{i,0}$ for $i\in\ttI$,  in which case
\[\fg = \fgp \oplus \base d \andd \fh = \fh' \oplus \base d.\]

Finally let $\form$ be the normalized invariant form on $\fg$ \cite[\S 6.2]{K2};  that is $\form$
is the unique nondegenerate symmetric bilinear form on $\fg$ satisfying
\[(\al_i\ck | \al_j\ck) = a_j a_j\ck a_{ij},\quad (\al_i\ck | d) =  \delta_{i,0}a_0 \andd (d | d) = 0\]
for $i,j\in \ttI$ (see \cite[\S 6.1 and 6.2]{K2} where $c$ is denoted by $K$).
Note in particular that
\begin{equation}
\label{eq:formef}
(e_i| f_i) = \frac{a_i}{a_i\ck} \quad \text{ for } i\in \ttI.
\end{equation}
We also have $(\al_i | \al_i) \ne 0$ and
\[a_{ij}  = \frac {2(\al_i | \al_j)}{(\al_i | \al_i)}\]
for $i,j\in \ttI$ \cite[\S 2.3]{K2}.

Also, as for general EALAs in \pref{pgraph:finitetype},  we let
\[V = \spann_\bbQ(\Dl).\]
Then $\set{\al}_{i\in\ttI}$
is a $\bbQ$-basis for $V$.
Note that the form $\form$ restricted to $V$ takes rational values \cite[\S 6.2--6.3]{K2}
and we denote its restriction to $V$ also by $\form$.
Then,
$\form : V \times V \to \bbQ$ is  positive
semi-definite  with radical
\[V^0 := \rad(V) = \bbQ \delta.\]
We let  $\bV = V/V^0$
with canonical map $^{-} : V \to \bV$,  in which case we have
the induced positive definite form
$\form : \bV \times \bV \to \bbQ$ on $\bV$.

If $S\subseteq V$, we let
\[S^\times = \set{\al\in S \suchthat (\al\vert\al) \ne 0}.\]
Then \cite[Prop.~5.10(c)]{K2}
\begin{equation}
\label{eq:realaff}
\Dl^\times = \Dl^\real = \cup_{i\in \ttI} W \al_i.
\end{equation}
\end{pgraph}

\subsection{Diagram automorphisms of affine algebras}\
\label{subsec:diagaff}

We now recall from \cite{Bau}
that the embedding
of $\Aut(A)$ into $\Aut(\fg)$ discussed in \pref{pgraph:autsubgp}
is unique in the affine case.  This gives a unique interpretation of the notion
of diagram automorphism.

\begin{proposition}
\label{prop:diagaff}
There exists a unique group homomorphism $\sg \mapsto \tsg$ of
$\Aut(A)$ into $\Aut(\fg)$
such that $\tsg(\fh) = \fh$,
\begin{equation}
\label{eq:diagaff1}
\tsg(e_i) = e_{\sg(i)} \andd \tsg(f_i) = f_{\sg(i)}
\end{equation}
for $\sg\in \Aut(A)$ and $i\in \ttI$.  This
homomorphism  is injective.
Finally, the normalized invariant form $\form$ on $\fg$
is invariant under $\tsg$ for $\sg\in \Aut(A)$.
\end{proposition}

\begin{proof}  We know that a homomorphism with the property indicated in the first sentence exists.
Also,
uniqueness follows from Lemma 2.2(b)  of \cite{Bau}. Furthermore, it is shown in the proof of that
lemma that $\tsg(d) \in d + \fh'$ for $\sg \in \Aut(A)$.

Next it is clear that the map $\sg \mapsto \tsg$ is injective, so all that remains
is to prove the last statement.  For this, let $\sg\in \Aut(A)$, and
define a new form $\form'$ on $\fg$ by
$(x \mid y)' = (\tsg x \mid  \tsg y)$.
Then, both $\form$ and $\form'$ are invariant nondegenerate symmetric
bilinear forms on $\fg$.  Let
\[\mathfrak m = \{ x\in \fg : (x \mid y)' = (x \mid y)
\text{ for all } y\in \fg\}.\]
Then $\mathfrak m$ is an ideal of $\fg$, which we must show is $\fg$.  First
$e_i\in\mathfrak m$ for all $i$. To see this, observe that $e_i$ is
orthogonal, using either form, to all root spaces of $\fg$ except
$\fg_{-\al_i}$.  Hence, it suffices to show that $(e_i \mid f_i)' = (e_i \mid f_i)$,
which follows from \eqref{eq:Xaffdiag1} and \eqref{eq:formef}.
So $e_i\in\mathfrak m$ and similarly
$f_i\in\mathfrak m$.  Thus, $\mathfrak m$ contains $\fgp$.
Finally, let $d' = \sum_{i=0}^{\order{\sg}-1} \tsg^i d$.
Then, since  $\tsg(d) \in d + \fh'$, we see that $d'\in \fh\setminus\fh'$.
Further, $\tsg d' = d'$ and so certainly
$(d' \mid d')' = (d' \mid d')$.
But since $\mathfrak m$ contains $\fgp$, we have $(x \mid d')' =
(x \mid d')$ for all $x\in \fgp$.
So  $d'\in \mathfrak m$, and we have $\mathfrak m = \fg$.
\end{proof}

\begin{pgraph}
\label{pgraph:diagaff}  Suppose that $\sg\in \Aut(A)$.
As in   \pref{pgraph:identOut} and \pref{def:diagsymm},  we  will, whenever convenient,
identify $\sg$ with the automorphism
$\tsg$, $\tsg|_\fgp$ or $\overline{\tsg|_\fgp}$ of $\fg$, $\fgp$ or $\fgb$ respectively.
\end{pgraph}

\subsection{The kind of a finite order automorphism of an affine algebra} \
\label{subsec:kindaff}

In the affine case, using Theorem \ref{thm:Kacreal}(b) and the results of \cite{ABP2.5}, we can characterize
the kind (as defined in Definition \ref{def:kind})  of a finite order automorphism of $\fgb$ in terms of the
centroid of the loop algebra $\Lp(\fgb,\sg)$.

\begin{proposition}
\label{prop:kind}
Suppose that $\sg$ is an automorphism
of finite order of $\fgb$.
Then,  $\sg$ is of first kind  if and only if
$\Cd(\Lp(\fgb,\sg)) \isom R_2$.
\end{proposition}

\begin{proof}
We  use the
isomorphism $\varphi$ in Theorem \ref{thm:Kacreal}(b)
to identify $\fgb = \Lp_{m_1}(\gd,\sgd)$ in $\gd\ot S_1$, and we then identify
$\Cd(\fgb) = R_1 = \base[t_1^{\pm 1}]$  as in Proposition \ref{prop:centloop},
where $\sgd$ has order $m_1$ and $t_1 = z_1^{m_1}$.
Then,  $\Cd(\sg) \in \Aut(R_1)$ and we have
\[\big(C(\sg)(t_1^j)\big)(x) = \sg(t_1^j (\sg^{-1} x))\]
for $x\in \fgb$, $j\in \bbZ$. Note that since $\Cd(\sg) \in \Aut(R_1)$, we have
\[\Cd(\sg)(t_1) \in \base^\times t_1^{\pm 1}.\]

Let $m_2$ be a period for $\sg$ and $\cL = \Lp_{m_2}(\fgb,\sg) \subseteq \fgb\ot \base[z_2^{\pm 1}]$.
(We are using $z_2$ for the Laurent variable in this second loop construction to avoid possible confusion.)
Then, by \cite[Prop.~4.11 and Prop.~4.13]{ABP2.5}, we have
$\Cd(\cL) \isom \Lp_{m_2}(\Cd(\fgb),\Cd(\sg)) = \Lp_{m_2}(R_1,\Cd(\sg))$, so
\[\Cd(\cL) \isom  \Lp_{m_2}( R_1,\Cd(\sg)) \subseteq R_1 \ot \base[z_2^{\pm 1}].\]
Therefore, by \cite[Lemma 9.1]{ABP2.5}, we see that $\Cd(\cL) \isom R_2$ if and only
if $t_1\ot z_2^j \in \Lp_{m_2}( R_1,\Cd(\sg))$ for some $j\in \bbZ$.  But this last condition
holds if and only if $\Cd(\sg)(t_1) \in \zeta_{m_2}^{-j} t_1$ for some $j\in \bbZ$.  So, since
$\Cd(\sg)$ has period ${m_2}$, we see that
$\Cd(\cL) \isom R_2$ if and only if $\Cd(\sg)(t_1) \in \base^\times t_1$.
Therefore, it remains to show that
\begin{equation}
\label{eq:condcent}
\bar p(\sg) \in \Aut(A) \iff \Cd(\sg)(t_1) \in \base^\times t_1.
\end{equation}

Now, by Proposition \ref{prop:Aut}, we can write
\[\sg = \mu\, \rho\, \nu\, \chev^i,\]
where $\mu$ is a product of automorphisms of $\fgb$ of the form
$\exp(\ad(x_\al))$ with $x_\al\in (\fgb)_\al$ and $\al\in \Dl^\real$;
$\rho\in (\chi_2\circ\chi_1)(\Aut(\fg,\fh))$; $\nu\in \Aut(A)$; and $i = 0$ or $1$.
Then, $\bar p(\sg) \in \Aut(A)$ if and only if $i=1$. So,
to prove \eqref{eq:condcent}, it suffices to show that
\[\Cd(\mu)(t_1) \in \base^\times t_1,\quad \Cd(\rho)(t_1) \in \base^\times t_1,\quad \Cd(\nu)(t_1) \in \base^\times t_1 \ \text{and} \
 \Cd(\chev)(t_1) \in \base^\times t_1^{-1}.
\]

First, $\mu$ commutes with the action of $\Cd(\fgb)$, and hence
$\Cd(\mu)(t_1) = t_1$. Next, if $\al\in \Dl$, we have
$\rho((\fgb)_\al) = (\fgb)_\al$.  Hence for $\al\in\Dl$,
\[\big(C(\rho)(t_1)\big)(\fgb)_\al = \rho (t_1(\fgb)_\al) = \rho ((\fgb)_{\al+m_1\delta}) = (\fgb)_{\al+m_1\delta}. \]
So $C(\rho)(t_1) \notin \base^\times t_1^{-1}$, and thus $C(\rho)(t_1) \in \base^\times t_1$.
Next $\nu(\delta) = \delta$ and
$\nu((\fgb)_\al) = (\fgb)_{\nu (\al)}$ for $\al\in\Dl$.  So calculating as above,
we get
$\big(C(\nu)(t_1)\big)(\fgb)_\al =(\fgb)_{\al + m_1\delta}$ for $\al\in \Dl$, and therefore
$C(\nu)(t_1) \in \base^\times t_1$.
Next $\chev(\delta) = -\delta$ and $\chev((\fgb)_\al) = (\fgb)_{-\al}$ for $\al\in\Dl$.
So, we get $\big(C(\chev)(t_1)\big) (\fgb)_\al =(\fgb)_{\al - m_1\delta}$ for $\al\in \Dl$, hence
$C(\chev)(t_1) \notin \base^\times t_1$, and therefore $C(\chev)(t_1) \in \base^\times t_1^{-1}$.
\end{proof}

\section{Loop algebras of affine algebras relative to diagram rotations} \
\label{sec:looprot}

Our goal in the next two sections is to obtain detailed information
about $\Lp(\fgb,\sg)$, where  $\fg = \fg(A)$ is an affine algebra and $\sg\in \Aut(A)$.

We say  that $\sg$ is \emph{transitive} if the group $\langle \sg\rangle $ acts transitively on $\ttI$.
Observe that by the classification  of affine diagrams, $\sg$ is transitive if and only if
$\fg$ is of type  $\typeaff{A}{\rkaff}{1}$, where $\rkaff\ge 1$, and $\sg$ has order $\rkaff+1$ in the rotation
group of the diagram.

We will consider two overlapping cases:
\begin{description}
\item[The rotation case] $\fg$ is of type  $\typeaff{A}{\rkaff}{1}$, where $\rkaff\ge 1$  and $\sg$ is in the rotation
group of the diagram. (When  $\rkaff$ = 1, $\sg = (1)$ or $(01)$.)
\item[The nontransitive case] $\sg$ is  not transitive.
\end{description}
Note that by the preceding observation every $\sg$ is covered by one of these  two cases.

In this section, we determine the structure of
$\Lp(\fgb,\sg)$ in the rotation case.  We will consider the nontransitive case in the next section.

\subsection{The map $\iota_m$} \

\begin{pgraph} \headtt{The map $\iota_m$} We next introduce a map $\iota_m : \bbZ_m \to \bbZ_m$
for $m\ge 1$, which  will appear in our description of the structure of $\Lp(\fgb,\sg)$ (Theorem \ref{thm:structurerot}).
To do this, we  begin by defining for each positive divisor $g$ of $m$ the set
\[U(m,g) = \set{\bar k \in \bbZ_m \suchthat \gcd(k,m) = g}.\]
(As is usual, we interpret $\gcd(0,m)$ as $m$.) So, in particular,
$U(m,m) = \set{\bar 0}$, and $U(m) := U(m,1)$ is the group of units of  the ring $\bbZ_m$.
(Our convention  is that $\bar 0$ is a unit in $\bbZ_1 = \set{\bar 0}$.)
Now
\[\bbZ_m = \textstyle \bigcup_{g|m} U(m,g) \ \ \text{(disjoint)};\]
and, for each $g\mid m$, we have the bijection
\[\textstyle U(\frac m g) \xrightarrow{{}\cdot g} U(m,g).\]
given by $\bar k \mapsto \overline{kg}$.  We let $\iota_m : \bbZ_m \to \bbZ_m$ be the unique map such that,
for each positive divisor $g$ of $m$, we have  $\iota_m(U(m,g)) \subseteq U(m,g)$ and the diagram
\[\begin{CD}
U(\frac m g) @>{{}\cdot g}>>  U(m,g) \\
@V{\text{inv}}VV @VV{\iota_m\mid_{U(m,g)}}V\\
U(\frac m g) @>{{}\cdot g}>>  U(m,g)
\end{CD} \]
commutes, where $\text{inv}$ is the inversion map on $U(\frac m g)$.  It is clear from this definition that
$\iota_m^2 = 1$  and hence  $\iota_m$ is a bijection.
\end{pgraph}

\subsection{Some associative loop algebras}\

Our goal now is to realize certain associative loop algebras as matrix algebras over
quantum tori.  We use the notation $d_m(u)$, $p_m(u)$, $D_m(u)$ and $P_m(u)$ of \pref{pgraph:Pauli}.

\begin{lemma}
\label{lem:assocloop1}
Suppose that $m\ge 1$, $\prim$ has order $m$ in $\base^\times$, and
$q\in \bbZ$.  Then, using the notation of Definition \ref{def:multiloop},
\begin{equation}
\label{eq:assocloop1}
\Lp_m\big(\Mat_m(R_1),P_m(t_1)^q,\prim\big) \simeq \Lp_{\boldm}\big(\Mat_m(\base),(D_m(\prim),P_m(1)^q),\boldprim\big),
\end{equation}
where $\boldm = (m,m)$ and $\boldprim = (\prim,\prim)$.
\end{lemma}

\begin{proof} We identify $\Mat_m(\base)\ot S_1 = \Mat_m(S_1)$, where $S_1 = \base[z_1^{\pm 1}]$.    Then
\begin{equation*}
\Lpm\big(\Mat_m\big(\base),D_m(\prim),\prim\big) = \Mat_m(S_1)^{D_m(\prim)\ot \eta}.
\end{equation*}
where $\eta\in \Aut(S_1)$ satisfies $\eta(z_1) = \prim^{-1} z_1$ and
$\Mat_m(S_1)^{D_m(\prim)\ot \eta}$ is the algebra of fixed points of $D_m(\prim)\ot \eta$ in $\Mat_m(S_1)$.
Similarly, we have
\begin{equation*}
\Lpm\big(\Mat_m(\base),1,\prim\big) = \Mat_m(S_1)^{1\ot \eta} = \Mat_m(R_1),
\end{equation*}
where $R_1 = \base[t_1^{\pm 1}]$ with   $t_1 = z_1^m$.
But setting
\[\ph = D_m(z_1)\in \Aut(\Mat_m(S_1)),\]
we have
\begin{align*}
\ph(1\ot \eta)\ph^{-1} &= \Ad\big(d_m(z_1)\big) (1\ot \eta)  \Ad\big(d_m(z_1)^{-1}\big)\\
&=\Ad\big(d_m(z_1)\big)   \Ad\big((1\ot \eta)d_m(z_1)^{-1}\big) (1\ot \eta)\\
&=\Ad\big(d_m(z_1) d_m(\eta z_1)^{-1}\big) (1\ot \eta)
=D_m(\prim)  \ot \eta.
\end{align*}
It follows that $\ph$ restricts to an isomorphism
\[\ph : \Lpm(\Mat_m(\base),1,\prim) \xrightarrow{\sim} \Lpm(\Mat_m(\base),D_m(\prim),\prim).\]
So using Remark \ref{rem:itloop}, we have
\begin{align*}
\Lp_{\boldm}\Big(\Mat_m(\base),\big(D_m(\prim),P_m(1)^q\big)&,\boldprim\Big) \simeq
\Lp_m\Big(\Lp_m\big(\Mat_m(\base),D_m(\prim),\prim\big),P_m(1)^q,\prim\Big)\\
&\simeq  \Lp_m\Big(\Lpm(\Mat_m(\base),1,\prim \big), \ph^{-1} P_m(1)^q \ph,\prim \Big).
\end{align*}
But $\ph^{-1}P_m(1)\ph = \Ad\big(d_m(z_1)^{-1}p_m(1)d_m(z_1)\big)$, and a direct calculation shows that
$d_m(z_1)^{-1}p_m(1)d_m(z_1) = z_1^{-1} p_m(z_1^m)$.
So $\ph^{-1}P_m(1)\ph = P_m(z_1^m)$.
Thus,
\begin{align*}
\Lp_{\boldm}\Big(\Mat_m(\base),(D_m(\prim),P_m(1)^q),\boldprim\Big)
&\simeq \Lp_m\Big(\Lp_m\big(\Mat_m(\base),1,\prim\big),P_m(z_1^m)^q,\prim\Big)\\
&= \Lp_m\big(\Mat_m(R_1),P_m(t_1)^q,\prim\big).\qedhere
\end{align*}
\end{proof}

\begin{proposition}
\label{prop:assocloop2} Suppose that $m\ge 1$, $\prim$ has order $m$ in $\base^\times$ and $q\in \bbZ$.
Then
\begin{equation}
\label{eq:assocloop2}
\Lp\big(\Mat_m(R_1),P_m(t_1)^q,\prim\big) \simeq M_{\gcd(q,m)}(Q(\prim^{\iota_{m}(\bar q)})),
\end{equation}
where  $\prim^{\iota_m(\bar q)}$ has the evident well-defined interpretation.
\end{proposition}

\begin{proof} By Lemma \ref{lem:assocloop1} and Lemma \ref{lem:quantum}(a), we have
\begin{equation}
\label{eq:assocloc2}
\Lp\big(\Mat_m(R_1),P_m(t_1),\prim\big) \simeq \Qp.
\end{equation}
Since $\iota_m(\bar 1) = \bar 1$, this proves \eqref{eq:assocloop2} when $q=1$.

We next suppose that $q = g$, where $g$ is a positive divisor of $m$.  Since $\iota_m(\bar g) = \bar g$,
we must show that
\begin{equation}
\label{eq:assocloc3}
\Lp_m\big(\Mat_m(R_1),P_m(t_1)^g,\prim\big) \simeq \Mat_g(Q(\prim^{g})).
\end{equation}
To  see this, we  regard matrices in $\Mat_m(R_1)$ as elements in $\End_{R_1}(R_1^m)$ by means
of the natural left action on column vectors.  Let $(x_1,\dots,x_m)$
be the standard ordered $R_1$-basis for $R_1^m$ and let $k= \frac m g$. Then, relative to the
new ordered $R_1$-basis
\[(x_1, x_{1+g},\dots,x_{1+(k-1)g}, x_2, x_{2+g},\dots,x_{2+(k-1)g},\dots,x_g, x_{g+g},\dots,x_{g+(k-1)g})\]
of $R_1^m$, $p_m(t_1)^g$ has matrix
\[p_{m,g}(t_1) := \diag(p_k(t_1),\dots,p_k(t_1))\]
in block diagonal form with $g$ diagonal $k\times k$-blocks.  So
\[\Lp_m\big(\Mat_m(R_1),P_m(t_1)^g,\prim\big) \simeq \Lp_m\big(\Mat_m(R_1),P_{m,g}(t_1),\prim\big),\]
where $P_{m,g}(t_1) = \Ad\big(p_{m,g}(t_1)\big)$.  It follows that
\[\Lp_m\big(\Mat_m(R_1),P_m(t_1)^g,\prim\big) \simeq \Mat_g\Big( \Lp_m\big(\Mat_k(R_1),P_k(t_1),\prim\big) \Big).\]
But using Lemma \ref{lem:changemult}(c) and \eqref{eq:assocloc2}, we have
\[ \Lp_m\big(M_k(R_1),P_k(t_1),\prim\big) \simeq \Lp_k\big(M_k(R_1),P_k(t_1),\prim^g\big) \simeq Q(\prim^g).\]
So we have \eqref{eq:assocloc3}.

Finally, suppose that $q\in \bbZ$ is arbitrary, and let $g = \gcd(m,q)$, $k = \frac m g$ and $r = \frac q g$.
Now $\gcd(r,k) = 1$.  Hence, since the  map $\bar s \mapsto \bar s$ from $U(m)$ to $U(k)$ is surjective,
we can choose  $s\in \bbZ$  satisfying
\[\gcd(s,m) =1 \andd sr \equiv 1 \pmod k.\]
Then
\begin{align*}
\Lp_m\big(\Mat_m(R_1),P_m(t_1)^q,\prim\big)&\simeq
\Lp_k\big(\Mat_m(R_1),P_m(t_1)^q,\prim^g\big)&&\text{(by Lemma \ref{lem:changemult}(c))}\\
&= \Lp_k\big(\Mat_m(R_1),(P_m(t_1)^g)^r,\prim^g\big)\\
&\simeq \Lp_k\big(\Mat_m(R_1),P_m(t_1)^g,(\prim^g)^s\big)&&\text{(by Lemma \ref{lem:changemult}(b))}\\
&\simeq \Lp_m\big(\Mat_m(R_1),P_m(t_1)^g,\prim^s\big)&&\text{(by Lemma \ref{lem:changemult}(c))}\\
&\simeq \Mat_g\big(Q(\prim^{sg})\big)&&\text{(by \eqref{eq:assocloc3})}\\
&\simeq \Mat_g\big(Q(\prim^{\iota_m(\bar q)})\big). \qedhere
\end{align*}
\end{proof}

\subsection{The structure of $\Lp(\fgb,\rho^q)$} \
\label{subsec:structurerot}

\emph{We now assume that $\fg = \fg(A)$, where $A$ is the affine GCM of type
$\typeaff{A}{\rkaff}{1}$, $\rkaff\ge~1$.}  We enumerate the fundamental roots
$\al_0,\al_1,\dots,\al_\rkaff$ in cyclic order
around the Dynkin diagram as in \cite[\S 4.8]{K2}; and we use  the
notation of  Section \ref{sec:autaff} for affine
algebras.  Let
\[\rho = (0,1,\dots,\rkaff)\in \Aut(A);\]
so $\rho$ is a rotation of the Dynkin diagram by 1 position.
Then a general rotation of the Dynkin diagram has the form
$\rho^q$, where $q\in \bbZ$.

We now prove a theorem that gives realizations of the loop algebras $\Lp(\fgb,\rho^q)$
as multiloop algebras (see \eqref{eq:rot2}) and as matrix algebras
over quantum tori (see \eqref{eq:rot3}).

\begin{theorem}
\label{thm:structurerot}
Suppose that $q\in \bbZ$.  Then, we have
\begin{align}
\label{eq:rot1}
\Lp(\fgb,\rho^q) &\simeq \Lp\big(\spl_{\rkaff+1}(R_1),P_{\rkaff+1}(t_1)^q\big),\\
\label{eq:rot2}
\Lp(\fgb,\rho^q) &\simeq \Lp\big(\spl_{\rkaff+1}(\base),(D_{\rkaff+1}(\zeta_{\rkaff+1}),P_{\rkaff+1}(1)^q)\big)\\
\intertext{and}
\label{eq:rot3}
\Lp(\fgb,\rho^q) &\simeq \spl_{\gcd(q,\rkaff+1)}\big(Q(\zeta_{\rkaff+1}^{\iota_{\rkaff+1}(\bar q)})\big).
\end{align}
\end{theorem}

\begin{proof}
By Theorem \ref{thm:Kacreal} (with $m=1$ and $t_1=z_1$), we can identify
\[\fgb = \spl_{\rkaff+1}(\base)\ot R_1 = \spl_{\rkaff+1}(R_1).\]
Furthermore,
by \cite[Thm.~7.4]{K2}, this can be done with
\begin{equation*}
\label{eq:chevgen}
\overline{e_0} = t_1 E_{{\rkaff+1},1},\quad \overline{f_0} = t_1^{-1}E_{1,{\rkaff+1}},\quad \overline{e_i} = E_{i,i+1},\quad \overline{f_i} = E_{i+1,i}\\
\end{equation*}
for $1\le i \le \rkaff$, where $E_{ij}$ denotes the $(i,j)$-matrix unit in $\Mat_{\rkaff+1}(R_1)$.
Now, a direct calculation shows that  $P_{\rkaff+1}(t_1^{-1})(\overline{e_i}) = \overline{e_{i+1}}$ and $P_{\rkaff+1}(t_1^{-1})(\overline{f_{i}}) = \overline{f_{i+1}}$
for $0\le i \le \rkaff$, where the subscripts are interpreted modulo $\rkaff+1$.  So
$\rho = P_{\rkaff+1}(t_1^{-1})$.
Therefore
$\Lp(\fgb,\rho^q)
= \Lp(\spl_{\rkaff+1}(R_1),P_{\rkaff+1}(t_1^{-1})^q)$.
Replacing $t_1$ by $t_1^{-1}$  we have~\eqref{eq:rot1}.

Now taking the derived algebras of both sides of \eqref{eq:assocloop1} (with $m= \rkaff+1$ and
$\prim = \zeta_{\rkaff+1}$) we see, using
Lemma \ref{lem:loopderived}(a), that the right hand sides
of \eqref{eq:rot1} and \eqref{eq:rot2} are isomorphic.
Similarly,  taking the derived algebras both sides  of \eqref{eq:assocloop2},
we see that the  right hand sides
of \eqref{eq:rot1} and \eqref{eq:rot3} are isomorphic.
\end{proof}

\begin{corollary}
\label{cor:rot}  If $q\in\bbZ$, $\Lp(\fgb,\rho^q)\in \bbM_2$.  Moreover,
if $q$ is relatively prime to $\rkaff+1$,
then the relative type of{\/} $\Lp(\fgb,\rho^q)$ is $\type{A}{0}$;  in other words
$\Lp(\fgb,\rho^q)$ is  anisotropic.
%(see Definition \ref{def:reltype1}).
\end{corollary}

\begin{proof} This  follows from \eqref{eq:rot3} and Proposition \ref{prop:spgquantum}(b).
\end{proof}

\section{Loop algebras of affine algebras relative to nontransitive diagram automorphisms} \
\label{sec:loopnt}

Suppose that
\emph{$\fg=\fg(A)$ is the Kac-Moody Lie algebra constructed from an affine GCM
$A = (a_{ij})_{i,j\in \ttI}$, and suppose that  $\sg\in \Aut(A)$,   $m$ is a positive integer and $\sg^m = 1$.}
We continue with the notation of Section \ref{sec:autaff} for affine
algebras.

In this section, we  recall results from \cite{ABP1} that show that if $\sg$ is  not transitive,
$\Lp(\fgb,\sg)$ is the centreless core of an EALA which is constructed as
the affinization of $\fg$ relative to $\sg$.  We use this and Theorem \ref{cor:EALAmult} about EALAs
to give a description
of the relative type of $\Lp(\fgb,\sg)$ in terms of $\sg$ and the root system of $\fg$.

\subsection{Affinization}\
\label{subsec:aff}

Initially we do not assume that $\sg$ is    nontransitive.

\begin{pgraph}
\label{pgraph:aff}
Let
\begin{equation}
\label{eq:Aff1}
\Aff_m(\fg,\sg) := \Lp_m(\fg,\sg) \oplus \base \tilde c \oplus \base \tilde d
\end{equation}
as a vector space, where $\tilde c, \tilde d\ne 0$, and define a skew-symmetric product
on   $\Aff_m(\fg,\sg)$ by
\begin{equation}
\label{eq:Aff2}
\begin{split}
[x \otimes z_1^i +r_1\tilde c +r_2\tilde d,  y \otimes & z_1^j + s_1\tilde c + s_2 \tilde d]= \\
&[x,y]\otimes z_1^{i+j}+j r_2y \otimes z^j-i s_2x \otimes z^i +
i\delta_{i+j,0} (x,y)\tilde c
\end{split}
\end{equation}
for $i,j\in \bbZ$, $x\in \fg^\bari$, $y\in \fg^\barj$, $r_1,r_2, s_1, s_2\in \base$.
One checks easily that $\Aff_m(\fg,\sg)$ is a Lie algebra which we call
the \emph{affinization of $\fg$} relative to $\sg$.
We define a form $\form$ on $\Aff_m(\fg,\sg)$ by
\begin{equation}
\label{eq:Aff3}
(x \otimes z_1^i +r_1\tilde c +r_2\tilde d \, \vert \, y \otimes z_1^j +s_1\tilde c +s_2 \tilde d)=
\delta_{i+j,0}(x\vert y)+r_1 s_2 +r_2 s_1.
\end{equation}
Then, since $\sg$ preserves the form on $\fg$ by Proposition \ref{prop:diagaff}, it follows
that $\form$ is a nondegenerate symmetric bilinear form on $\Aff_m(\fg,\sg)$ \cite[Lemma 3.2]{ABP1}.
Finally, let $\cH = (\fh^\sg\otimes 1) \oplus  \base \tilde c \oplus \base \tilde d$, in which case
$\cH$ is an abelian subalgebra of $\Aff_m(\fg,\sg)$.
\end{pgraph}

To discuss the root system
of $\Aff_m(\fg,\sg)$  we need some further  notation.

\begin{pgraph}
\label{pgraph:pidef}
We have $\sg(\fh) = \fh$, so $\sg$ acts on $\fh^*$ by the inverse dual action:
\[(\sg(\al))(h) =  \al(\sg^{-1}(h))\]
for $\al\in \fh^*$, $h\in \fh$.
Then
\begin{equation}
\label{eq:sgact1}
\sg(\Dl) = \Dl, \quad\text{so}\quad \sg(V) = V.
\end{equation}
Next, since $\sg$ has finite order, we have
\begin{equation}
\label{eq:Vdecomp}
V = V^\sg \oplus (1-\sg)(V),
\end{equation}
where $V^\sg$ denotes the fixed point set of $\sg$ acting on $V$.
Let
$\pi_\sg : V \to V^\sg$ be the \emph{projection of $V$ onto $V^\sg$ relative to the decomposition
\eqref{eq:Vdecomp}}.
\end{pgraph}

\begin{proposition}
\label{prop:affEALA}
Suppose that $\fg = \fg(A)$ is an affine Kac-Moody Lie
algebra with  root system $\Dl$, $\sg\in\Aut(A)$ is not transitive, and $\sg^m = 1$.
Then the triple
$(\Aff_m(\fg,\sg),\form,\cH)$ defined in \pref{pgraph:aff} is an EALA of nullity 2
which is discrete
if $\base = \bbC$.  Moreover, the centreless core of this EALA
is isomorphic to $\Lp_m(\fgb,\sg)$.  Finally,
$\piDlb$ is an irreducible finite root system
in $\overline{V^\sg}$ and
the   quotient type of $(\Aff_m(\fg,\sg),\form,\cH)$  is
equal to the type of the finite root system   $\piDlb$.
\end{proposition}
\begin{proof} If $\base = \bbC$, these statements are proved in \cite{ABP1}.  (See
Theorem 4.8, Lemma 3.57 and Equation (3.48) of \cite{ABP1}.) The result for
general $\base$ is proved by straightforward modification of the arguments given there.
In particular, instead of working with the $\bbR$-spans of the root systems for
$\fg$ and  $\Aff_m(\fg,\sg)$, one works with the $\bbQ$-spans.  Also,
Lemma 3.33 of \cite{ABP1} must be replaced by the statement
that the group generated by the isotropic roots of $\Aff_m(\fg,\sg)$ is finitely generated.  This is easily checked
using \cite[(3.31)]{ABP1}.  The remaining changes are evident, so we omit
further  explanation.
\end{proof}

\subsection{Relative type in the nontransitive case}

\begin{theorem}
\label{thm:relativent}   Suppose that $\fg = \fg(A)$ is an affine Kac-Moody Lie
algebra and $\sg\in\Aut(A)$ is not transitive. Then
$\Lp(\fgb,\sg)\in \bbM_2\cap \bbE_2$, and the relative type of
$\Lp(\fgb,\sg)$ is equal to the type of the finite root system $\piDlb$.
\end{theorem}
\begin{proof}  Let $\cL = \Lp(\fgb,\sg)$.  Then $\cL$ is fgc by Corollary
\ref{cor:KacReal4} and Proposition \ref{prop:ppfgc},
and $\cL \in \bbE_2$ by
Proposition \ref{prop:affEALA}.  Our conclusion now
follows from Proposition \ref{prop:affEALA} and Theorem~\ref{cor:EALAmult}.
\end{proof}

\begin{pgraph}  We will return in Section \ref{sec:projaff} to study the finite root system
$\piDlb$ and see how to calculate its type.
\end{pgraph}

\section{Characterizations of algebras in $\bbM_2$}
\label{sec:char}

\subsection{Characterizations}\
\label{subsec:char}

Our first main theorem  gives three characterizations of the algebras in~$\bbM_2$.

\begin{theorem}
\label{thm:char}
For a Lie algebra $\cL$, the following statements are equivalent:
\begin{itemize}
\item[(a)] $\cL \in \bbM_2$.
\item[(b)] $\cL\simeq \Lp(\fgb,\sg)$, where $\fg$ is an untwisted affine Kac-Moody Lie algebra
and $\sg$ is a diagram automorphism of $\fgb$.
\item[(c)] $\cL\simeq \Lp(\fgb,\sg)$, where $\fg$ is an  affine Kac-Moody Lie algebra
and $\sg$ is a finite order automorphism of $\fgb$ of first kind.
\item[(d)]  $\cL\in \bbI_2$ and $\Cd(\cL)$ is isomorphic to $R_2$.
\end{itemize}
\end{theorem}

\begin{proof} (c) and (d) are equivalent by Corollary \ref{cor:KacReal4} and
Proposition \ref{prop:kind}.  It remains therefore to show that (a), (b) and (c) are equivalent.

``(a) $\Rightarrow$  (b)'':  We have $\cL \simeq \Lp(\gd,\boldsg)$, where
$\gd$ is a finite dimensional simple Lie algebra and $\boldsg = (\sg_1,\sg_2)$ is a pair
of commuting finite order automorphisms of~$\gd$.  If
$P = \left[\begin{smallmatrix} a & b \\ c & d \end{smallmatrix}\right]\in \GL_2(\bbZ)$,
we write $\boldsg^P = (\sg_1^a \sg_2^c, \sg_1^b \sg_2^d)$.  Then,
$\Lp(\gd,\boldsg^P) \simeq \Lp(\gd,\boldsg)$ \cite[Lemma 5.3]{GP1}, so we can replace
$\boldsg$ by $\boldsg^P$ as needed.

Next $\gd = \fg(\dot A)$ for some GCM $\dot A$  of finite type, so, by Proposition \ref{prop:Aut}, we have $\Aut(\gd) = \Autz(\gd) \rtimes \Aut(\dot A)$, since  $\Out(\dot A) = \Aut(\dot A)$.  We let
$\pd : \Aut(\gd) \to \Aut(\dot A)$ be the projection onto the second factor relative to this decomposition.

We claim that for suitable
$P = \left[\begin{smallmatrix} a & b \\ c & d \end{smallmatrix}\right]\in \GL_2(\bbZ)$, we have
\begin{equation}
\label{eq:main1claim}
\pd(\sg_1^a \sg_2^c) = 1.
\end{equation}
Now $\Aut(\dot A)$ is isomorphic to $S_3$, $\bbZ/2\bbZ$ or $\set{1}$.  So, since
$\pd(\sg_1)$ and $\pd(\sg_2)$ commute, one of these is a power of the other.  If
$\pd(\sg_2)= \pd(\sg_1)^k$, then  \eqref{eq:main1claim} holds with
$P = \left[\begin{smallmatrix} -k & 1 \\ 1& 0 \end{smallmatrix}\right]$.
On the other hand, if $\pd(\sg_1)= \pd(\sg_2)^k$, then  \eqref{eq:main1claim} holds with
$P = \left[\begin{smallmatrix} 1 & 0 \\ -k& 1 \end{smallmatrix}\right]$.
This demonstrates the claim.

So we can replace
$\boldsg$ by $\boldsg^P$ and assume that $\pd(\sg_1) = 1$.  Hence, using Theorem \ref{thm:erase}(a)
(applied to $\gd$) and Corollary  \ref{cor:KacReal0}, we have
$\Lp(\gd,\sg_1) \simeq \Lp(\gd,\pd(\sg_1)) = \Lp(\gd,1) \simeq \fgb$, where
$\fg = \fg(A)$ is an untwisted affine algebra.  Therefore,
\begin{equation*}
\label{eq:main1loop}
\cL \simeq \Lp(\gd,\boldsg) \simeq \Lp(\Lp(\gd,\sg_1),\sg_2 \ot 1 |_{\Lp(\gd,\sg_1)}) \simeq \Lp(\fgb,\tau),
\end{equation*}
for some automorphism $\tau$ of $\fgb$ of finite order.

Next,  by Proposition
\ref{prop:centloop}, the centroid $\Cd(\cL)$ of $\cL$ is isomorphic to $R_2$.
Thus, by Proposition \ref{prop:kind}, $\tau$ is of first kind.  So using the notation of \pref{pgraph:project},
$\bar p(\tau) \in \Aut(A)$.  Finally, by Theorem \ref{thm:erase}(c), we have
$\cL \simeq \Lp(\fgb,\tau) \simeq \Lp(\fgb, \bar p(\tau)) = \Lp(\fgb,\sg)$.

``(b) $\Rightarrow$  (c)'':  This is trivial.

``(c) $\Rightarrow$  (a)'':  Suppose $\cL\simeq \Lp(\fgb,\sg)$, where $\fg = \fg(A)$ is an affine Kac-Moody Lie algebra
and $\sg$ is a finite order automorphism of $\fgb$ of first kind. By Theorem \ref{thm:erase}(c) again,
we have $\cL \simeq \Lp(\fgb, \bar p(\sg))$; so replacing $\sg$ by $\bar p(\sg)$, we can assume that
$\sg\in \Aut(A)$. Then $\cL\in \bbM_2$  by Corollary  \ref{cor:rot} and Theorem  \ref{thm:relativent}.
\end{proof}

\begin{remark}
If  $\cL\simeq \Lp(\fgb,\sg)$,  where $\sg$ is a diagram automorphism
of an affine algebra of type $\typeaff{X}{k}{1}$, it follows from Theorem \ref{thm:char} and
Proposition \ref{prop:loopperm} that  $\cL \simeq \Lp(\gd,\boldtau)$, where $\boldtau$ is a $2$-tuple of
commuting finite order automorphisms of a finite
dimensional simple Lie algebra $\gd$ of type $\type{X}{k}$.  It would  be interesting to obtain a simple
algorithm to find such a $\boldtau$ in each case. This may be possible with a careful examination
of the steps in the proof of   Theorem \ref{thm:char}.
\end{remark}

We now look more closely at the relationship between the class $\bbM_2$ and the classes $\bbE_2$
and $\bbI_2$.  This is presented in a sequence of corollaries of Theorem \ref{thm:char}.

\begin{corollary}
\label{cor:char1}
For a Lie algebra $\cL$, the following statements are equivalent:
\begin{itemize}
\item[(a)] $\cL \in \bbM_2\cap \bbE_2$.
\item[(b)] $\cL \in \bbI_2\cap \bbE_2$.
\item[(c)] $\cL \in \bbE_2$ and $\cL$ is fgc.
\item[(d)] $\cL\in \bbM_2$ and $\cL$ is isotropic.
\item[(e)] $\cL\simeq \Lp(\fgb,\sg)$, where $\fg$ is an untwisted affine Kac-Moody Lie algebra
and $\sg$ is a nontransitive diagram automorphism of $\fgb$.
\end{itemize}
\end{corollary}

\begin{proof} ``(a) $\Rightarrow$  (b)" holds since $\bbM_2\subseteq \bbI_2$,
``(b) $\Rightarrow$  (c)" follows from Proposition \ref{prop:ppfgc},
and ``(c) $\Rightarrow$  (d)" follow from Theorem \ref{cor:EALAmult}.

``(d) $\Rightarrow$  (e)": By Theorem \ref{thm:char}, we have $\cL\simeq \Lp(\fgb,\sg)$,
where $\fg = \fg(A)$ is an untwisted affine algebra
and $\sg\in \Aut(A)$.  Then, since $\cL$ is isotropic, it follows from   Corollary \ref{cor:rot}
that $\sg$ is not transitive.

``(e) $\Rightarrow$  (a)":   This  holds by Theorem \ref{thm:relativent}.
\end{proof}

\begin{corollary}
\label{cor:char2}
For a Lie algebra $\cL$, the following statements are equivalent:
\begin{itemize}
\item[(a)] $\cL \in \bbM_2\setminus \bbE_2$.
\item[(b)] $\cL\in \bbM_2$ and $\cL$ is anisotropic.
\item[(c)] $\cL\simeq \Lp(\fgb,\sg)$, where $\fg$ is an untwisted affine Kac-Moody Lie algebra
and $\sg$ is a transitive diagram automorphism of $\fgb$.
\item[(d)] $\cL \simeq \spl_1(\Qp)$, where $\prim \ne 1$ is a root of unity in $\base^\times$.
\end{itemize}
\end{corollary}

\begin{proof} First ``(a) $\Rightarrow$  (c)" by Theorem \ref{thm:char} and Corollary \ref{cor:char1};
second ``(c) $\Rightarrow$  (b)" by  Corollary \ref{cor:rot}; and third
``(b) $\Rightarrow$  (a)" by Corollary \ref{cor:char1}.  Therefore,
(a), (b) and (c) are equivalent.  Finally (d) and (c) are equivalent
by   \eqref{eq:rot3}, since any root of unity $\ne 1$ in $\base^\times$ equals
$\zeta_{\rkaff+1}^r$ for some $\rkaff \ge 1$ and some $r$ relatively prime to $\rkaff+1$.
\end{proof}

\begin{remark}
\label{rem:EminusL}  We have described the  algebras in
$\bbM_2\setminus\bbE_2$ as matrix algebras.  This can also be done for the algebras
in $\bbE_2\setminus \bbM_2$ using the coordinatization
theorems for algebras in $\bbE_n$ (see \pref{pgraph:coordinate}).
In nullity 2, it turns out that the only algebras in $\bbE_2$
that are not fgc, and hence not in $\bbM_2$, are  the algebras  isomorphic to
$\spl_g{(\Qp})$, for some $g\ge 2$ and some $\prim$
of infinite order in $\base^\times$.  (See \cite{Neh2} for a more general theorem.)
We will not use this fact in this paper except to include it in  Figure~\ref{fig:classdiag}  below.
\end{remark}

\begin{corollary}
\label{cor:char3}
For a Lie algebra $\cL$, the following statements are equivalent:
\begin{itemize}
\item[(a)] $\cL \in \bbI_2\setminus \bbM_2$.
\item[(b)] $\cL\simeq \Lp(\fgb,\chev\sg)$, where $\fg$ is an untwisted affine Kac-Moody Lie algebra,
$\chev$ is the Chevalley automorphism and $\sg$ is a  diagram automorphism of $\fgb$.
\item[(c)] $\cL\simeq \Lp(\fgb,\sg)$, where $\fg$ is an affine Kac-Moody Lie algebra
and $\sg$ is a finite order automorphism of second kind of $\fgb$.
\end{itemize}
\end{corollary}

\begin{proof} `(a) $\Rightarrow$  (b)"  By Corollary \ref{cor:KacReal4},
$\cL \simeq \Lp(\fgb,\sg)$ for some affine Lie algebra $\fg = \fg(A)$
and some finite order $\sg\in\Aut(\fg)$.  By Theorem \ref{thm:erase}(c), we can
assume that $\sg\in \Out(A)$.  Finally, $\sg\notin\Aut(A)$, by Theorem \ref{thm:char}.

``(b) $\Rightarrow$  (c)" is trivial.

``(c) $\Rightarrow$  (a)" By Corollary \ref{cor:KacReal4}, $\cL\in \bbI_2$.  Further, if
$\cL\in \bbM_2$, then  $\Cd(\cL) \simeq R_2$ by Proposition \ref{prop:centloop},
which implies that $\sg$ is of first kind by Proposition~\ref{prop:kind}.
\end{proof}

\begin{pgraph}
Much of the information from this section is summarized in Figure \ref{fig:classdiag} below.
In the figure, $\bbM_2$ has a bold line border, $\bbI_2$ is the tall vertical rectangle, and
$\bbE_2$ is the wide horizontal rectangle.
\end{pgraph}

%%%%%CLASS DIAGRAM%%%%%%%%%%%%%%%%%%%
%%%Set up text for the boxes.
%%%Box McapE
\newcommand\McapE{%
\parbox[t]{3in}
{
\small
$\bbM_2\cap \bbE_2 = \bbI_2\cap \bbE_2$\\
$= \{\Lp(\fgb,\sg) \suchthat \fg \text{ untwisted affine,}$\\
$\text{ \hspace{.95in} $\sg\in \Aut(A)$ not transitive}\}$\\
$= \set{\text{isotropic algebras in } \bbM_2}$\\
$= \set{\text{fgc algebras in } \bbE_2}$
}
}%
%%%Box MminusE
\newcommand\MminusE{%
\parbox[t]{3in}
{\small
$\bbM_2\setminus \bbE_2$\\
$= \{\Lp(\fgb,\sg) \suchthat \fg \text{ untwisted affine,}$\\
$\text{ \hspace{.95in} $\sg\in \Aut(A)$ transitive}\}$\\
$= \set{\text{anisotropic algebras in } \bbM_2}$\\
$= \set{\spl_1(\Qp) \suchthat \theta \text{  a root of unity} \ne 1}$
%$= \set{\spl_1(\Qp) \suchthat 1 < \order{\prim} < \infty}$
}
}%
%%%Box EminusM
\newcommand\EminusM{%
\parbox[t]{2in}
{\small
$\bbE_2\setminus \bbM_2$\\
$=$ non-fgc algebras in $\bbE_2$\\
%$= \set{\spl_g(\Qp) \suchthat g \ge 2,\ \order{\prim} = \infty}$
$= \{\spl_g(\Qp) \suchthat g \ge 2,$\\
$\text{ \hspace{.35in} $\theta$ not a root of unity}\}$
}
}%
%%%Box IminusM
\newcommand\IminusM{%
\parbox[t]{3in}
{\small
$\bbI_2\setminus \bbM_2$\\
$= \set{\Lp(\fgb,\omega\sg) \suchthat \fg \text{ untwisted affine}, \sg\in \Aut(A)}$\\
$= \{\text{algebras in } \bbI_2
\text{ with centroid}  \not\simeq R_2\}$
\\
}
}%

%%%
%%&Draw class diagram
\begin{figure}[ht]
\centering
\parbox{6.7in}
{
\setlength{\unitlength}{5pt}
\begin{picture}(66,39)(-5,0)
%%%E2box%%
\drawline(0,0)(66,0)(66,14)(0,14)(0,0)
%%%I2box%%
\drawline(0,0)(39,0)(39,38)(0,38)(0,0)
%%%M2box%%
\thicklines
\drawline(0,0)(39,0)(39,28)(0,28)(0,0)
\drawline(.1,.1)(38.85,.1)(38.85,27.9)(.1,27.9)(.1,.1)
%Put text in boxes
\put(1,35){\IminusM}
\put(1,25){\MminusE}
\put(1,11){\McapE}
\put(41,11){\EminusM}
%Labels
\put(63,15){$\bbE_2$}
\put(40,26){$\bbM_2$}
\put(-3,36){$\bbI_2$} % (2,39)
\end{picture}
}
\caption{The classes $\bbM_2$, $\bbI_2$ and $\bbE_2$}
\label{fig:classdiag} %Place immediately after \caption.
\end{figure}

\begin{pgraph}
There   is another interesting class of algebras related to multiloop algebras which has
been investigated in \cite{GP1} using cohomological methods.
This is the class, which we denote here by $\bbF_n$, of \emph{$R_n$-forms}
of finite dimensional simple Lie algebras.  (If $\fg$ is an algebra, an \emph{$R_n$-form of $\fg$} is an algebra $\cL$ over $R_n$
such that $\cL\ot_{R_n} S \simeq_S \fg\ot S$ for some faithfully flat and finitely presented
extension $S/R_n$ of unital commutative associative algebras over $\base$.)
For any $n\ge 1$, $\bbF_n$ contains $\bbM_n$ \cite[\S 5.1]{GP1}; if $n=1$, $\bbF_1 = \bbM_1$ \cite{P2}, so
$\bbF_1 = \bbM_1 = \bbI_1 = \bbE_1$;
but already if $n=2$, an example of B.~Margaux shows that $\bbF_2$ properly contains $\bbM_2$  \cite[Example 5.7]{GP1}.
Furthermore, algebras in $\bbF_2$ are fgc and have centroids isomorphic
to $R_2$ \cite[Lemma 4.6]{GP1}, so $\bbF_2 \cap \bbI_2 = \bbM_2$ and
$\bbF_2 \cap \bbE_2 = \bbM_2 \cap \bbE_2$.  We will not discuss the
class $\bbF_n$ further here, and instead refer the reader to \cite{GP1} for more information.
\end{pgraph}

\section{Isomorphism conditions for matrix algebras over quantum tori}
\label{sec:isomat}

In the next two sections we  develop the results we need to determine
when two algebras of the form $\Lp(\fgb,\sg)$, where $\fg$ is an affine algebra
and $\sg$ is a diagram automorphism of $\fgb$, are isomorphic.  We do this for the rotation
case in this section and for the nontransitive case in the next section.

Recall that  $R_2 = \base[t_1^{\pm 1},t_2^{\pm 1}]$. We let
\[K_2 =  \base(t_1,t_2)\]
be the quotient field of $R_2$; that is $K_2$ is the field of rational functions in the variables $t_1,t_2$ over
$\base$.

Recall also from
\pref{pgraph:centquantum} that if
$\prim$ is an element of $\base^\times$ of finite order~$m$ and
$\Qp$ is the quantum torus  determined by $\prim$  with distinguished generators
$x_1,x_2,x_1^{-1},x_2^{-1}$, then we  have identified $R_2$
with the centre of $\Qp$ by setting $t_1 = x_1^m$ and $t_2 = x_2^m$.

\subsection{Comparing $Q(\prim)$ and $Q(\prim^{-1})$}\
\label{subsec:compare}

Suppose that $\prim\in \base^\times$ has order $m$.

\begin{pgraph}\label{pgraph:exchange}  Note  that $\Qp \simeq Q(\prim^{-1})$ under the isomorphism that exchanges
the distinguished generators $x_1$ and~$x_2$.  However, it is not true that
$\Qp \simeq_{R_2} Q(\prim^{-1})$ unless $m = 1$ or $2$, as we'll see below
in Lemma \ref{lem:isomquantum}.
\end{pgraph}

\begin{pgraph}
If $\cA$ is an arbitrary algebra,
we denote the
\emph{opposite algebra} of $\cA$, with underlying vector space $\cA$ and multiplication $(x,y) \mapsto yx$,
by $\cA^\op$.
\end{pgraph}

\begin{lemma}\label{lem:opposite}
Suppose that $g\ge 1$.  Then
\begin{gather*}
\notag
\Mat_g(\Qp) \simeq_{R_2} \Mat_g\big(Q(\prim^{-1})\big)^\op,
\\
\gl_g(\Qp) \simeq_{R_2} \gl_g\big(Q(\prim^{-1})\big) \andd \spl_g(\Qp) \simeq_{R_2} \spl_g\big(Q(\prim^{-1})\big).
\end{gather*}
\end{lemma}

\begin{proof} Let $*$ denote the product in $Q(\prim^{-1})^\op$.  Then  in $Q(\prim^{-1})^\op$, we have
$x_1*x_2 = x_2x_1 = \prim  x_1 x_2 = \prim x_2 * x_1$.  Hence, we have a $\base$-algebra isomorphism of
$\Qp$ onto $Q(\prim^{-1})^\op$  with $x_i \mapsto x_i$.    It is clear that this is
an $R_2$-algebra isomorphism, so $\Qp \simeq_{R_2} Q(\prim^{-1})^\op$.  Thus,
\[\Mat_g(\Qp) \simeq_{R_2} \Mat_g(Q(\prim^{-1})^\op) \simeq_{R_2} \Mat_g\big(Q(\prim^{-1})\big)^\op,\]
where the last isomorphism is the transpose map; and
\[\gl_g(\Qp) \simeq_{R_2} \gl_g\big(Q(\prim^{-1})^\op\big) \simeq_{R_2} \gl_g\big(Q(\prim^{-1})\big),\]
where the last isomorphism is additive inversion.  Taking derived algebras now finishes the proof.
\end{proof}

\begin{lemma}\label{lem:extension} If $\ph \in \Aut(R_2)$, then there exists a $\base$-algebra isomorphism
of $\Qp$ onto either $\Qp$ or $Q(\prim^{-1})$ that extends $\ph$.
\end{lemma}

\begin{proof}    It is well known that there
exists a unique $P= \left[\begin{smallmatrix} p_{11} & p_{12} \\ p_{21} & p_{22}\end{smallmatrix}\right]\in \GL_2(\bbZ)$ such that
\[\ph(t_i) \sim t_1^{p_{i1}}  t_2^{p_{i2}}\]
for $i=1,2$, where
$\sim$ means that the first element is obtained from the second by multiplying by an element of $\base^\times$.
Let $\varepsilon = \det(P) \in \set{\pm 1}$.  We will show that $\ph$
extends to an isomorphism from $\Qp$ to  $Q(\prim^\varepsilon)$.
Let $\psi_{\text{ex}} : \Qp \to Q(\prim^{-1})$ be the isomorphism mentioned in \pref{pgraph:exchange}.
Replacing $\ph$ by $\psi_{\text{ex}}|_{R_2} \ph$ if necessary, we can assume that $\varepsilon = 1$.
But  then, since $P\in \SPL_2(\bbZ)$, there exists $\psi_P\in \Aut(\Qp)$ such
that  $\psi_P(x_i) \sim  x_1^{p_{i1}}  x_2^{p_{i2}}$ for $i=1,2$
\cite[Lemmas IV.3 and  IV.6]{Neeb}.
Then
$\psi_P(t_i) \sim  t_i^{p_{i1}}  t_2^{p_{i2}}$,
so $(\psi_P^{-1}\ph)(t_i) \sim t_i$ for $i=1,2$.  Hence, replacing $\ph$ by $\psi_P^{-1}|_{R_2}\ph$, we can assume that
$\ph(t_i) \sim t_i$ for $i=1,2$.  Since $\base^\times = (\base^\times)^m$, the proof in this case is clear.
\end{proof}

\subsection{Cyclic algebras and quantum tori}\
\label{subsec:cyclicquantum}

Suppose that $\prim\in \base^\times$ has order $m$.

\begin{pgraph}\headtt{Cyclic algebras}
\label{pgraph:cyclic}
Suppose that $F$ is a unital commutative associative  $\base$-algebra and
$u_1,u_2$ are units in $F$.  We let
\[\cA = (u_1,u_2;m,F,\prim)\]
denote the algebra over $F$ presented by the generators
$y_1,y_2$ subject to the relations $y_1y_2 = \zeta y_2 y_1$, $y_1^m = u_1$ and $y_2^m = u_2$.  We call
$\cA$ the \emph{cyclic algebra} over $F$ determined by
$u_1$, $u_2$ and $\prim$, and we call
$y_1,y_2$ the \emph{distinguished generators} of $\cA$ over $F$.  It is easy to see that
$\cA$ is a free $F$-module of rank $m^2$ with basis
$\set{y_1^{k_1}y_2^{k_2}\suchthat 0\le k_1,k_2\le m-1}$.
If $F$ is a field, it is known that $\cA$ is a central simple
algebra over $F$ \cite[Theorem 11.1]{D}, and we denote the element represented by $\cA$
in the Brauer group $\Br(F)$ of $F$ by
$[u_1,u_2;m,F,\prim]$.
\end{pgraph}

\begin{lemma} \label{lem:cyclicquantum}\
\begin{itemize}
\item[(a)] $Q(\prim) \simeq_{R_2} (t_1,t_2;m,R_2,\prim)$
\item[(b)] $Q(\prim)\ot_{R_2} K_2  \simeq_{K_2} (t_1,t_2;m,K_2,\prim)$
\item[(c)] $Q(\prim)\ot_{R_2} K_2$ is a central division algebra of dimension $m^2$ over $K_2$.
\item[(d)] $[t_1,t_2;m,K_2,\prim]$ has order $m$ in $\Br(K_2)$.
\end{itemize}
\end{lemma}

\begin{proof}  (a) follows easily from the presentations for the two algebras, and (b) follows from
(a).

(c) is well  known (see for example the arguments in \cite[\S 2--3]{Y1}).  For the reader's convenience, we briefly recall the proof.
In view of (b) and the remarks in \pref{pgraph:cyclic}, it suffices to show that
$Q(\prim)\ot_{R_2} K_2$ is a division algebra.   Since $Q(\prim)\ot_{R_2} K_2$ is finite dimensional over $K_2$,
it is enough to show that $Q(\prim)\ot_{R_2} K_2$ is a domain, and thus it suffices to show that
$Q(\prim)$ is a domain.  This is easily seen looking at components of highest degree (relative
to the lexicographic ordering) in the $\bbZ^2$-grading for $\Qp$.

To prove (d), let $r$ be the order of $[t_1,t_2;m,K_2,\prim]$ in $\Br(K_2)$.  Since
the algebra $(t_1,t_2;m,K_2,\prim)$ is  central simple of dimension $m^2$ over $K_2$, we know $r$ divides $m$.
But by \cite[Lemma 11.6]{D}, we have $r[t_1,t_2;m,K_2,\prim] = [t_1,t_2;\frac mr,K_2,\prim^r]$,
so \[[t_1,t_2;\textstyle \frac mr,K_2,\prim^r] = 0.\]
Also, by (b),  $(t_1,t_2;\textstyle \frac mr,K_2,\prim^r) \simeq_{R_2} Q(\prim^{r}) \ot_{R_2}K_2$,
which is a division algebra of dimension $(\frac mr)^2$ over $K_2$ by (c).  Hence, $\frac mr = 1$, so $r = m$.
\end{proof}

\begin{remark}  It is known that the homomorphism
$\Br(R_2) \mapsto \Br(K_2)$ induced by base ring extension is an injection \cite[Cor. IV.2.6]{Mi},
where $\Br(R_2)$ is the Brauer group
of the ring $R_2$.  Using this fact, Lemma \ref{lem:cyclicquantum}(d)
is equivalent to the statement, proved in   \cite[Prop.~3.16]{GP1}
in a different way, that  $(t_1,t_2;m,R_2,\prim)$
represents an element of order $m$ in $\Br(R_2)$.
\end{remark}

\subsection{Isomorphism conditions}\
\label{subsec:isocond}

\begin{lemma} \label{lem:isomquantum}\  Suppose that $\prim_1$ and $\prim_2$
are elements of $\base^\times$ of finite order.   If
$Q(\prim_1)\ot_{R_2} K_2 \simeq_{K_2}  Q(\prim_2)\ot_{R_2} K_2$, then $\prim_1 = \prim_2$.
\end{lemma}

\begin{proof} By Lemma \ref{lem:cyclicquantum}(c), $\prim_1$ and $\prim_2$ have the same order
$m$ in $\base^\times$.  So we can write $\prim_2 = \prim_1^q$, where $q\in \bbZ$ and $\gcd(q,m) = 1$.  Then, by Lemma \ref{lem:cyclicquantum}(b),
we have
$[t_1,t_2;m,K_2,\prim_1] = [t_1,t_2;m,K_2,\prim_1^q]$
in $\Br(K_2)$. So  $q[t_1,t_2;m,K_2,\prim_1] = q[t_1,t_2;m,K_2,\prim_1^q]$.  But, by \cite[Lemma 11.5]{D},
the right hand side is equal to $[t_1,t_2;m,K_2,\prim_1]$. So $(q-1)[t_1,t_2;m,K_2,\prim_1] = 0$, and therefore,
by Lemma \ref{lem:cyclicquantum}(d), we have $q\equiv 1 \pmod {m}$.  Thus $\prim_1 = \prim_2$.
\end{proof}

\begin{theorem}
\label{thm:isomquantum} Suppose that $g_1$ and $g_2$ are positive integers,
and that $\prim_1$ and $\prim_2$ are elements of $\base^\times$ of finite order.
Then, the following statements are equivalent:
\begin{itemize}
\item[(a)] $\spl_{g_1}(Q(\prim_1)) \simeq \spl_{g_2}(Q(\prim_2))$
\item[(b)] $g_1= g_2$ and $\prim_1 = \prim_2^{\pm 1}$.
\item[(c)] $\Mat_{g_1}(Q(\prim_1)) \simeq \Mat_{g_2}(Q(\prim_2))$
\end{itemize}
\end{theorem}

\begin{proof}  For simplicity, we write $R = R_2$ and $K = K_2$.

``(c)$\Rightarrow$(a)'' is clear, and, since $Q(\prim_1) \simeq Q(\prim_1^{-1})$,
``(b)$\Rightarrow$(c)'' is also clear. So all that must be proved
is that ``(a)$\Rightarrow$(b)''.

Let $\cL_i = \spl_{g_i}(Q(\prim_i))$ for $i=1,2$, and suppose that $\rho: \cL_1 \to \cL_2$
is an isomorphism of $\base$-algebras. Let $m_i$ be the order of $\prim_i$ for $i=1,2$.
If $g_1m_1 = 1$, then $g_1 = 1$ and $\prim_1 = 1$.  Thus $\cL_1 = 0$, so $\cL_2 = 0$.  Therefore,
$g_2= 1$ and $\prim_2 = 1$, so (b) holds.  Hence, we can assume that $g_1m_1 > 1$ and similarly
$g_2m_2 > 1$.  Thus, by Lemma \ref{lem:quantum}(e), the natural homomorphism $R \to C(\cL_i)$
is an isomorphism which we regard as an identification for $i=1,2$.  So
we have the induced automorphism $\ph = C(\rho)$ of $R$ satisfying
$\rho(rx) = \ph(r) \rho(x)$
for $r\in R$, $x\in \cL_1$ (see \pref{pgraph:cfunctor}).  Then, by Lemma \ref{lem:extension}, there exists an isomorphism
$\psi : Q(\prim_1) \to Q(\prim_1^\varepsilon)$, where $\varepsilon = \pm 1$, such that
$\psi|_R = \ph^{-1}$.  Further, by the last isomorphism in Lemma \ref{lem:opposite}, we can assume that $\varepsilon = 1$.
Now $\psi\in \Aut(Q(\prim_1))$ induces  $\tilde \psi\in \Mat_{g_1}(Q(\prim_1))$, which restricts to
$\tilde \psi|_{\cL_1}\in \Aut(\cL_1)$.  Then, replacing $\rho$ by $\rho \tilde \psi|_{\cL_1}$, we can assume
that $\rho$ is $R$-linear.  Thus, $\cL_1 \simeq_R \cL_2$, so $\cL_1\ot_{R}K \simeq_K \cL_2\ot_{R}K$.

Now  since $K/R$ is a flat extension of commutative $k$-algebras \cite[Chap~I, \S~2.4, Thm~1]{Bo1}, we have
$(\mathcal{M}\otimes_R K)' \simeq_K \mathcal{M}'\otimes_R K$ for any Lie algebra
$\mathcal{M}$ over $R$.
Thus,
\begin{align*}
\spl_{g_i}(Q(\prim_i)\ot_R K) &\simeq_K \left(\gl_{g_i}(Q(\prim_i)\ot_R K)\right)'
\simeq_K \left(\gl_{g_i}(Q(\prim_i))\ot_R K\right)'\\
&\simeq_K \left(\gl_{g_i}(Q(\prim_i))\right)'\ot_R K
\simeq_K \spl_{g_i}(Q(\prim_i))\ot_R K \simeq \cL_i \ot_R K
\end{align*}
for $i=1,2$.  So
\begin{equation}
\label{eq:isomquantum1}
\spl_{g_1}(Q(\prim_1)\ot_R K) \simeq_K \spl_{g_2}(Q(\prim_2)\ot_R K).
\end{equation}

Now, by Lemma \ref{lem:cyclicquantum}, $Q(\prim_i)\ot_R K$ is an  $m_i^2$-dimensional central division algebra
over $K$ for $i=1,2$.  Thus, by  the isomorphism theorem for central simple Lie algebras of type $A$ \cite[Thm.~X.10]{J},
\eqref{eq:isomquantum1} implies that
$\Mat_{g_1}(Q(\prim_1)\ot_R K)$ is isomorphic as an algebra over $K$ to either $\Mat_{g_2}(Q(\prim_2)\ot_R K)$
or   $\left(\Mat_{g_2}(Q(\prim_2)\ot_R K)\right)^\op$.  Now,  by Lemma \ref{lem:opposite}, we can replace $\prim_2$ by $\prim_2^{-1}$ if necessary
and assume that
\[\Mat_{g_1}(Q(\prim_1)\ot_R K) \simeq_{K} \Mat_{g_2}(Q(\prim_2)\ot_R K).\]
By the uniqueness part of Wedderburn's structure theorem \cite[Thm.~3.5(ii)]{Pi}, it follows that
$g_1 = g_2$ and $Q(\prim_1)\ot_R K \simeq_K Q(\prim_2)\ot_R K$. Thus,
$\prim_1 = \prim_2$ by Lemma~\ref{lem:loopderived}.
\end{proof}

\section{Calculating the Relative type in the nontransitive case}
\label{sec:projaff}

Suppose  that
\emph{$\fg=\fg(A)$ is the Kac-Moody Lie algebra constructed from an affine GCM
$A = (a_{ij})_{i,j\in \ttI}$, where $\ttI = \set{0,\dots,\rkaff}$ with $\rkaff\ge 1$, and that
$\sg\in \Aut(A)$ is not transitive}. We  continue with the notation of Section \ref{sec:autaff} for affine
algebras, and we choose a positive integer $m$ such that  $\sg^m = 1$.

We  saw in Theorem \ref{thm:relativent} that the relative type of  $\Lp(\fgb,\sg)$
is the type of the finite root system $\piDlb$.
In this section  we use methods from   \cite{FSS} and \cite{Bau}
to compute the type of  $\piDlb$.

\subsection{Calculating $\pi_\sg(\Dl)^\times$} \
\label{subsec:projnt}

\begin{pgraph}
\label{pgraph:setuppi}  Recall from \pref{pgraph:pidef} that $\sg$ acts on $\fh^*$ and that
$\sg(\Dl) = \Dl$, so  $\sg(V) = V$.  Also, by \eqref{eq:diagaff1}, we have
$\sg(\al_i) = \al_{\sg(i)}$ for $i\in\ttI$. Further,
since $\sg$ preserves the form $\form$ on $\fh^*$ by Proposition \ref{prop:diagaff},
$\sg$ also preserves the form $\form$ on $V$.

Recall next that $V = V^\sg \oplus (1-\sg)(V)$. In fact,
since $\sg$ preserves the form $\form$ this decomposition is orthogonal:
\begin{equation}
\label{eq:Vdecomp1}
V = V^\sg \perp (1-\sg)(V).
\end{equation}
As in \pref{pgraph:pidef}, let
$\pi_\sg : V \to V^\sg$ be the projection of $V$ onto $V^\sg$ relative to the decomposition
\eqref{eq:Vdecomp1}.
Then, since $\sg$ has period $m$, we have
\begin{equation*}
\label{eq:piact1}
\pi_\sg(\al) = \frac 1m \sum_{i=0}^{m-1} \sg^i(\al)
\end{equation*}
for $\al\in V$.
\end{pgraph}

\begin{pgraph}  Although we will not use this fact, it is clear that
$\fh = \fg^\sg \perp (1-\sg)\fh$,
$\pi_\sg(\al)|_{\fh^\sg} = \al |_{\fh^\sg}$ and  $\pi_\sg(\al)\mid_{(1-\sg)\fh} = 0$.  In this way, $\pi_\sg(\al)$ can be identified
with $\al|_{\fh^\sg}$.  This is the point of view taken in \cite{Bau} and \cite{ABP1}.
\end{pgraph}

\begin{pgraph}
Let
\[\brvI = \set{i\in \ttI \suchthat \sg^k (i) \ge i \text{ for } k\in \bbZ}.\]
For $i\in \brvI$, let
$\cO(i)$ denote the orbit containing $i$ under the action of the group $\langle \sg \rangle$.
Then the sets $\cO(i)$, $i\in \brvI$, are the distinct orbits
for the action of $\langle \sg \rangle$ on $\ttI$; and,
for $i\in \ttI$, $i$ is the least element of
$\cO(i)$.

For $i\in \brvI$, set
\[\mu_i = \pi_\sg(\al_i).\]
Then $\pi_\sg(\al_p) = \mu_i$
for $i\in \brvI$ and $p\in \cO(i)$.
Also $\set{\mu_i}_{i\in \brvI}$ is a $\bbQ$-basis for $V^\sg$,
so $\dim_\bbQ(V^\sg) = \cardnum{\brvI}$.
\end{pgraph}

\begin{pgraph}
\label{pgraph:dimVsgb}
Since $\sg(\al_i) = \al_{\sg(i)}$ for $i\in\ttI$, we have
$\sg(\delta) = \delta$ by  \eqref{eq:Xaffdiag1} and \eqref{eq:Xaffdiag2}.
So $\delta\in V^\sg$, and hence
$\dim_\bbQ(\overline{V^\sg}) = \cardnum{\brvI} - 1$.
\end{pgraph}

The next lemma was observed in the proof of Proposition III.3.4 of \cite{Bau} (see also \cite[\S 2.5]{FSS}).   It is proved by
checking the claim for each possible affine matrix $A$ and each $\sg\in \Aut(A)$ that is not  transitive.

\begin{lemma}
\label{lem:si1} If $i\in \brvI$, then exactly one of the following holds:
\begin{itemize}
\item[(a)]  The elements of $\cO(i)$ are pairwise orthogonal.
\item[(b)]  $\cO(i) = \set{\al_p,\al_q}$, where
$p,q\in \brvI$, $p\ne q$ and  $a_{pq} = a_{qp} = -1$.
\end{itemize}
\end{lemma}

\begin{pgraph}
\label{pgraph:sidef}
For $i\in \brvI$,
 we now define
\begin{equation}
\label{eq:si1}
s_i = \left\{
        \begin{array}{ll}
          1, & \hbox{if (a) holds in Lemma \ref{lem:si1};} \\
          2, & \hbox{if (b) holds in Lemma \ref{lem:si1}.}
        \end{array}
      \right.
\end{equation}
\end{pgraph}

\begin{lemma}  \cite[\S 2.1]{FSS}
\label{lem:si2}
Let $i\in \brvI$. Then $s_i (3-s_i) = 2$ and
\begin{equation} \label{eq:si2}
s_i = \textstyle 3- \sum_{p \in \cO(i)} a_{pi}.
\end{equation}
\end{lemma}

\begin{proof} This  follows immediately from Lemma \ref{lem:si1} and
the definition of $s_i$.\end{proof}

We now calculate the Cartan integers for the set $\set{\mu_i}_{i\in \brvI}$.

\begin{proposition} \label{prop:cartanbeta}  Suppose that $i,j\in \brvI$.  Then
\begin{itemize}
\item[(a)]  $(\mu_i\vert \mu_i) \ne 0$, so $\mu_i\in \pi_\sg(\Dl)^\times$.
\item[(b)] If $\al\in V$, we have $\displaystyle\frac {(\mu_i \vert \pi_\sg(\al))}{(\mu_i\vert \mu_i)}
= s_i \sum_{p\in \cO(i)}  \frac{(\al_p\vert\al)}{(\al_p\vert\al_p)}$.
\item[(c)] $\displaystyle 2 \frac {(\mu_i \vert \mu_j)}{(\mu_i\vert \mu_i)}
= s_i \sum_{p\in \cO(i)} a_{pj} \in s_i\bbZ$.
\end{itemize}
\end{proposition}

\begin{proof}  Suppose that $\al\in V$.  Then
\begin{align}\notag
(\mu_i\vert \pi_\sg(\al)) &= \textstyle (\mu_i \vert \frac 1m \sum_{p=0}^{m-1}\sg^p(\al))
= (\mu_i \vert \al)\quad\text{(since $\mu_i\in V^\sg$)}\\ \notag
&=  \frac 1{m} ( \sum_{p=0}^{m-1}\sg^p(\al_i) \vert \al)
=  \frac 1{m} \frac m{\cardnum{\cO(i)}} ( \sum_{p\in \cO(i)}\al_p \vert \al)\\\label{eq:GCM1}
&=    \frac {(\al_i\vert \al_i)}{\cardnum{\cO(i)}} \sum_{p\in \cO(i)} \frac{(\al_p \vert \al)}{(\al_i\vert \al_i)}
= \frac {(\al_i\vert \al_i)}{\cardnum{\cO(i)}} \sum_{p\in \cO(i)} \frac{(\al_p \vert \al)}{(\al_p\vert \al_p)}.
\end{align}
Now putting $\al = \al_i$ in \eqref{eq:GCM1}, we get
\begin{equation} \label{eq:GCM2}
(\mu_i\vert\mu_i) =
\frac {(\al_i\vert \al_i)}{2\cardnum{\cO(i)}} \sum_{p\in \cO(i)} a_{pi}
=\frac {(\al_i\vert \al_i)}{2\cardnum{\cO(i)}} (3-s_i) =  \frac {(\al_i\vert \al_i)}{\cardnum{\cO(i)}} \frac 1{s_i}.
\end{equation}
using Lemma \ref{lem:si2}.  This implies (a).  Also, dividing \eqref{eq:GCM1} by \eqref{eq:GCM2}, we obtain (b).
Finally, setting $\al = \al_j$ in (b) yields (c).
\end{proof}

\begin{pgraph}  Let  $\brvA = (\brva_{ij})_{i,j\in \brvI} \in \Mat_{\cardnum{\brvI}}(\bbQ)$,
where
\begin{equation}
\label{eq:brvA1}
\brva_{ij} = 2 \frac {(\mu_i \vert \mu_j)}{(\mu_i\vert \mu_i)}
\end{equation}
for $i,j\in \brvI$.  By Proposition \ref{prop:cartanbeta}(c), we have
\begin{equation}
\label{eq:brvA2}
\brva_{ij} = s_i \sum_{p\in \cO(i)} a_{pj} \in  s_i \bbZ.
\end{equation}
for $i,j\in \brvI$, so in particular  $\brvA \in \Mat_{\cardnum{\brvI}}(\bbZ)$.
\end{pgraph}

The next result is a special case of a result proved in
\cite{FSS} about diagram automorphisms
of symmetrizable GCM's  (see also
\cite[Prop.III.3.3]{Bau}).

\begin{proposition}
\label{prop:brvA} $\brvA$ is an affine GCM.
\end{proposition}

\begin{proof}  In \cite{FSS}, $s_i$ and $\brva_{ij}$ are defined by \eqref{eq:si2}
and \eqref{eq:brvA2} respectively.
With those definitions the proposition is proved in \cite[\S 2.3]{FSS}.
\end{proof}

\begin{remark}
\label{rem:folding}
It follows from
\eqref{eq:brvA2} that the Dynkin diagram for $\brvA$ is obtained
from the diagram for $A$ as follows.  If $i,j\in \brvI$,
\begin{equation}
\label{eq:brvA3}
\parbox{1.6in}{Multiplicity of the arrow\\
from $\beta_j$ to $\beta_i$ in the \\
diagram for $\brvA$}
=
s_i \sum_{p\in \cO(i)}\
\parbox{2in}{number of arrows from $\mu_j$ to $\mu_p$,\\ including multiplicity,
in the\\
diagram for $A$.}
\end{equation}
(If  the sum on the right is zero, there is no arrow from $\beta_j$ to $\beta_i$.)  See
\pref{ex:nt} for an example.
\end{remark}

\begin{remark}
\label{rem:FSS}  The type of the matrix $\brvA$  for each $A$ and each $\sg$ (up to conjugacy in $\Aut(A)$)
is recorded in \cite[List II, \S III.4]{Bau} and in \cite[Table 2.24]{FSS}.
\end{remark}

\begin{pgraph}
\label{pgraph:Weyl}
If $\al\in V^\times$, we define the \emph{orthogonal reflection $r_\al : V \to V$ along $\al$} as usual as
\[r_\al(\beta) = \beta - 2 \frac{(\al\vert\beta)}{(\al\vert\al)} \al\]
for $\beta\in V$.  Then $r_\al^2 = 1$ and $r_\al$ is in the orthogonal group of the form $\form$ on $V$.
Let
\[\brvW = \langle r_{\mu_i} \suchthat i\in \ttI\rangle \le \GL(V).\]
Note that each $r_{\mu_i}$ fixes $(1-\sg)(V)$ pointwise, so we can, by restriction, identify
$\brvW$  with a subgroup of $\GL(V^\sg)$.
\end{pgraph}

\begin{pgraph}
\label{pgraph:breveDl}  Let $\brvDl$ be the set of roots (including 0)  of the affine Kac-Moody Lie algebra $\fg(\brvA)$, and let
$\breve V = \spann_\bbQ(\brvDl)$.  It follows from Proposition \ref{prop:brvA} and \eqref{eq:brvA1} that we can
identify $V^\sg$ with $\breve V$ so that $\set{\mu_i}_{i\in \brvI}$ is the standard base for $\brvDl$ and the
form $\form$ on $V^\sg$ is a nonzero rational multiple of the standard form on $\breve V$.
We make this identification from this point on.  Then, $\brvW$ is the Weyl group
of $\brvDl$.  Hence, by \eqref{eq:realaff} applied to $\fg(\brvA)$, we have
\begin{equation}
\label{eq:breveDl}
\brvDl^\times =  \cup_{i\in \brvI} \brvW \mu_i.
\end{equation}
\end{pgraph}

\begin{pgraph}
In \cite[pp.~33--37]{Bau}, J.~Bausch  proved the next proposition in the case when $\base = \bbC$
by constructing a subalgebra
of $\fg^\sg$ that is isomorphic to $\fg(\brvA)$ and then
using  a characterization \cite[Prop.~5.8(a)]{K2} of the Tits cone of $\fg(\brvA)$.  The result for arbitrary
$\base$ can likely be proved along the same lines.  Instead, for the convenience
of the reader, we present a proof by induction on the height of an element of $\pi_\sg(\Delta)$.
\end{pgraph}

\begin{proposition}\
\label{prop:calcpiDl} If $s_i = 1$ for $i\in \brvI$, then
\begin{equation}
\label{eq:piDlc}
\textstyle
\pi_\sg(\Dl)^\times = \brvDl^\times.
\end{equation}
\end{proposition}
\begin{proof}  We first show as in \cite{Bau} that
\begin{equation}
\label{eq:Weylinv1}
\brvW(\pi_\sg(\Dl)) \subseteq \pi_\sg(\Dl).
\end{equation}
To prove this, we must show that $r_{\mu_i}(\pi_\sg(\Dl)) \subseteq \pi_\sg(\Dl)$ for $i\in \brvI$.
Indeed, since  $s_i = 1$, the refections $r_{\al_p}$, $p\in \cO(i)$,  commute.
Set $w = \prod_{p\in \cO(i)} r_{\al_p}\in W$.  Then, for $\al\in \Dl$, we have, using Proposition \ref{prop:cartanbeta}(b),
\begin{align*}
%\label{eq:Weylclosure}
r_{\mu_i}(\pi_\sg(\al)) &=  \pi_\sg(\al) - 2 \frac{(\mu_i\vert\pi_\sg(\al))}{(\mu_i\vert\mu_i)} \mu_i
=\textstyle  \pi_\sg(\al) - 2  \left(\sum_{p\in \cO(i)}  \frac{(\al_p\vert\al)}{(\al_p\vert\al_p)}\right) \mu_i \\
&=\textstyle \pi_\sg \left(\al - \sum_{p\in \cO(i)}  2\frac{(\al_p\vert\al)}{(\al_p\vert\al_p)}\al_p \right) = \pi_\sg(w(\al)) \in \pi_\sg(\Dl).
\end{align*}
This proves \eqref{eq:Weylinv1}; and, since reflections preserve the form $\form$, it follows that
\begin{equation}
\label{eq:Weylinv2}
\brvW(\pi_\sg(\Dl)^\times) \subseteq \pi_\sg(\Dl)^\times.
\end{equation}

The inclusion ``$\supseteq$'' in \eqref{eq:piDlc} now follows from
\eqref{eq:breveDl} and \eqref{eq:Weylinv2}.
For the inclusion ``$\subseteq$'', observe first that any nonzero element $\nu$ of $\pi_\sg(\Dl)$ can be written uniquely in the form
$\nu = \sum_{i\in \brvI} n_i\mu_i$, where the $n_i$ are integers which are all either
nonnegative or nonpositive.  We say that $\nu$ is positive or negative accordingly,
and we define the height of $\nu$ to be $\sum_{i\in \brvI} n_i$.
We must show that
\begin{equation*}
\label{eq:inductpi}
\textstyle \nu\in \pi_\sg(\Dl)^\times \implies \nu \in  \brvDl^\times.
\end{equation*}
To do this, we can assume that $\nu$ is positive and induct on the height of $\nu$, the case of height $1$
being clear.  Write $\nu = \sum_{i\in \brvI} n_i\mu_i$, where each $n_i$ is nonnegative.  Then, since
$(\nu\vert \nu) > 0$, we have $(\nu\vert \mu_i) > 0$ for some $i\in \brvI$ with $n_i > 0$.
If $r_{\mu_i}(\nu)$ is positive, we're done by induction.  So we can assume that $r_{\mu_i}(\nu)$ is negative.
Then $\nu = q \mu_i$, where $q$ is a positive integer.  Thus,  we have $\nu = \pi_\sg(\al)$, where
$\al =  \sum_{p\in \cO(i)} m_p \al_p \in \Dl$ and $\sum_{p\in \cO(i)} m_p = q$. Since $s_i = 1$
for $i\in \brvI$, this implies that
$q=1$.
\end{proof}

\subsection{Calculating the relative type}
\label{subsec:quotientpi}\

\begin{theorem}
\label{thm:relativecalc}  Suppose  that $\fg = \fg(A)$ is an affine Kac-Moody Lie
algebra and $\sg\in\Aut(A)$ is not transitive. If $s_i = 1$ for all $i\in \brvI$ then
the relative type of $\Lp(\fgb,\sg)$
is equal to the quotient  type of $\fg(\brvA)$, where $\brvA$ is obtained from
$A$ using \eqref{eq:brvA2}; whereas
if $s_i = 2$ for some $i\in \brvI$ then
the relative type of $\Lp(\fgb,\sg)$ is equal to
$\type{BC}{\cardnum{\brvI}-1}$.
(Here the integers $s_i$, $i\in \brvI$, are
defined in \ref{pgraph:sidef}.)
\end{theorem}
\begin{proof}  By Theorem \ref{thm:relativent}, the relative type of  $\Lp(\fgb,\sg)$
is the type of the finite root system $\piDlb$.

If  $s_i = 1$ for all $i\in \brvI$,  then  by
\eqref{eq:piDlc} we have  $\pi_\sg(\Dl)^\times = \brvDl^\times$, and hence
\[\overline{\pi_\sg(\Delta)} =
\overline{\pi_\sg(\Delta)}^\times \cup \set{0} =
\overline{\pi_\sg(\Delta)^\times} \cup \set{0} =
\overline{ \brvDl^\times} \cup \set{0} =
\overline{ \brvDl}^\times \cup \set{0} =
\overline{ \brvDl}.
\]

Suppose finally that  $s_i = 2$ for some $i\in \brvI$.
Then,
$\cO(i) = \set{\al_p,\al_q}$ , where
$p,q\in \brvI$, $p\ne q$ and $a_{pq} = a_{qp} = -1$.
So $\al_p+\al_q\in \Dl$, and hence $2\mu_i = \pi_\sg(\al_p+\al_q) \in \pi_\sg(\Dl)$.
Thus $\mu_i$ and $2\mu_i$ are elements
of $\pi_\sg(\Dl)^\times$.  So $\overline{\mu_i}$ and $2\overline{\mu_i}$ are nonzero elements of   $\piDlb$. Therefore
$\piDlb$  is a nonreduced irreducible finite root system  in $\overline{V^\sg}$.  Thus, by the classification of
irreducible finite root systems \cite[Chap. VI, \S 4, no.~14]{Bo2},
$\piDlb$ has type $\type{BC}{r}$, where $r = \dim{\overline{V^\sg}} = \cardnum{\brvI}-1$,
using  \pref{pgraph:dimVsgb}.
\end{proof}

%Dynkin diagrams
%Prepare for diagrams
\newcommand\node{\circle*{1}}

%%*******************************
%%Diagram for $D_5^1$
\newcommand\Dfiveone{%
\parbox{1.3in}
{
\setlength{\unitlength}{3pt}
\begin{picture}(25,20)(0,0)\thicklines
%%Lines
\put(4,4){\line(1,1){6}}
\put(4,16){\line(1,-1){6}}
\put(10,10){\line(1,0){8}}
\put(18,10){\line(1,1){6}}
\put(18,10){\line(1,-1){6}}
%%Points & labels
\put(4,4){\node} \put(-1,3){$\al_0$}
\put(4,16){\node} \put(-1,15){$\al_1$}
\put(10,10){\node} \put(9.5,6.5){$\al_2$}
\put(18,10){\node}  \put(16.5,6.5){$\al_3$}
\put(24,16){\node} \put(26,15){$\al_4$} % 26 = 20 + \dd
\put(24,4){\node}  \put(26,3){$\al_5$}
\end{picture}
}
}%

%%*******************************
%%Diagram for $D_4^2$
\newcommand\Dfourtwo{%
\parbox[c]{1.3in}
{
\setlength{\unitlength}{3pt}
\begin{picture}(25,12)(0,1)\thicklines
%%Lines
\put(2,7){\line(1,0){8}}
\put(2,5.7){$\boldsymbol<$}
\put(2,6.4){\line(1,0){8}}
\put(10,6.7){\line(1,0){8}}
\put(18,7){\line(1,0){8}}
\put(23.6,5.7){$\boldsymbol>$}
\put(18,6.4){\line(1,0){8}}
%%Points & labels
\put(2,6.7){\node} \put(2,3){$\mu_0$}
\put(10,6.7){\node} \put(9.5,3){$\mu_2$}
\put(18,6.7){\node}  \put(17.5,3){$\mu_3$}
\put(26,6.7){\node} \put(25,3){$\mu_4$} % 26 = 20 + \dd
\end{picture}
}
}%

\begin{example}
\label{ex:nt}
In this example let $\fg = \fg(A)$, where $A$ has type $\typeaff{D}{5}{1}$  with Dynkin
diagram
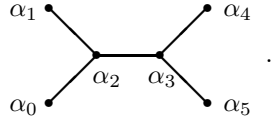
\[\Dfiveone.\]
If $\sg\in \Aut(A)$, we know  from Corollary \ref{cor:KacReal4} that
the absolute type of $\Lp(\fgb,\sg)\in \bbM_2$  is
$\type{D}{5}$.  We now
use Theorem \ref{thm:relativecalc} to calculate  the relative type of $\Lp(\fgb,\sg)$ in two cases.

(a)  Suppose first that $\sg = (0,1)(4,5) \in \Aut(A)$.   Then
$\brvI = \set{0,2,3,4}$ and
\begin{equation}
\label{eq:orbitex}
\cO(0) = \set{0,1},\quad \cO(2) = \set{2},\quad \cO(3) = \set{3},\quad \cO(4) = \set{4,5}.
\end{equation}
Since $a_{01} = 0$, we have $s_0 = 1$, and similarly $s_2 = s_3 = s_4 = 1$.
We now use Remark \ref{rem:folding} to calculate the Dynkin diagram for
$\brvA$.  We get
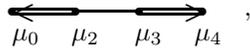
\[\Dfourtwo,\]
which is the diagram for the affine matrix of type $\typeaff{D}{4}{2}$.  So,
by Table \ref{tab:fq}, the  quotient type of  $\fg(\brvA)$ is  $\type{B}{3}$; hence,
by Theorem \ref{thm:relativecalc},
$\Lp(\fgb,\sg)$ has relative type $\type{B}{3}$.

(b)  Suppose next that $\sg = (0,5)(1,4)(2,3) \in \Aut(A)$.   Then
$\brvI = \set{0,1,2}$, and
\begin{equation*}
\cO(0) = \set{0,5},\quad \cO(1) = \set{1,4},\quad \cO(2) = \set{2,3}.
\end{equation*}
Since $a_{23} = a_{32} = -1$, we have $s_2 = 2$.
Thus, by Theorem \ref{thm:relativecalc},
$\Lp(\fgb,\sg)$ has relative type $\type{BC}{2}$.
\end{example}

\section{Classification of algebras in $\bbM_2$}
\label{sec:class}

In this section, we obtain a classification of the algebras in $\bbM_2$ up to isomorphism.

\subsection{Tables of relative types}
\label{subsec:tabreltype}

\begin{pgraph}
If $A$ is an affine GCM and $\sg\in \Aut(A)$,
the relative type of $\Lp(\overline{\fg(A)'},\sg)$ can be calculated
using the method  described in  Example \ref{ex:nt} when $\sg$ is not transitive
and using  Corollary \ref{cor:rot}   when $\sg$ is transitive.
In Tables \ref{tab:reltypeu}  and  \ref{tab:reltypet}  below, we record  this
relative type for all $A$ up to isomorphism and all $\sg \in \Aut(A)$ up to conjugacy.
Table \ref{tab:reltypeu} covers the cases when $\fg(A)$ is untwisted, whereas
Table \ref{tab:reltypet} covers the twisted cases.
Our enumeration of the fundamental roots of the root system corresponding to $A$ follows
\cite[\S 4.8]{K2}.

The  first column of the tables contains the type $\typeaff{X}{\rkfin}{m}$
of $A$; and the second column contains $\sg$ (when $\sg$ is uniquely determined up to conjugacy in $\Aut(A)$
by its order, we simply list its order).
The third column contains a label that we assign to the algebra
$\Lp(\overline{\fg(A)'},\sg)$ for ease of reference;
and the fourth column contains  the relative type of $\Lp(\overline{\fg(A)'},\sg)$.
When the entry in column 4 depends on the parity of $\rkfin$,
the first   expression in the column applies when $\rkfin$ is even and the second when $\rkfin$ is odd.

The labels in column 3 were chosen to indicate the method of construction of the algebras as iterated loop algebras.
For example, we used the labels
$\typemult{D}{\rkfin}{1}{2a}$, $\typemult{D}{\rkfin}{1}{2b}$ and $\typemult{D}{\rkfin}{1}{2c}$ in lines
9, 10 and 11 of Table $\ref{tab:reltypeu}$, since these algebras are constructed
starting from a finite dimensional Lie algebra of type $\type{D}{\rkfin}$ using a loop construction
relative  to an automorphism of order 1 followed by a loop construction relative to an automorphism of order 2.
The one exception to this scheme is the label $\typemult{A}{\rkfin}{1}{\rot(q)}$ in Line 1 of Table
$\ref{tab:reltypeu}$, where we needed to convey more information with the label.

If  $A$ and $\sg$ are chosen from columns 1 and 2 of Tables $\ref{tab:reltypeu}$
or $\ref{tab:reltypet}$, \emph{we will often denote the algebra
$\Lp(\overline{\fg(A)'},\sg)$ by the corresponding label from  column  3}.
\end{pgraph}

\begin{table}
\renewcommand{\arraystretch}{1.3}
\begin{tabular}[ht]
{|c |   c | c| c | c |}
\whline
\small
$A$ & $\sg$
& Label
& Relative type
\\

\whline
%%%%%%%%Absolute Type A%%%%%%%%%%

\rule{0ex}{3ex} $\typeaff{A}{\rkfin}{1}$, $\rkfin \ge 1$
& $(0,1, \dots, \rkfin)^q,\ 0\le q\le \lfloor\frac{\rkfin+1}2\rfloor$
& $\typemult{A}{\rkfin}{1}{\rot(q)}$
& $\type{A}{\gcd(q,\rkfin+1)-1}$
\\

\cline{2-4}
& \rule{0ex}{3ex} $\Pi_{i=1}^{\lfloor \frac \rkfin 2\rfloor}(i,\rkfin+1-i)$, $\rkfin \ge 2$
& $\typemult{A}{\rkfin}{1}{2a}$
&  $\type{BC}{\frac \rkfin 2}$\slashspace$\type{C}{\frac {\rkfin+1}2}$
\\

\cline{2-4}
& \rule{0ex}{3ex} $\Pi_{i=0}^{\frac {\rkfin-1} 2}(i,\rkfin-i)$, $\rkfin$ odd $\ge 3$
& $\typemult{A}{\rkfin}{1}{2b}$
&  $\type{BC}{\frac {\rkfin-1}2}$
\\

\whline
%%%%%%%%Absolute Type B%%%%%%%%%%

$\typeaff{B}{\rkfin}{1}$, $\rkfin \ge 3$
& $(1)$
& $\typemult{B}{\rkfin}{1}{1}$
& $\type{B}{\rkfin}$
\\

\cline{2-4}
& $\order{\sg} = 2$
& $\typemult{B}{\rkfin}{1}{2}$
& $\type{B}{\rkfin-1}$
\\

\whline
%%%%%%%%Absolute Type C%%%%%%%%%%

$\typeaff{C}{\rkfin}{1}$, $\rkfin \ge 2$
& $(1)$
& $\typemult{C}{\rkfin}{1}{1}$
& $\type{C}{\rkfin}$
\\

\cline{2-4}
& $\order{\sg} = 2$
& $\typemult{C}{\rkfin}{1}{2}$
&  $\type{C}{\frac \rkfin 2}$\slashspace $\type{BC}{\frac {\rkfin-1}2}$
\\

\whline
%%%%%%%%Absolute Type D%%%%%%%%%%

$\typeaff{D}{\rkfin}{1}$, $\rkfin \ge 4$
& $(1)$
& $\typemult{D}{\rkfin}{1}{1}$
& $\type{D}{\rkfin}$
\\

\cline{2-4}
& $(\rkfin-1,\rkfin)$
& $\typemult{D}{\rkfin}{1}{2a}$
& $\type{B}{\rkfin-1}$
\\

\cline{2-4}
& $(0,1)(\rkfin-1,\rkfin)$
& $\typemult{D}{\rkfin}{1}{2b}$
& $\type{B}{\rkfin-2}$
\\

\cline{2-4}
&\rule{0ex}{3.5ex}$\Pi_{i=0}^{\lfloor \frac {\rkfin -1} 2 \rfloor} (i,\rkfin-i)$,  $\rkfin \ge 5$
& $\typemult{D}{\rkfin}{1}{2c}$
& $\type{C}{\frac \rkfin 2}$\slashspace $\type{BC}{\frac {\rkfin-1}2}$
\\

\cline{2-4}
& $\order{\sg} = 3$, $\rkfin = 4$
& $\typemult{D}{4}{1}{3}$
& $\type{G}{2}$
\\

\cline{2-4}
& $\order{\sg} = 4$
& $\typemult{D}{\rkfin}{1}{4}$
& $\type{BC}{\lfloor \frac{\rkfin-2}2 \rfloor }$
\\

\whline
%%%%%%%%Absolute Type E6%%%%%%%%%%

$\typeaff{E}{6}{1}$
& $(1)$
& $\typemult{E}{6}{1}{1}$
& $\type{E}{6}$
\\

\cline{2-4}
& $\order{\sg} = 2$
& $\typemult{E}{6}{1}{2}$
& $\type{F}{4}$
\\

\cline{2-4}
& $\order{\sg} = 3$
& $\typemult{E}{6}{1}{3}$
& $\type{G}{2}$
\\

\whline
%%%%%%%%Absolute Type E7%%%%%%%%%%

$\typeaff{E}{7}{1}$
& $(1)$
& $\typemult{E}{7}{1}{1}$
& $\type{E}{7}$
\\

\cline{2-4}
& $\order{\sg} = 2$
& $\typemult{E}{7}{1}{2}$
& $\type{F}{4}$
\\

\whline
%%%%%%%%Absolute Type E8%%%%%%%%%%

$\typeaff{E}{8}{1}$
& $(1)$
& $\typemult{E}{8}{1}{1}$
& $\type{E}{8}$
\\

\whline
%%%%%%%%Absolute Type F4%%%%%%%%%%

$\typeaff{F}{4}{1}$
& $(1)$
& $\typemult{F}{4}{1}{1}$
& $\type{F}{4}$
\\

\whline
%%%%%%%%Absolute Type G2%%%%%%%%%%

$\typeaff{G}{2}{1}$
& $(1)$
%%% & $\typeaff{G}{2}{1}$
& $\typemult{G}{2}{1}{1}$
& $\type{G}{2}$
\\

\whline
\end{tabular}
\medskip
\caption{Relative type of $\Lp(\overline{\fg(A)'},\sg)$ for untwisted $A$}
\label{tab:reltypeu}
\end{table}

%\clearpage

\begin{table}
\renewcommand{\arraystretch}{1.3}
\small
\begin{tabular}[ht]
{|c |   c | c| c | c |}
\whline

$A$
& $\sg$
& Label
& Relative type
\\

\whline
%%%%%%%%Absolute Type A%%%%%%%%%%

$\typeaff{A}{\rkfin}{2}$, $\rkfin \ge 2, \rkfin \ne 3$
& $(1)$
& $\typemult{A}{\rkfin}{2}{1}$
& $\type{BC}{\frac \rkfin 2}$ \slashspace $\type{C}{\frac{\rkfin+1}2}$

\\

\cline{2-4}
& $\order{\sg} = 2$, $\rkfin$ odd
& $\typemult{A}{\rkfin}{2}{2}$
& $\type{BC}{\frac{\rkfin-1}2}$
\\

\whline
%%%%%%%%Absolute Type D%%%%%%%%%%

$\typeaff{D}{\rkfin}{2}$, $\rkfin\ge 3$
& $(1)$
& $\typemult{D}{\rkfin}{2}{1}$
& $\type{B}{\rkfin-1}$
\\

\cline{2-4}
& \rule{0ex}{4ex} $\Pi_{i=0}^{\lfloor\frac {\rkfin-2}2\rfloor}(i,\rkfin-1-i)$
& $\typemult{D}{\rkfin}{2}{2}$
&  $\type{BC}{\lfloor\frac{\rkfin-1}2\rfloor}$
\\

\whline
%%%%%%%%Absolute Type D4%%%%%%%%%%

$\typeaff{D}{4}{3}$
& $(1)$
& $\typemult{D}{4}{3}{1}$
& $\type{G}{2}$
\\

\whline
%%%%%%%%Absolute Type E6%%%%%%%%%%

$\typeaff{E}{6}{2}$
& $(1)$
& $\typemult{E}{6}{2}{1}$
&$\type{F}{4}$
\\

\whline
\end{tabular}
\medskip
\caption{Relative type of $\Lp(\overline{\fg(A)'},\sg)$ for twisted $A$}
\label{tab:reltypet}
\end{table}

\begin{remark}  It is important to recall that the absolute type
of $\Lp(\overline{\fg(A)'},\sg)$ can also be read from Tables
\ref{tab:reltypeu} and \ref{tab:reltypet}.  Indeed if $A$ has type
$\typeaff{X}{\rkfin}{m}$, then $\Lp(\overline{\fg(A)'},\sg)$ has absolute type
$\type{X}{\rkfin}$ by Corollary \ref{cor:KacReal4}.
\end{remark}

\begin{pgraph}
Note that  by \eqref{eq:rot3} the algebras occurring in
row 1 of Table \ref{tab:reltypeu} have matrix realizations:
\begin{equation}
\label{eq:TypsAAreal}
\typemult{A}{\rkfin}{1}{\rot(q)} \simeq \spl_{\gcd(q,\rkfin+1)}(Q(\zeta_{\rkfin+1}^{\iota_{\rkfin+1}(\bar q)})),
\end{equation}
$0\le q\le \lfloor\frac{\rkfin+1}2\rfloor$. (Recall that if $q=0$, $\gcd(q,\rkfin+1)$ is interpreted as $\rkfin+1$.)
\end{pgraph}

The following proposition follows from Theorem \ref{thm:char} and Table
\ref{tab:reltypeu}.

\begin{proposition}
\label{prop:AA} Let $\cL\in \bbM_2$.  Then $\cL$ satisfies condition (AA)  (described  in the introduction) if and only if
$\cL \simeq \typemult{A}{\rkfin}{1}{\rot(q)}$ for some $\rkfin\ge 1$ and some~$0\le q \le \lfloor \frac {\rkfin+1}2 \rfloor$.
\end{proposition}

\subsection{Isomorphism}\
\label{subsec:isomL2}

We know from Theorem \ref{thm:char} that the algebras in $\bbM_2$ are,  up to isomorphism, the algebras
of the form $\Lp(\fgb,\sg)$, where $\fg$ is an untwisted affine algebra
and $\sg$ is a diagram automorphism of $\fgb$.  The next theorem
give necessary and sufficient conditions for two algebras  of this form to be isomorphic.

\begin{theorem}
\label{thm:isom} \

\emph{(a)} If $\fg_i$ is an untwisted affine Kac-Moody Lie algebra
and $\sg_i$ is a diagram automorphism of $\overline{\fg'_i}$, $i=1,2$,
 then
$\Lp(\overline{\fg'_1},\sg_1) \isom\Lp(\overline{\fg'_2},\sg_2)$ implies that $\fg_1 \simeq \fg_2$.

\emph{(b)}  If $\fg$ is an affine Kac-Moody Lie algebra (twisted or untwisted)
and $\sg_i$ is a diagram automorphism of $\fgb$, $i=1,2$,
then $\Lp(\fgb,\sg_1) \isom\Lp(\fgb,\sg_2)$
if and only if
$\sg_1$ and $\sg_2$ are conjugate in the group of diagram automorphisms of $\fgb$.
\end{theorem}

\begin{proof} (a) follows from Corollary \ref{cor:KacReal4}, so it remains to prove (b).
The implication ``$\Leftarrow$'' in (b) follows from  Lemma \ref{lem:changemult}(a).

To prove the implication ``$\Rightarrow$'' in (b), suppose that $\cL_1 \simeq \cL_2$, where
$\cL_i = \Lp(\fgb,\sg_i)$ for $i=1,2$.  Now $\fg = \fg(A)$, where $A$ is one of the
affine matrices in column 1 of Tables \ref{tab:reltypeu} or \ref{tab:reltypet}.  Further,
replacing each $\sg_i$ by a conjugate in $\Aut(A)$, we can, by  Lemma \ref{lem:changemult}(a), assume that
$\sg_1$ and $\sg_2$ are among the automorphisms listed in column 2 of Tables \ref{tab:reltypeu}
or \ref{tab:reltypet}.

Case 1:  Suppose that $\cL_1$, and hence also $\cL_2$, satisfies Condition (AA).
Then  by Proposition \ref{prop:AA}, $A$ has type $\typeaff{A}{\rkfin}{1}$ for some $\rkfin\ge 1$ and $\sg_i = (0,1,\dots,\rkfin)^{q_i}$,
where $0\le q_i \le  \lfloor \frac {\rkfin+1}2 \rfloor$, $i=1,2$.
So $\cL_i = \Lp(\fgb,(0,1,\dots,\rkfin)^{q_i})$ for $i=1,2$.  Thus,  since $\cL_1\simeq \cL_2$, we have
\[ \spl_{\gcd(q_1,\rkfin+1)}\big(Q(\zeta_{\rkfin+1}^{\iota_{\rkfin+1}(\bar q_1)})\big)
\simeq
 \spl_{\gcd(q_2,\rkfin+1)}\big(Q(\zeta_{\rkfin+1}^{\iota_{\rkfin+1}(\bar q_2)})\big)
\]
by \eqref{eq:rot3}.  Hence, by Theorem \ref{thm:isomquantum}, we have
$\iota_{\rkfin+1}(\bar q_1) = \pm \iota_{\rkfin+1}(\bar q_2)$ in $\bbZ_{\rkfin+1}$,
and therefore $\bar q_1 = \pm \bar q_2$.
So $q_1 = q_2$,  and $\sg_1 = \sg_2$.

Case 2:  Suppose that $\cL_1$, and hence also $\cL_2$, does not satisfy Condition (AA).
If $\fg$ is untwisted, then $\sg_1$
and $\sg_2$ appear in one of the rows after row 1 of Table~\ref{tab:reltypeu}.
Since the relative type of $\cL_1$ is the same as the relative type of $\cL_2$,
it follows checking Table \ref{tab:reltypeu} that $\sg_1$ and $\sg_2$ appear in the same row.
Thus, $\sg_1 = \sg_2$.  The argument in the twisted case is the same using Table
\ref{tab:reltypet}.
\end{proof}

\begin{pgraph}
As a corollary of this theorem, we  next prove that the converse in Theorem \ref{thm:erase}(c) is valid
for an affine algebra $\fg = \fg(A)$ and  automorphisms of first kind.  Before
doing so, we note that by the classification of affine matrices, $\Aut(A)$ is one
of the following: a cyclic group of order 1 or 2, a dihedral group of order $\ge 6$, or a symmetric group $S_3$ or $S_4$.
Hence, any element of $\Aut(A)$ is conjugate to its inverse, so the relation
$\sim$ in $\Aut(A)$ described prior to  Theorem \ref{thm:erase}  is simply conjugacy.
\end{pgraph}

\begin{corollary} \label{cor:class1}  Suppose that
$\fg = \fg(A)$ is an affine Kac-Moody Lie algebra and
$\sg_1$ and $\sg_2$
are finite order automorphisms of first kind of  $\fgb$.
Then $\Lp(\fgb,\sg_1)\isom \Lp(\fgb,\sg_2)$ if and only if
$\bar p(\sg_1)$ is conjugate to $\bar p(\sg_2)$ in $\Aut(A)$.
\end{corollary}

\begin{proof}  By Theorem \ref{thm:erase}(c), we only need to prove the implication ``$\Rightarrow$''.
Suppose that $\Lp(\fgb,\sg_1)\isom \Lp(\fgb,\sg_2)$.   By Theorem \ref{thm:erase}(c), we
have $\Lp(\fgb,\sg_i)\isom \Lp(\fgb,\bar p(\sg_i))$, for $i=1,2$.  Hence, replacing
$\sg_i$ by $\bar p(\sg_i)$, we can assume that
$\sg_i\in \Aut(A)$, $i=1,2$.  The result now follows from Theorem \ref{thm:isom}(b).
\end{proof}

\subsection{Classification}\
\label{subsec:classL2}

Our second  main theorem gives a classification of the algebras in $\bbM_2$.

\begin{theorem}
\label{thm:class}  If
$A$ and $\sg$ are chosen from Columns 1 and 2 of
Table \ref{tab:reltypeu}, then  $\Lp(\overline{\fg(A)'},\sg)$ is in $\bbM_2$.  Conversely,
any algebra in $\bbM_2$ is isomorphic to exactly one such algebra.
\end{theorem}

\begin{proof}  This follows from Theorems \ref{thm:char} and Theorems \ref{thm:isom}.
\end{proof}

\begin{pgraph}  Using the labels in column 3 of Table \ref{tab:reltypeu}, Theorem \ref{thm:class}
states that a   nonredundant list of   all algebras in  $\bbM_2$ up to isomorphism is
$\typemult{A}{\rkfin}{1}{\rot(q)},\dots,\typemult{G}{2}{1}{1}$.
\end{pgraph}

The following corollary follows from  Proposition \ref{prop:AA}, Theorem \ref{thm:class} and Table \ref{tab:reltypeu}.

\begin{corollary}
\label{cor:class2}
Two algebras in $\bbM_2$ that do not satisfy Condition (AA) are isomorphic if and only
if they have the same absolute type and the same relative type.
\end{corollary}

\begin{pgraph}
\label{pgraph:twistuntwist}
Any algebra appearing in  Table \ref{tab:reltypet}
is in $\bbM_2$ by Theorem \ref{thm:char}, and hence it is isomorphic to an algebra in Table
\ref{tab:reltypeu}.  But any such algebra does not satisfy Condition (AA) (by Table
\ref{tab:reltypet}), so to match it with an algebra in Table \ref{tab:reltypeu},
we just have to match the relative and absolute types.  In this way, we see that
\begin{gather*}
\typemult{A}{\rkfin}{2}{1} \simeq \typemult{A}{\rkfin}{1}{2a} \quad \text{ if } \rkfin\ge 2, \rkfin\ne 3,\\
\typemult{A}{\rkfin}{2}{2} \simeq \typemult{A}{\rkfin}{1}{2b} \quad \text{ if } \rkfin\ge 5, \rkfin \text{ odd},\\
\typemult{D}{\rkfin}{2}{1} \simeq
\left\{
  \begin{array}{ll}
    \typemult{D}{\rkfin}{1}{2a}, & \hbox{if $\rkfin\ge 4$;} \\
    \typemult{A}{3}{1}{2a}, & \hbox{if $\rkfin=3$,}
  \end{array}
\right.\\
\typemult{D}{\rkfin}{2}{2} \simeq
\left\{
  \begin{array}{ll}
    \typemult{D}{\rkfin}{1}{4}, & \hbox{if $\rkfin\ge 4$, $\rkfin$ even;} \\
    \typemult{D}{\rkfin}{1}{2c}, & \hbox{if $\rkfin\ge 5$, $\rkfin$ odd;} \\
    \typemult{A}{3}{1}{2b}, & \hbox{if $\rkfin=3$,}
  \end{array}
\right.\\
\typemult{D}{4}{3}{1} \simeq \typemult{D}{4}{1}{3} \andd
\typemult{E}{6}{2}{1} \simeq \typemult{E}{6}{1}{2}.
\end{gather*}
The  isomorphisms $\typemult{A}{\rkfin}{2}{1} \simeq \typemult{A}{\rkfin}{1}{2a}$,
$\typemult{D}{\rkfin}{2}{1} \simeq  \typemult{D}{\rkfin}{1}{2a}$,
$\typemult{D}{4}{3}{1} \simeq \typemult{D}{4}{1}{3}$ and
$\typemult{E}{6}{2}{1} \simeq \typemult{E}{6}{1}{2}$
are of course not surprising and can be proved directly without our classification
results.

\end{pgraph}

\section{Links to some related work.}
\label{sec:links}

In this section, we  consider two invariants associated to an algebra $\cL$ in $\bbM_2$, the index and the  Saito
extended affine
root system, thereby linking our results to the work of several other  researchers.

\subsection{The index of a prime perfect fgc Lie algebra}\
\label{subsec:TitsI}

In order to define the index of a prime perfect fgc Lie algebra,
we first discuss the corresponding notion for algebraic groups.

\begin{pgraph}
\label{pgraph:indexgroup} Suppose that $F$ is a field, $F_s$ is a separable closure of $F$, and
$\Gamma = \Gal(F_s/F)$.

(a)\  Starting from a semisimple algebraic group $G$ defined over $F$,
one can construct a triple $(\Pi,\Pi_0,*)$, called the \emph{index} of $G$  \cite[\S2.3]{T}.
This triple, which is often also called the \emph{Tits index} of $G$,
consists of a finite Dynkin diagram $\Pi$, a
(possibly empty) subset $\Pi_0$ of $\Pi$, and a homomorphism $\sg \xrightarrow{*} \sg^*$
(called the $*$-action)
from $\Gamma$ into $\Aut(\Pi)$
satisfying $\sg^*(\Pi_0) = \Pi_0$ for $\sg\in \Gamma$.
(Here $\Aut(\Pi)$ is the group of diagram automorphisms of $\Pi$.)
We will not need to describe the construction of  $(\Pi,\Pi_0,*)$, but instead recall
it below in \pref{pgraph:indexsimple}(b) in the special case of interest here.

(b)\  If $G_i$ is a semisimple algebraic group defined over $F$ with index
$(\Pi_i,\Pi_{i0},*)$ for $i=1,2$, we say that
$(\Pi_1,\Pi_{10},*)$ and
$(\Pi_2,\Pi_{20},*)$ are \emph{isomorphic} if there is a diagram isomorphism
from $\Pi_1$ onto $\Pi_2$ which maps $\Pi_{10}$ onto
$\Pi_{20}$ and commutes with the $*$-actions.
As a special case of \cite[Prop.~2.6.3]{T}, we know that
if $G_1$ and $G_2$ are $F$-isomorphic then their indices are isomorphic.
In particular,  the index is independent up to isomorphism of the choices
made in its construction.
\end{pgraph}

\begin{pgraph}
\label{pgraph:indexlabel}
Besides isomorphism, there is a another notion of equivalence for indices which is useful.  To describe this,
suppose that for $i=1,2$, $F_i$ is a field, $(F_i)_s$ is a separable closure of $F_i$, $\Gamma_i = \Gal((F_i)_s/F_i)$, and
$(\Pi_i,\Pi_{i0},*)$ is the index for a semisimple algebraic group $G_i$ defined over $F_i$.  We say that the indices
$(\Pi_1,\Pi_{10},*)$ and $(\Pi_2,\Pi_{20},*)$ are \emph{similar}
if there is a diagram isomorphism $\lambda : \Pi_1 \to \Pi_2$ and a group isomorphism
$\mu : (\Gamma_1)^* \to (\Gamma_2)^*$ such that
$\lambda(\Pi_{10}) = \Pi_{20}$ and
$\lambda(g \alpha) = \mu(g) \lambda(\alpha)$ for $g\in (\Gamma_1)^*$
and $\alpha\in \Pi$.  Clearly, if $F_1 = F_2$,
isomorphic indices are similar.

In the famous Table 2 of \cite{T},  Tits listed and labelled all indices up to similarity
that can occur for some absolutely  simple algebraic group defined over some $F$.
We will use those labels  $\Tindex{1}{A}{n,r}{d},\dots, \Tindex{}{G}{2,2}{0}$
in Section \ref{subsec:TitsM2} below.
\end{pgraph}

\begin{pgraph}
\label{pgraph:indexsimple}
Suppose next that $\cL$ is a simple fgc Lie algebra.  Let $F = \Cd(\cL)$, which
is a field extension of $\base$.

(a) Recall from \pref{pgraph:centralsimple} that $\cL$ is a finite dimensional central simple Lie algebra
over $F$, and hence  $\bfAut^0(\cL)$ is a simple algebraic group over $F$, where
$\bfAut^0(\cL)$ denotes the connected component of the automorphism group $\bfAut(\cL)$ of $\cL$ over $F$.
We define the \emph{index} of $\cL$
to be the index of $\bfAut^0(\cL)$.

(b) Since $\characteristic(\base) = 0$, the index of $\cL$ can equivalently be defined purely
in terms of the structure of $\cL$.  For the convenience of the reader, we do this here following \cite[\S2.3]{T},
\cite[\S 4]{Sat}  and \cite{Se1} (in the Lie algebra case).
Let $\bF$ (= $F_s$) be an algebraic closure of $F$ with $\Gamma = \Gal(\bF/F)$, let $\cL_\bF = \cL\otimes_F \bF$, and identify
$\cL$ as usual as an $F$-subalgebra of $\cL_\bF$.
Let $\cT$ be a MAD $F$-subalgebra of $\cL$, and let
$\cH$ be a  Cartan  $F$-subalgebra of $\cL$ containing $\cT$  \cite[Prop.~I.1]{Se2}.
Then $\cH_\bF =  \cH\otimes_F \bF$ is a MAD $\bF$-subalgebra of $\cL_\bF$.
Let $\Delta$ be the  root system (including 0) of $\cL_\bF$ relative to $\cH_\bF$, in which case the restriction map   $\alpha \mapsto \alpha_\cT$ maps
$\Delta$ onto the root system of $\cL$ with respect to $\cT$.
Choose an additive linear order $>$ on the root lattice $Q(\Delta)$
that is compatible with restriction (that is if $\al_1,\al_2\in Q(\Delta)$,
$\al_1|_\cT = \al_2|_\cT$ and $\al_1>0$ then $\al_2>0$).
Let $\Pi$ be the base for $\Delta$ determined by $>$, and let $\Pi_0 = \set{\alpha\in \Pi : \alpha|_{\cT} = 0}$.
Finally,
$\Gamma$ acts on $\cH_\bF$ (trivially on $\cH$ and naturally on $\bF$), and this action
transfers as usual to an action of $\Gamma$ on $\cH_\bF^* = \Hom_\bF(\cH_\bF,\bF)$ which stabilizes
$\Delta$.  For $\sg \in\Gamma$, we define $\sg^*\in \Aut(\Pi)$, by
$\sg^*(\alpha) = w_\sigma \sg(\alpha)$, where
$w_\sg$ is the element of the Weyl group of $\Delta$ such that
$w_\sigma \sg(\Pi) = \Pi$.  Then $(\Pi,\Pi_0,*)$
is the index of $\cL$%.

(c) If follows from the definition in (b) (or (a)) that the absolute type
of $\cL$ is the type of the Dynkin diagram $\Pi$.  (This type is used
as the ``base'' for the label of the index  of $\cL$.)
Moreover,  Tits described in \cite[\S 2.5.2]{T} an algorithm for
calculating the relative type of $\cL$ given its index up to similarity.
Hence, both the absolute and relative type of $\cL$ are
determined by its index up to similarity.  For example if $\cL$ has index
$\Tindex{}{E}{7,4}{9}$, then
$\cL$ has absolute type $\type{E}{7}$ and, using the algorithm,
relative type $\type{F}{4}$.
\end{pgraph}

\begin{definition}
\label{def:indexprime}
Suppose finally that $\cL$ is an prime perfect fgc Lie algebra.
 We know from Proposition \ref{prop:cclosure} that the central closure
$\tcL := \cL \ot_{\Cd(\cL)} \tCd(\cL)$
is a simple fgc Lie algebra.  We define
the \emph{index of $\cL$} to be the index of
$\tcL$.
\end{definition}

\begin{lemma}
\label{lem:indexinvariant}
Suppose that $\cL_1$ and $\cL_2$
are prime perfect fgc Lie algebras that are isomorphic (as $\base$-algebras). Then the indices of  $\cL_1$ and $\cL_2$ are similar.
\end{lemma}

\begin{proof}  We sketch this proof; the details are straightforward and left to the reader. As in the proof
of Lemma \ref{lem:typeinvariant}, we can assume that
$\cL_1$ and $\cL_2$ are simple.
Choose
$F_i$, $\bar F_i$ and $\Gamma_i$ for $\cL_i$ as in
\pref{pgraph:indexsimple}(b).
Let
$\varphi : \cL_1 \to \cL_2$ be an isomorphism.
Let $\gamma = \Cd(\varphi) : F_1 \to F_2$ which we extend
to an isomorphism $\bar\gamma : \bar F_1 \to \bar F_2$.
Choose $\cT_1$, $\cH_1$ and $>$ for $\cL_1$
as in \pref{pgraph:indexsimple}(b), and use these
to compute the index  $(\Pi_1,\Pi_{10},*)$ for $\cL_1$.
Next using $\varphi$ and $\bar\gamma$, we can transfer
$\cT_1$, $\cH_1$ and $>$ to corresponding
$\cT_2$, $\cH_2$ and $>$ for $\cL_2$,  and use these
to compute
$(\Pi_2,\Pi_{20},*)$ for $\cL_2$.  (This is permitted
by \pref{pgraph:indexgroup}(b).) Then we can show that
there exists a diagram isomorphism $\lambda : \Pi_1 \to \Pi_2$
and a group isomorphism $\nu : \Gamma_1 \to \Gamma_2$ (conjugation
by $\bar\gamma$) such that
$\lambda(\sigma^* \alpha) = \nu(\sigma)^* \lambda(\alpha)$
for $\sg\in \Gamma_1$, $\alpha\in \Pi_1$.  This clearly implies
similarity.
\end{proof}

\subsection{The index of algebras in $\bbM_2$}
\label{subsec:TitsM2}

\begin{pgraph}
If $\cL\in \bbM_n$, then $\cL$ is prime perfect and fgc by Proposition
\ref{prop:ppfgc}, and hence the index of $\cL$ is defined.
\end{pgraph}

\begin{theorem}
\label{thm:class3}
Two algebras in $\bbM_2$ that do not satisfy Condition (AA) are isomorphic if and only if they have similar indices.
\end{theorem}

\begin{proof} The implication ``$\implies$'' follows from
Lemma \ref{lem:indexinvariant}.  The converse follows from
\pref{pgraph:indexsimple}(c) and Corollary \ref{cor:class2}.
\end{proof}

\begin{pgraph}  Theorem \ref{thm:class3} provides a positive answer to Conjecture
6.4 in \cite{GP1}.\footnote{A partial proof of this conjecture for
nullity two algebras of classical absolute type of ``sufficiently large"
rank was given by   A.~Steinmetz-Zikesch in \cite{S-Z}. His proof is in fact rather explicit,
and the algebras in question are described in full detail. Another proof
of a more general version of the conjecture will appear in \cite{GP2}.}
\end{pgraph}

\begin{pgraph}
In rest of this section, we describe how to calculate the index of an algebra $\cL$
in $\bbM_2$.   We omit most of the details here, but rather sketch three examples that
will be enough for the interested reader to complete the details.

We summarize our results in Table \ref{tab:TitsEARSu}.
In that table, column 1 lists  our label from Table \ref{tab:reltypet}  for the algebra $\cL$,
and
column 2 lists the index of $\cL$.
(Column 3 contains information
that we will discuss in  Section \ref{subsec:SEARSM2}.)
As before, when an entry in columns 2 or 3 depends on the parity of $\rkfin$,
the first  expression in the column applies when $\rkfin$ is even and the second when $\rkfin$ is odd.
\end{pgraph}

\begin{table}
\renewcommand{\arraystretch}{1.3}
\begin{tabular}[ht]
{|c |   c | c| c | c |}
\whline
\small
%$A$
Label for $\cL$
%& Relative type
& Index
& SEARS \\
\whline
%%%%%%%%Absolute Type A%%%%%%%%%%

$\typemult{A}{\rkfin}{1}{\rot(q)},\  \rkfin \ge 1,$
%& $\type{A}{r},\ r =  \gcd(q,\rkfin+1)-1$
&\rule{0ex}{3.7ex} $\Tindex{1}{A}{k,r}{\left(\frac {k+1}{r+1}\right)},$
&$\sears{A}{r}{(1,1)},$
\\[-.5ex]

%%Extra line for condition (AA).
$0\le q\le \lfloor\frac{\rkfin+1}2\rfloor$
%& $\gcd(q,\rkfin+1)-1$
&$r =  \gcd(q,\rkfin+1)-1$
&$\text{if $r\ge 1$}$
\\

\hline
$\typemult{A}{\rkfin}{1}{2a},\ \rkfin \ge 2$
%&  $\type{BC}{\frac \rkfin 2} $\slashspace $\type{C}{\frac {\rkfin+1}2}$
&$\Tindex{2}{A}{k,\frac k2}{(1)}$\slashspace $\Tindex{2}{A}{k,\frac {k+1}{2}}{(1)}$
&$\sears{BC}{\frac k2}{(2,1)}$\slashspace$\sears{C}{\frac {k+1}2}{(1,2)}$
\\

\hline
$\typemult{A}{\rkfin}{1}{2b},\ \rkfin$ odd $\ge 3$
%&  $\type{BC}{\frac {\rkfin-1}2}$
&$\Tindex{2}{A}{k,\frac {k-1}{2}}{(1)}$
&$\sears{BC}{\frac {k-1}2}{(2,2)}(2)$
\\

\whline
%%%%%%%%Absolute Type B%%%%%%%%%%

$\typemult{B}{\rkfin}{1}{1},\ \rkfin \ge 3$
%& $\type{B}{\rkfin}$
&$\Tindex{}{B}{k,k}{}$
&$\sears{B}{k}{(1,1)}$
\\

\hline
$\typemult{B}{\rkfin}{1}{2},\ \rkfin \ge 3$
%& $\type{B}{\rkfin-1}$
&$\Tindex{}{B}{k,k-1}{}$
&$\sears{B}{k-1}{(2,2)*}$
\\

\whline
%%%%%%%%Absolute Type C%%%%%%%%%%

$\typemult{C}{\rkfin}{1}{1},\ \rkfin \ge 2$
%& $\type{C}{\rkfin}$
&$\Tindex{}{C}{k,k}{(1)}$
&$\sears{C}{k}{(1,1)}$
\\

\hline
$\typemult{C}{\rkfin}{1}{2},\ \rkfin \ge 2$
%&  $\type{C}{\frac \rkfin 2}$\slashspace $\type{BC}{\frac {\rkfin-1}2}$
&$\Tindex{}{C}{k,\frac k 2}{(2)}$ \slashspace $\Tindex{}{C}{k,\frac {k-1} 2}{(2)}$
&$\sears{C}{\frac k2}{(1,1)*}$\slashspace $\sears{BC}{\frac {\rkfin-1}2}{(1,1)}$
\\

\whline
%%%%%%%%Absolute Type D%%%%%%%%%%

$\typemult{D}{\rkfin}{1}{1},\ \rkfin \ge 4$
%& $\type{D}{\rkfin}$
&$\Tindex{1}{D}{k,k}{(1)}$
&$\sears{D}{k}{(1,1)}$
\\

\hline
$\typemult{D}{\rkfin}{1}{2a},\ \rkfin \ge 4$
%& $\type{B}{\rkfin-1}$
&$\Tindex{2}{D}{k,k-1}{(1)}$
&$\sears{B}{k-1}{(1,2)}$
\\

\hline
$\typemult{D}{\rkfin}{1}{2b},\ \rkfin \ge 4$
%& $\type{B}{\rkfin-2}$
&$\Tindex{1}{D}{k,k-2}{(1)}$
&$\sears{B}{k-2}{(2,2)}$
\\

\hline
$\typemult{D}{\rkfin}{1}{2c},\  \rkfin \ge 5$
%& $\type{C}{\frac \rkfin 2}$\slashspace $\type{BC}{\frac {\rkfin-1}2}$
&$\Tindex{1}{D}{k,\frac k2}{(2)}$\slashspace$\Tindex{2}{D}{k,\frac {k-1}2}{(2)}$
&$\sears{C}{\frac k2}{(2,2)}$\slashspace$\sears{BC}{\frac {k-1}2} {(2,2)}(1)$
\\

\hline
$\typemult{D}{4}{1}{3}$
%& $\type{G}{2}$
&$\Tindex{3}{D}{4,2}{2}$
&$\sears{G}{2}{(1,3)}$
\\

\hline
$\typemult{D}{\rkfin}{1}{4},\ \rkfin \ge 4$
%& $\type{BC}{\frac{\rkfin-2}2}$\slashspace$\type{BC}{\frac{\rkfin-3}2}$
&$\Tindex{2}{D}{k,\frac {k-2}2}{(2)}$\slashspace$\Tindex{1}{D}{k,\frac {k-3}2}{(2)}$
&$\sears{BC}{\frac {k-2}2}{(2,4)}$\slashspace $\sears{BC}{\frac{\rkfin-3}2}{(4,4)}$
\\

\whline
%%%%%%%%Absolute Type E6%%%%%%%%%%

$\typemult{E}{6}{1}{1}$
%& $\type{E}{6}$
&$\Tindex{1}{E}{6,6}{0}$
&$\sears{E}{6}{(1,1)}$
\\

\hline
$\typemult{E}{6}{1}{2}$
%& $\type{F}{4}$
&$\Tindex{2}{E}{6,4}{2}$
&$\sears{F}{4}{(1,2)}$
\\

\hline
$\typemult{E}{6}{1}{3}$
%& $\type{G}{2}$
&$\Tindex{1}{E}{6,2}{16}$
&$\sears{G}{2}{(3,3)}$
\\

\whline
%%%%%%%%Absolute Type E7%%%%%%%%%%

$\typemult{E}{7}{1}{1}$
%& $\type{E}{7}$
&$\Tindex{}{E}{7,7}{0}$
&$\sears{E}{7}{(1,1)}$
\\

\hline
$\typemult{E}{7}{1}{2}$
%& $\type{F}{4}$
&$\Tindex{}{E}{7,4}{9}$
&$\sears{F}{4}{(2,2)}$
\\

\whline
%%%%%%%%Absolute Type E8%%%%%%%%%%

$\typemult{E}{8}{1}{1}$
%& $\type{E}{8}$
&$\Tindex{}{E}{8,8}{0}$
&$\sears{E}{8}{(1,1)}$
\\

\whline
%%%%%%%%Absolute Type F4%%%%%%%%%%

$\typemult{F}{4}{1}{1}$
%& $\type{F}{4}$
&$\Tindex{}{F}{4,4}{0}$
&$\sears{F}{4}{(1,1)}$
\\

\whline
%%%%%%%%Absolute Type G2%%%%%%%%%%

$\typemult{G}{2}{1}{1}$
%& $\type{G}{2}$
&$\Tindex{}{G}{2,2}{0}$
&$\sears{G}{2}{(1,1)}$
\\

\whline

\end{tabular}
\medskip
\caption{The index and the SEARS of $\cL\in \bbM_2$}
\label{tab:TitsEARSu}
\end{table}

\begin{example}
\label{ex:TitsE7}
To give our first example, suppose that $\cL = \typemult{E}{7}{1}{2}$.  Then,
$\cL$ has absolute type $\type{E}{7}$ and relative type $\type{F}{4}$.
But using the list of indices in Table 2 of \cite{T} and the algorithm (see  \pref{pgraph:indexsimple}(c)) for calculating
the relative type from the index, we find that there is
only one possible index, namely $\Tindex{}{E}{7,4}{9}$,  yielding these relative and absolute types. (In fact this is true over any field $F$.)
So we conclude that $\cL$
has index  $\Tindex{}{E}{7,4}{9}$.
\end{example}

\begin{pgraph} The same simple method (which with some practice goes quite quickly) works for most of the
algebras in $\bbM_2$.  The exceptions are
the following algebras:
\begin{gather}
\label{eq:TitsAk}
\typemult{A}{\rkfin}{1}{\rot(q)},\ \rkfin \ge 1,\ \textstyle 0\le q\le \lfloor\frac{\rkfin+1}2\rfloor,\ \gcd(q,\rkfin+1)=1;\\
\label{eq:Titsrest}
\typemult{C}{\rkfin}{1}{2}, k \text{ odd}, \rkfin \ge 3;\quad
\typemult{D}{\rkfin}{1}{2b}, \rkfin \ge 4;\quad
\typemult{D}{4}{1}{3};\quad
\typemult{D}{\rkfin}{1}{4}, \rkfin \ge 4; \quad\text{and }
\typemult{E}{6}{1}{1}.
\end{gather}
To handle each of these  cases, we construct a matrix model for the algebra and then
use the  Descriptions in Table 2 of \cite{T} (or a direct calculation) to identify the index.
\end{pgraph}

\begin{example}
\label{ex:TitsAk}
Suppose that
$\cL = \typemult{A}{\rkfin}{1}{\rot(q)}$, with $\rkfin$ and $q$ as in \eqref{eq:TitsAk}.
Then,  by \eqref{eq:TypsAAreal}, we  can identify
$\cL = \spl_{1}(Q(\theta))$, where $\theta$ has order
$m = \rkfin +1$ in $\base^\times$. Then, we  identify the centroid of $\cL$ with $R_2$
as in Lemma \ref{lem:quantum}(e); and, since $K_2/R_2$ is flat, we have
$\tcL = \spl_{1}(Q(\theta)) \otimes_{R_2} K_2
\simeq  \spl_{1}(Q(\theta)\otimes_{R_2} K_2)$.
But, by Lemma \ref{lem:cyclicquantum}, $Q(\theta)\otimes_{R_2} K_2$ is a central division algebra of degree
$k+1$ over $K_2$.    Thus, by Table 2 of \cite{T} (see the first Description on p.~55 of  \cite{T}),
$\cL$ has index   $\Tindex{1}{A}{k,0}{(k+1)}$.  \end{example}

We are left with the algebras in \eqref{eq:Titsrest}.  For these we use the
easy part (the construction) in the coordinatization theorems  for centerless Lie tori  (see \pref{pgraph:coordinate}).

\begin{example}
\label{ex:TitsDk}
Suppose that
$\cL = \typemult{D}{\rkfin}{1}{4}$
with $\rkfin \ge 4$ and $\rkfin$ even.
Then,
$\cL$ has relative type $\type{BC}{\frac{k-1}2}$.
To describe a construction of $\cL$, let
$\cA = Q(-1)$,
the unital associative algebra presented by the generators $x_1,x_2,x_1^{-1},x_2^{-1}$ subject to the inverse relations
and the relation $x_1x_2 = - x_2x_1$. We  identify the centre of $\cA$ with $R_2$ as in \pref{pgraph:centquantum} with $m=2$.  Let
$*$ be the $R_2$-linear involution on $\cA$ such that $x_i^* = x_i$ for $i=1,2$.
Let $J_{k-2}$ be the matrix in $\Mat_{k-2}(\cA)$ whose $(i,j)$ entry
is $\delta_{i,k-1-j}$, and let $G = \diag(J_{k-2},1,x_1)$ (in block diagonal form) in $\cE = \Mat_\rkfin(\cA)$.
Note that $G$ is a unit
in the $R_2$-algebra $\cE$ and $(G^*)^t = G$, so we may define an $R_2$-linear involution
$\tau$ on $\cE$ by $\tau(T) = G^{-1}(T^*)^t G$.
We set
\[\cS = \set{T\in \cE \suchthat \tau(T) = -T},\]
in which case $\cS$ is a Lie algebra over $R_2$ under the commutator product.

We chose to look at $\cS$ since it arises in \cite[Theorem 6.3.1]{AB},
which is the coordinatization theorem for centreless $(\Lm,\Dl)$-tori  for $\Dl$
of type $\type{BC}{r}$,
where in our case $r =  \frac{k-1}2$.  Indeed the first part of that theorem
tells us that $\cS$ is a centreless Lie $(\Lm,\Delta)$-torus,  where
$\Lm$ has rank $2$ and $\Delta$ is of type $\type{BC}{\frac{k-1}2}$.
(See also \textbf{7.2.3} of \cite{AB}.)
Moreover, since $*$ is nontrivial, $\cS$ has full root support \cite[Remark 8.1.1]{AB}.
Also, it follows from  Propositions 4.10 and 6.21 of \cite{AB} that
the centroid of $\cS$ is identified with $R_2$ using the natural action of $R_2$
on $\cS$.  Since it is clear that $\cS$ is finitely generated as an $R_2$-module,
$\cS$ is fgc.  Hence, by Theorem \ref{thm:LT}, $\cS\in \bbM_2$
and the relative type of $\cS$ is $\type{BC}{\frac{k-1}2}$.

We next look at the central
closure of $\cS$.  For this we set $\widetilde \cA = \cA \otimes_{R_2} K_2$, $\widetilde \cE = \cE \otimes_{R_2} K_2$,
and $\widetilde \cS = \cS \otimes_{R_2} K_2$;
and we extend $*$ and $\tau$ to $K_2$-linear involutions of $\widetilde \cA$ and
$\widetilde \cE$ respectively.  One checks that
\[\widetilde \cS = \set{T\in \widetilde\cE \suchthat \tau(T) = -T}.\]
Further, we know from Lemma \ref{lem:cyclicquantum} that $\cA$ is a central division algebra of degree 2
over $K_2$; and its clear that $*$ is an orthogonal involution on $\cA$.
Thus, by \cite[Theorem X.9]{J}, the absolute type of $\cS$ is $\type{D}{k}$.
Since we have matched relative and absolute types, we see that $\cL \simeq \cS$.

Finally from the information in the previous paragraph and from
Table 2 of \cite{T} (see the Descriptions on p.~57 of \cite{T}),
it follows that $\cS$ has index $\Tindex{t}{D}{k,\frac{k-1}2}{(2)}$, where
$t = 1$ or $2$.  Moreover, $t= 1$ if and only if
the discriminant $\disc(\widetilde\cE,\tau)$ of the algebra with involution $(\widetilde\cE,\tau)$
over $K_2$ is trivial in $K_2^\times/ (K_2^\times)^2$. Finally one calculates using
\cite[7.2]{KMRT} that $\disc(\widetilde\cE,\tau) = -t_1^{k+1}t_2^k (K_2^\times)^2 \ne
1 (K_2^\times)^2$, since $k$ is even.  So $\cL \simeq \cS$ has index $\Tindex{2}{D}{k,\frac{k-1}2}{(2)}$.
\end{example}

\begin{pgraph}
The remaining cases in \eqref{eq:Titsrest} can be handled in a similar fashion.
For the algebras $\typemult{C}{\rkfin}{1}{2}$ ($k$ odd), $\typemult{D}{\rkfin}{1}{2b}$, and $\typemult{D}{\rkfin}{1}{4}$  ($k$ odd), one
again uses the construction from \cite[Theorem 6.3.1]{AB}.  (That theorem covers  quotient types $\type{BC}{r}$ and $\type{B}{r}$.)
Finally for the algebras $\typemult{D}{4}{1}{3}$ and $\typemult{E}{6}{1}{3}$, one uses the construction from
\cite[Thm.~5.63]{AG} which covers  quotient type
$\type{G}{2}$.\footnote{A different ``cohomological"
approach to computing the index will be described in \cite{GP2}.}
\end{pgraph}

\subsection{Saito's extended affine root systems.}\
\label{subsec:SEARS}

In  this section and the next, we assume for convenience that $\base = \bbC$, although it is very likely that with minor
modifications the results mentioned carry over to the general case.

\begin{pgraph}
Let $V_\bbR$ be a finite dimensional real vector space with positive definite real valued symmetric bilinear
form $\form$. For $\al\in V_\bbR$ with $(\al,\al)\ne 0$, let $r_\al$ be the orthogonal reflection along $\al$.
We will call a subset $\Sigma$ of $V_\bbR$ a \emph{Saito extended affine root system},
or \emph{SEARS} for short, if the following conditions hold:
the natural map $\spann_\bbZ(\Sigma) \otimes_\bbZ \bbR \to V_\bbR$ is bijective,
$(\al,\al)\ne 0$ for $\al\in \Sigma$, $r_\al(\Sigma) = \Sigma$ for $\al\in \Sigma$, $2\frac{(\al,\beta)}{(\al,\al)}\in \bbZ$
for $\al,\beta\in\Sigma$, and $\Sigma$ is irreducible in the usual sense.  In that case, we say that $\Sigma$
is \emph{reduced} if $\al\in \Sigma$ implies $2\al\notin \Sigma$, and
we define the \emph{null dimension} of $\Sigma$ to be the
the dimension of the radical of $V_\bbR$.  If $\Sigma$ is a SEARS, then the image of $\Sigma$
in $V_\bbR$ modulo the radical of $\form$
is the set of nonzero roots of an irreducible finite root system called the \emph{finite quotient
root system} of $\Sigma$ \cite[Example 1.3 and Assertion 1.8]{Sai}.  Two SEARS
$\Sigma$ in $V_\bbR$ and $\Sigma'$ in $V'_\bbR$
are said to be isomorphic if there is a vector space  isomorphism from
$V_\bbR$ onto $V'_\bbR$ which maps $\Sigma$ onto $\Sigma'$, in which
case the forms are preserved up to nonzero scalar by the isomorphism \cite[Lemma 1.4]{Sai}.
\end{pgraph}

\begin{pgraph} SEARS of null dimension
$n$ were introduced by Saito in \cite[\S 1.3]{Sai}, where they were called $n$-extended affine root systems.
His motivation came from the study of elliptic singularities of complex analytic surfaces.
\end{pgraph}

\begin{pgraph}
\label{pgraph:SEARSEALA}
SEARS also play an important role in the theory of EALAs, because the set of nonisotropic roots
$\Phi^\times$ of a discrete EALA $(\fg,\form,\fh)$ of nullity n is a reduced SEARS
of null dimension $n$ (by \cite[Thm.~2.16]{AABGP} and \cite[Lemma 1.4]{Az}).  Here $V_\bbR$ is the real span of $\Phi^\times$ with
symmetric form $\form$ induced and suitably normalized from the given form  on~$\cE$.
\end{pgraph}

\begin{pgraph}
\label{pgraph:SEARSclass} \headtt{Classification of SEARS}
In \cite[\S's 5.2 and 5.4]{Sai}, Saito classified up to isomorphism all reduced SEARS of null dimension
2 that satisfy an additional condition, the existence of a marking $G$
such that $\Sigma/G$ is reduced.
In \cite{Az}, Azam completed the classification of reduced SEARS of null dimension
2 by establishing a relationship between reduced SEARS
and the root systems studied in \cite[Chap.~2]{AABGP}.  He showed that, in addition
to the root systems in Saito's classification, there are two
infinite families that we will denote  here by $\sears{BC}{\ell}{(1,1)}$, $\ell \ge 1$, and $\sears{BC}{\ell}{(4,4)}$, $\ell\ge1$,
following the notational conventions in \cite{Sai}.  They are
\begin{equation}
\label{eq:moreSEARS}
\begin{gathered}
\sears{BC}{\ell}{(1,1)} =  \big(\Dl_\text{sh} +\Lm\big)
\cup \big(\Dl_\text{lg} +\Lm\big)
\cup \big(\Dl_\text{ex} +\Lm\setminus(2\Lm)\big) \quad \text{ and } \\
\sears{BC}{\ell}{(4,4)} =  \big(\Dl_\text{sh} +\Lm\setminus(2\Lm+a)\big)
\cup \big(\Dl_\text{lg} +2\Lm\big)
\cup \big(\Dl_\text{ex} +4\Lm+2a\big),
\end{gathered}
\end{equation}
where $\Dl_{\text{sh}}$, $\Dl_{\text{lg}}$ and $\Dl_{\text{ex}}$ are the sets of roots
of length $1$, $2$ and $4$ respectively in the irreducible finite root system of type $\type{BC}{\ell}$,
$a,b$ is a basis for the radical of $\form$, and
$\Lm = \bbZ a \oplus \bbZ b$. Note that the second term in each of these unions is empty
if~$\ell = 1$.
The complete nonredundant list of reduced SEARS of null dimension 2  up to isomorphism is then
\begin{equation}
\label{eq:SEARSlist}
\begin{gathered}
\sears{A}{\ell}{(1,1)}\, (\ell \ge 1),\quad \sears{A}{1}{(1,1)*},\\
\sears{B}{\ell}{(1,1)}\, (\ell \ge 3),\quad \sears{B}{\ell}{(1,2)}\, (\ell \ge 3),
    \quad \sears{B}{\ell}{(2,2)}\, (\ell \ge 2), \quad \sears{B}{\ell}{(2,2)*}\, (\ell \ge 2),\\
\sears{C}{\ell}{(1,1)}\, (\ell \ge 2),\quad \sears{C}{\ell}{(1,2)}\, (\ell \ge 2),
    \quad \sears{C}{\ell}{(2,2)}\, (\ell \ge 3), \quad \sears{C}{\ell}{(1,1)*}\, (\ell \ge 2),\\
\sears{BC}{\ell}{(2,1)}\, (\ell \ge 1),\quad \sears{BC}{\ell}{(2,4)}\, (\ell \ge 1),
    \quad \sears{BC}{\ell}{(2,2)}(1)\, (\ell \ge 2),\quad \sears{BC}{\ell}{(2,2)}(2)\, (\ell \ge 1),\\
\sears{BC}{\ell}{(1,1)}\, (\ell \ge 1),\quad \sears{BC}{\ell}{(4,4)}\, (\ell \ge 1),\\
\sears{D}{\ell}{(1,1)}\, (\ell \ge 4),\quad \sears{E}{\ell}{(1,1)} (\ell = 6,7,8),\\
\sears{F}{4}{(1,1)},\quad \sears{F}{4}{(1,2)},\quad  \sears{F}{4}{(2,2)},\quad \sears{G}{2}{(1,1)},\quad \sears{G}{2}{(1,3)}.
\end{gathered}
\end{equation}
(See \cite[\S's 5.2]{Sai} and \eqref{eq:moreSEARS} above for the definitions of these root systems.)
Note that the type of the finite quotient root system of a reduced SEARS
$\Sigma$ of null dimension 2 is used as the ``base'' in the notation
for $\Sigma$.  For example,
the reduced SEARS of null dimension 2 with finite quotient root system of type $\type{BC}{\ell}$ occur in Rows 4 and 5
of~\eqref{eq:SEARSlist}.
\end{pgraph}

\subsection{The SEARS of an isotropic algebra  in $\bbM_2$.}\
\label{subsec:SEARSM2}

Suppose that $\base = \bbC$ and that the primitive roots of unity
are chosen as $\zeta_m = e^{\frac{2\pi \sqrt{-1}}m}$ for $m\ge 1$.  Let $\cL$
be an isotropic algebra in $\bbM_2$.

\begin{pgraph}
\label{pgraph:SEARSM2}
By Theorem \ref{thm:class},  there exists a unique affine GCM $A$
and a unique $\sg \in \Aut(A)$ such that $(A,\sg)$ appears
in columns 1 and 2 of Table \ref{tab:reltypeu}
and $\cL \simeq \Lp(\fgb,\sg)$, where $\fg = \fg(A)$.
We note that the assumption that $\cL$ is anisotropic
precisely rules out the case in Row 1 of Table \ref{tab:reltypeu}
where $\gcd(q,k+1) =1$.
We set $\fg = \fg(A)$ and use the above isomorphism to identify
\begin{equation*}
\label{eq:Saitolink}
\cL = \Lp_m(\fgb,\sg),
\end{equation*}
where $m$ is the order of $\sg$.
Since $\sg$ is not transitive,  we can
follow \pref{pgraph:aff} and construct a discrete EALA
\[\cE = \Lp_m(\fg,\sg) \oplus \base \tilde c \oplus \base \tilde d,\]
with nondegenerate symmetric bilinear  form $\form$ and
ad-diagonalizable abelian subalgebra  $\cH = \fh^\sg \oplus \base \tilde c \oplus \base \tilde d$.
Then, by Proposition \ref{prop:affEALA}, $\cE$ is a discrete EALA of nullity 2 with centreless core isomorphic to $\cL$. Let $\Phi$ be the  root system for $\cE$ relative to  $\cH$. By
\pref{pgraph:SEARSEALA},  $\Phi^\times$ is a reduced SEARS of null dimension 2 which we call
the \emph{SEARS of $\cL$}.
\end{pgraph}

\begin{remark}
\label{rem:conjugacy}  The SEARS of $\cL$ just defined depends only on  the isomorphism class of
$\cL$ because of the uniqueness
of $A$ and $\sg$ in the above discussion.  Indeed, once these have been  selected,
$\cE$, $\cH$ and $\form$ are defined in terms of $A$ and $\sg$.  It would be desirable
to have a more intrinsic proof of the invariance of the SEARS of $\cL$.
\end{remark}

\begin{proposition}
\label{prop:SEARStype}
Let $\Phi^\times$ be the SEARS of $\cL$,
and let $X_\ell$ be the type of the finite quotient root system of $\Phi^\times$.
Then, $\cL$ has relative type $X_\ell$.
\end{proposition}

\begin{proof} Let $\cE$ be the EALA described in described in
\pref{pgraph:SEARSM2}.  It follows from Remark \ref{rem:finitetype} that
$X_\ell$ is the  quotient type of $\cE$.
But, by Proposition \ref{prop:affEALA}, we have $\cE_\ccore \simeq \cL$, so our conclusion follows from Corollary
\ref{cor:EALAmult}(a).
\end{proof}

\begin{pgraph}
We now use the work of U.~Pollmann  to list, in Column 3 of  Table~\ref{tab:TitsEARSu},
the SEARS of each isotropic $\cL$ in $\bbM_2$.

Indeed Pollmann did most of the work.
For each algebra $\Lp(\fgb,\sg)$ in Table  \ref{tab:reltypeu}, except those in Row 1,
she calculated the root system $\Phi$
of the EALA $\cE$ described in \pref{pgraph:SEARSM2}.
(She actually worked with a central quotient of $\cE$ rather than $\cE$,
but this does not change the root system.)  In view of
Lemma \ref{lem:EALAbasic}(b), this determines $\Phi^\times$ as well.

This leaves us with the case when $\cL =\typemult{A}{\rkfin}{1}{\rot(q)}$, where $\rkfin\ge 1$,
$0\le q \le \lfloor \frac {\rkfin+1}2 \rfloor$, and $r :=  \gcd(q,\rkfin+1)-1 \ge 1$.
Then, by Table \ref{tab:reltypeu},
$\cL$ has relative type $\type{A}{r}$.
Now, if $r\ge 2$, there is only one reduced SEARS of null dimension 2 with finite quotient root system $\type{A}{r}$,
so by Proposition \ref{prop:SEARStype},
$\Phi^\times \simeq \sears{A}{r}{(1,1)}$.  It remains to consider the case when $r=1$.
In this case, the root system $\Phi$ (and hence $\Phi^\times$) can be found by direct calculation
using the realization \eqref{eq:TypsAAreal} of $\cL$ in terms of a quantum  torus, and we obtain
$\Phi^\times \simeq \sears{A}{1}{(1,1)}$.  The details of this calculation are somewhat delicate but straightforward,
and we leave them to the reader.
\end{pgraph}

\begin{remark}
\label{rem:Saitolowrank}
In Row 7, Column 3 of Table~\ref{tab:TitsEARSu},
$\smash{\sears{C}{\frac k2}{(1,1)*}}$ should be interpreted as
$\sears{A}{1}{(1,1)*}$ when $k=2$.
\end{remark}

A  comparison of Column 3 of  Table~\ref{tab:TitsEARSu} with the list \eqref{eq:SEARSlist}
now yields our final result.

\begin{theorem}  Every reduced SEARS of null dimension 2  arises as the SEARS of some
isotropic algebra in $\bbM_2$. Moreover, if $\cL$ and $\cL'$ are algebras in $\bbM_2$
that do not satisfy Condition (AA) (and in particular are isotropic), then $\cL$ and $\cL'$ are isomorphic
if and only if their SEARS are  isomorphic.
\end{theorem}

\end{document}